\documentclass[11pt]{article}
\usepackage[margin=1.2in]{geometry}
\usepackage{amsmath,amsfonts,amssymb,amsthm}
\usepackage{graphicx} 
\usepackage{subcaption}
\usepackage{xspace}
\usepackage{xcolor}
\usepackage{authblk}
\usepackage{wrapfig}
\usepackage[most]{tcolorbox}
\numberwithin{equation}{section}
\usepackage{tikz}
\usetikzlibrary{calc,arrows.meta}

\usepackage{dsfont}

\usepackage[numbers]{natbib}
\usepackage[pagebackref=true,linktocpage=true]{hyperref}
\hypersetup{
	colorlinks=true,
	linkcolor=blue,
	citecolor=blue
}

\newcommand{\risk}{\mathbf{R}}
\newcommand{\prisk}{\mathbf{S}} %
\newcommand{\ind}[1]{\mathbf{1}\left\{#1\right\}}
\newcommand{\Xbar}{\overline{X}}
\newcommand{\sigmahat}{\widehat{\sigma}}
\renewcommand{\d}{\textnormal{d}\xspace}
\newcommand{\eps}{\epsilon}
\newcommand{\1}{\mathds{1}}
\newcommand{\E}{\mathbb{E}}
\newcommand{\NN}{\mathbb{N}}
\renewcommand{\Re}{\mathbb{R}}
\newcommand{\Var}{\mathbb{V}}
\renewcommand{\Bar}[1]{\overline{#1}}
\newcommand{\shat}{\hat{s}}
\newcommand{\dan}{\textsc{DAN}}

\newcommand{\taumax}{{\tau \lor m}}

\newcommand{\Pin}{{P \in \calP}}
\newcommand{\brackm}{{(m)}}

\newcommand{\iwr}{{\textsc{iwr}}\xspace}
\newcommand{\sn}{{\textsc{sn}}\xspace}
\newcommand{\reg}{{\textsc{reg}}\xspace}
\newcommand{\rws}{{\textsc{r-ws}}\xspace}
\newcommand{\mix}{{\textsc{mix}}\xspace}
\newcommand{\wald}{{\textsc{wald}}\xspace}

\newcommand{\infseqkm}[1]{(#1)_{k \geq m}}
\newcommand{\infseqm}[1]{(#1)_{m \in \NN}}

\newcommand{\aphci}{\textsc{aph-ci}\xspace}
\newcommand{\aphcs}{\textsc{aph-cs}\xspace}
\newcommand{\aphcis}{\textsc{aph-ci}s\xspace} %
\newcommand{\aphpval}{\textsc{aph-pval}\xspace}
\newcommand{\aphpvals}{\textsc{aph-pval}s\xspace}

\def\ddefloop#1{\ifx\ddefloop#1\else\ddef{#1}\expandafter\ddefloop\fi}
\def\ddef#1{\expandafter\def\csname cal#1\endcsname{\ensuremath{\mathcal{#1}}}}
\ddefloop ABCDEFGHIJKLMNOPQRSTUVWXYZ\ddefloop

\newtheorem{theorem}{Theorem}[section]
\newtheorem{proposition}[theorem]{Proposition}
\newtheorem{corollary}[theorem]{Corollary}
\newtheorem{lemma}[theorem]{Lemma}
\newtheorem{observation}[theorem]{Observation}
\theoremstyle{definition}
\newtheorem{definition}[theorem]{Definition}
\newtheorem{remark}[theorem]{Remark}

\title{Post-Hoc Large-Sample Statistical Inference}
\author[1]{Ben Chugg\thanks{Equal contribution. Correspondence: \texttt{benchugg@cmu.edu}}} 
\newcommand{\equalcontrib}{\footnotemark[\arabic{footnote}]}
\author[2]{Etienne Gauthier\protect\equalcontrib}
\author[2,3]{Michael I.\ Jordan}
\author[1]{\\Aaditya Ramdas}
\author[3]{Ian Waudby-Smith}
{ 
\affil[1]{\small Carnegie Mellon University}
\affil[2]{Inria \& \'Ecole Normale Sup\'erieure}
\affil[3]{University of California, Berkeley}
}
\date{March 2026 }

\begin{document}

\maketitle

\begin{abstract}
	We derive inferential procedures for large sample sizes that remain valid under data-dependent significance levels (so-called ``post-hoc valid inference''). 
    Classical statistical tools require 
    that the significance level---the ``type-I error''---is selected prior to seeing or analyzing any data. This restriction leads to some drawbacks. For instance, if an analyst generates an inconclusive 
    confidence interval, repeating the process with a larger 
    significance level is not an option---the result is final. 
    Recently, e-values have emerged as the solution to this problem, being both necessary and sufficient tools for performing various forms of post-hoc inference. 
    All such results, however, have thus far been nonasymptotic. As a result, they inherit familiar limitations of nonasymptotic inferential procedures such as requiring strong moment assumptions and being conservative in general. This paper develops a theory of post-hoc inference in the asymptotic setting, yielding asymptotic post-hoc confidence sets and asymptotic post-hoc p-values that make weaker assumptions and are sharper than their nonasymptotic counterparts. 
\end{abstract}

{
\small 
\setcounter{tocdepth}{2}
\tableofcontents
}

\section{Introduction}

The majority of applied statistical methods are asymptotic in the sense that their validity holds in the limit as number of samples tends to infinity. Asymptotic inference is widely deployed in practice because it typically requires weak moment assumptions. This is in contrast to nonasymptotic inference, which necessarily relies on knowledge of strong moment assumptions; see \citet{bahadur1956nonexistence} for a formal description of how nonasymptotic inference fails without those assumptions. 

Despite its wide applicability, asymptotic inference as it is used in practice still suffers the following drawback: 
the significance level---often denoted by $\alpha \in (0, 1)$---must be set before any data analysis takes place. Take confidence intervals (CIs) for instance. Once such an interval is computed, recomputing it on the same data with a different significance parameter rids it of the statistical guarantees it would have had under data-independent significance values. 

In this work we show how asymptotic e-values offer a solution---and as we will show, the only solution---to this problem. We provide a flexible framework for constructing asymptotic post-hoc confidence intervals, p-values, and hypothesis tests that allow $\alpha$ to be chosen after seeing the data. The post-hoc potential of e-values has been investigated in nonasymptotic hypothesis testing~\citep{grunwald2024beyond,koning2025post,koning2024continuous, chugg2025admissibility} as well as other areas of nonasymptotic inference; see, e.g.,~\citet{gauthier2025backward,gauthier2025values,gauthier2026adaptive} for applications in conformal prediction. In this work we extend this potential to the asymptotic setting. We are motivated by the wide applicability of asymptotic procedures, and seek to provide an alternative that is just as flexible while also providing post-hoc guarantees. 

To set the stage more formally, let us focus on confidence intervals. Suppose that data $X_1,\dots,X_n$ are drawn from some unknown distribution $P$ and there is a real-valued parameter $\theta \equiv \theta(P)$ that we wish to estimate. Given a significance level $\alpha\in(0,1)$, an asymptotic confidence interval for $\theta$ is a random set $\calC_n(\alpha)$ which covers $\theta$ with probability at most $\alpha$ in the limit. That is: 
\begin{equation}
\label{eq:intro-guarantee}
\text{For all }\alpha\in(0,1), \quad \limsup_{n\to\infty}P( \theta \notin \calC_n(\alpha)) \leq \alpha.           
\end{equation}
Note that $\alpha$ is assumed to be a constant, and it is for this reason that it must be chosen prior to doing any analysis. The following example illustrates how choosing $\alpha$ prior to analysis can lead to data being wasted on inconclusive results.  

\begin{tcolorbox}[
breakable, 
sharp corners, 
colback=pink!10,      
]
In epidemiology and immunology, confidence intervals are often built for the \emph{rate ratio} of a disease or a vaccine, the ratio of adverse cases to some comparator class within an appropriate time window~\citep{schouten1993risk,symons2002hazard}. The estimate of the rate ratio may suggest different actions; if too high for a vaccine then we might stop roll out, for example. 
Suppose a practitioner produces the confidence interval $\calC_n(0.01)$, designed to contain the true rate ratio with 99\% probability. 
If the interval is too wide for easy interpretation, it is tempting to compute $\calC_n(0.05)$, $\calC_n(0.1)$, and so on, in order to be able to provide some sort of public guidance. 
Unfortunately, most current approaches to constructing confidence intervals do not allow for such data-driven decisions and recalculating $\calC_n$ in this way will result in violations of the guarantee in \eqref{eq:intro-guarantee}. 
\end{tcolorbox}

The inability of classical confidence intervals to handle post-hoc significance levels is well-known. 
Selecting $\alpha$ as a function of the data has been dubbed the ``problem of roving alphas''~\citep{goodman1993p} and has been observed to be a problem in various disciplines, including medicine, psychology, business, and marketing (see \citealp{hubbard2003p,dar1994misuse,hubbard2003confusion,goodman1999toward}). 

Among analysts who are conscious of the problem of roving alphas, a common solution is ``$\alpha$-spending''~\citep{demets1994interim}, part of a broader class of what are often called \emph{group sequential} methods. Covid-19 monitoring was carried out using $\alpha$ spending, for instance~\citep{FDA_CBER_2021_COVID19_Vaccine_Safety_Master_Protocol}.  
The idea is to fix some global $\alpha$ and perform the initial analysis with some  $\alpha_1<\alpha$. If needed, one does a second analysis with $\alpha_2 < \alpha - \alpha_1$, then a third with $\alpha_3 < \alpha - (\alpha_1 + \alpha_2)$, and so on. (Here $\alpha_1,\alpha_2,\dots$ and indeed $\alpha$ must be predetermined.) Budgeting this way has clear drawbacks, however. Each analysis has significantly lower power  than it would if using the original $\alpha$, and we are still limited to only a handful preplanned of downstream analyses.  Note that $\alpha$-spending can be viewed as a form of Bonferroni correction on a collection of independent statistical analyses.

In this work we investigate and extend an alternative approach to $\alpha$ spending. The methodology to be presented allows for an arbitrary number of downstream analyses to be run---no partitioning of $\alpha$ is needed. The main difference between this approach and classical approaches is that the kind of guarantee we seek differs from~\eqref{eq:intro-guarantee}. Instead of bounding probability of error, we bound a particular \emph{risk} (to be defined) for many values of $\alpha$ simultaneously. This follows recent trends in post-hoc hypothesis testing~\citep{grunwald2024beyond,koning2025post,koning2024continuous,chugg2025admissibility} and conformal risk control~\citep{angelopoulos2024conformal}. 

There are both benefits and drawbacks to considering risk control instead of error probability. Our goal here is not to advocate for either side, but instead to develop the theoretical foundations of post-hoc asymptotic inference based on risk control as a complementary tool in the statistical toolbox. 

This paper will study asymptotics both in the pointwise setting and the distribution-uniform setting~\citep{li1989honest}, the latter requiring the asymptotics to be uniformly valid over a class of distributions. We will introduce distribution-uniformity more formally in Section~\ref{sec:dist-uniform}, but let us briefly make use of the guarantee in~\eqref{eq:intro-guarantee} 
to discuss the concept informally. In a distribution-pointwise statistical guarantee, no matter how large $n$ is, there may be \emph{some} distribution $P_n$ such that $P_n(\theta\notin \calC_n(\alpha))$ is still far away from $\alpha$ (and in some cases, arbitrarily close to 1). In a distribution-uniform guarantee, by contrast, it is required that asymptotic statements hold simultaneously for all $P$ in some class $\calP$ (mathematically, a supremum is taken over $\calP$ inside the the limit superior; see Definition~\ref{def:du-asymp-post-hoc-cs}). Such uniformity prohibits the aforementioned anomalous behavior.

\subsection{Contributions and outline}

After defining asymptotic post-hoc confidence intervals (\aphcis) and asymptotic post-hoc p-values (\aphpvals) in Sections~\ref{sec:asymp-post-hoc-cs} and~\ref{sec:dist-uniform},  our main contributions are as follows: 
\begin{enumerate}
    \item Section~\ref{sec:aphci-via-evalue} gives a general construction for deriving \aphcis and \aphpvals from asymptotic e-variables. 
    In fact, Proposition~\ref{prop:cs-via-evalue} shows that every \aphci which is monotonic and left-continuous (two natural desiderata that are often required for CIs) must be based on asymptotic e-variables. We also show that \aphpvals must be the inverses of asymptotic e-variables, thereby extending a fundamental result about nonasymptotic post-hoc inference \citep[Theorem 2]{koning2024continuous} to the asymptotic regime. 
    \item 
    Section~\ref{sec:iwr-evariable} begins with an analysis of the so-called \iwr asymptotic e-variable introduced by \citet{ignatiadis2024asymptotic}. 
    We weaken the conditions under which it holds in the pointwise setting, and prove that it remains an asymptotic e-variable under a  uniform third-moment assumption in the distribution-uniform setting; see Theorem~\ref{thm:iwr-asymp-evar}. 
    Using two distinct methods of choosing the free parameters in the \iwr e-variable---ex ante anchoring (Section~\ref{sec:ex-ante-tuning}) and the method of mixtures (Section~\ref{sec:aphcis-via-mom})---we give two \aphcis based on \iwr. These are stated as Theorems~\ref{thm:iwr-aphci} and \ref{thm:iwr-mixture-aphci}. 
    \item 
    Section~\ref{sec:rws-evariable} then provides a different method for constructing asymptotic e-variables and \aphcis. We introduce the so-called \rws asymptotic e-variable, which has the benefit of holding in the distribution-uniform setting under a $2+\delta$ moment assumption, for any $\delta>0$. While the resulting \aphci is somewhat looser than those based on the \iwr e-variable, we show that it is in fact providing a stronger guarantee. In particular, it is an example of a post-hoc asymptotic \emph{confidence sequence} (\aphcs), and the \rws e-variable is an asymptotic \emph{e-process}. These objects, introduced formally in Section~\ref{sec:asymp-eprocesses}, can be seen as extensions of asymptotic inference to both the post-hoc and time-uniform settings, thereby combining two recent strands of mathematical statistics.  
\end{enumerate}

Section~\ref{sec:simulations} contains experimental comparisons of the \aphcis, and Appendix~\ref{app:additional-results} provides additional theoretical results. In particular, Appendices~\ref{app:alternative-evars} and \ref{app:reg-aphci} contain additional asymptotic e-variables and \aphcis not central to our story here but still possibly of interest.

While the contributions of this work are largely theoretical, our goal is to improve the workflow of practicing statisticians. We hope to lay the foundations for a practical mode of inference which provides significantly more flexibility than traditional statistical technology.  With this in mind, the final section of the paper (Section~\ref{sec:summary}) lays out our recommendations for which of the post-hoc CIs that we introduce are best used in practice. Python implementations of all methods presented in this paper are available at \url{https://github.com/bchugg/asymp-posthoc-confint}. 

\subsection{Related work}

E-values, which are derived by analyzing expectations under the null instead of tail probabilities, provide an alternative approach to obtaining valid inference in repeated analyses;  see~\citet{ramdas2024hypothesis} and \citet{ramdas2023game} for broad introductions, and~\cite{shafer2021testing,grunwald2024safe,waudby2024estimating} for discussion papers that elaborate on connections to betting.  Our particular interest in e-values is their post-hoc validity, which as we will show allows the use of data-dependent error thresholds.

The link between e-values and data-dependent thresholds was forged by \citet{katsevich2020simultaneous} in the context of false discovery rate (FDR) control. These authors worked with  nonnegative supermartingales; note that nonnegative supermartingales yield e-values at stopping times, a fact that follows from Doob's optional stopping theorem. The formal connection between e-values and FDR control for data-dependent thresholds was presented by \citet[][Lemma 1]{wang2022false}. Setting the number of hypotheses to one specializes multiple testing to just testing a single hypothesis, and reduces FDR to the standard Type-I error. As such their Lemma 1 provides the first proof of the post-hoc error control delivered by thresholding e-values.

Analogous results to those of FDR control were presented in the setting of confidence intervals by \cite{xu2024post}.  This work focused on controlling the false coverage rate (FCR) using \emph{e-confidence intervals}.\footnote{Note that their results for the post-hoc validity of FCR also apply to the standard miscoverage rate if there is only one interval in question, and also that similar comments apply when FDR is when specialized to a single hypothesis, which yields classical Type-I error control. 
}
In a similar vein, \citet{grunwald2024beyond} studied post-hoc hypothesis testing with the goal of generalizing the Neyman-Pearson paradigm to allow post-hoc significance levels. This work establishes in particular that e-CIs are post-hoc valid (as we will formally describe below). 

Since then, a number of authors have noted that e-values catalyze a variety of kinds of data-adaptive decision-making. 
For example, \citet{xu2025bringing} and \citet{fischer2025knockoffs} employ e-values for post-selection inference and FDR control in multiple testing, allowing the nominal FDR level to be data-dependent. While these past works have emphasized that e-values yield ``post-hoc p-values,''  \citet{koning2025post,koning2024continuous}  proves the converse that every ``post-hoc p-value'' is the inverse of an e-value. 
\citet{grunwald2023posterior} develops the ``e-posterior'' which serves as a frequentist alternative to the Bayesian posterior using e-values while also allowing for post-hoc loss functions. 

\citet{gauthier2025values} consider post-hoc confidence intervals in the setting of conformal inference, developing adaptive coverage policies which strategically change the significance level to obtain conformal sets that are optimized under  constraints on their size~\citep{gauthier2026adaptive,gauthier2025backward}. \citet{chugg2025admissibility} build on the research program of \citet{grunwald2024beyond} and investigate admissibility in post-hoc hypothesis testing.

For a general overview of the post-hoc power of e-values, see \citet[Chapter 4]{ramdas2024hypothesis}. 

\subsection{Notation and background}
\label{sec:background}
Fix a measurable space $(\Omega, \calF)$.
Throughout, we let $\calP$ denote a set of distributions over real-valued random variables so that $(\Omega, \calF, P)$ is a probability space for each $\Pin$. For a random variable $X_1$ on $(\Omega, \calF)$ and some real $q \geq 1$, we let $\calM_q(\theta)$ be the set of distributions on $\Re$ with mean $\E_P[X_1] = \theta$ and a finite $q$-th moment (i.e., $\E_P|X_1-\theta|^q<\infty$). Let $\calM_q = \cup_{\theta \in\Re}\calM_q(\theta)$ be the set of all distributions with finite $q$-th moment. 

Recall that a sequence of random variables $(X_n)_{n \geq 1}$ on $(\Omega, \calF)$ are said to be uniformly integrable with respect to $\calP$ if
\begin{equation}
  \sup_\Pin\lim_{K\to\infty} \sup_{n \in \NN} \E_{P} [X_n \ind{X_n \geq K}]=0. 
\end{equation}
Note that this is a pointwise (in $\calP$) notion of uniform integrability. We will also make use of a distribution-uniform version. In particular, 
we call $(X_n)_{n \geq 1}$ $\calP$-uniformly integrable if 
\begin{equation}
    \lim_{K\to\infty} \sup_{n \in \NN} \sup_{P\in\calP}\E_{P} [X_n \ind{X_n \geq K}]=0. 
\end{equation}
Also important in the distribution-uniform setting will be the notion of uniform moment control. We say that a random variable $X$ has a $\calP$-uniformly bounded $q$-th moment if 
\begin{equation}
    \sup_{P\in\calP} \E_P|X|^q<\infty. 
\end{equation}
For each $n \in \NN$, let $\calF_n:=\sigma(X_1,\dots,X_n)$ be the $\sigma$-field generated by the first $n$ observations.
We will often abbreviate a sequence $(A_n)_{n\geq 1}$ as $(A_n)$, and shorthand the first $n$ elements $(A_1,\dots,A_n)$ as $A^n$. The index set of a sequence $(A_n)$ should always be assumed to be the natural numbers $\NN$.

Upon observing $X^n\sim P$ for some unknown $\Pin$, we are interested in estimating $\theta = \theta(P)$ for some real-valued functional $\theta:\calP \to \Re$ (e.g., the mean $\theta(P) = \int X\d P$). Define the parameter space $\Theta = \{ \theta(P): P\in\calP\}$. 

An e-variable for $\calP$ is a nonnegative random variable $E$ such that $\sup_{P\in\calP}\E_P[E]\leq 1$. An e-value is the realization of an e-variable, though in this work will use the two terms interchangeably (just as p-variables and p-values are typically used interchangeably). The majority of our results will rely on \emph{asymptotic} e-values: sequences of nonnegative random variables which are an e-value in the limit. These will be introduced more formally in Section~\ref{sec:asymp-post-hoc-cs}.

\section{Preliminaries}
\label{sec:prelims}

\subsection{Post-hoc confidence intervals and p-values}
\label{sec:post-hoc-CIs}

Before introducing asymptotic post-hoc inference, let us first describe nonasymptotic post-hoc inference procedures to help motivate the definitions of their asymptotic analogues to come. 

Given a set of distributions $\calP$ and a desired type-I error $\alpha \in (0, 1)$, an interval $\calC(\alpha)$ is said to be a $(1-\alpha)$-confidence interval for the functional $\theta$ if it satisfies 
\begin{equation}
\label{eq:type-I-error}
 \sup_{P\in\calP} P(\theta(P) \notin \calC(\alpha))\leq \alpha,
\end{equation}
whereas a p-value $Q$ satisfies 
\begin{equation}
\label{eq:pval} 
\text{for all } \alpha \in (0, 1),\quad \sup_{P\in\calP} P(Q \leq \alpha) \leq \alpha.
\end{equation}
Moving towards the definition of post-hoc confidence intervals and p-values, notice that the conditions in~\eqref{eq:type-I-error} and \eqref{eq:pval} are equivalent to writing
\begin{equation}
\label{eq:nonasymp-cs-as-risk}
    \sup_{\alpha\in(0,1)}\sup_{P\in\calP} \E_P\left[\frac{\ind{\theta(P) \notin \calC_n(\alpha)}}{\alpha}\right] \leq 1 \quad\text{and}\quad \sup_{\alpha\in(0,1)}\sup_{P\in\calP} \E_P\left[\frac{\ind{Q \leq \alpha}}{\alpha}\right] \leq 1,
\end{equation}
respectively.
Importantly, in \eqref{eq:nonasymp-cs-as-risk}, $\alpha$ is fixed before the expectation---it is data-independent. 
With \eqref{eq:nonasymp-cs-as-risk} in mind, the definitions of post-hoc confidence intervals and p-values with data-dependent $\alpha$ values can be intuited. These definitions have appeared in the literature; see \citet[Corollary 1]{xu2024post},  \citet[Section 3]{grunwald2024beyond} and \citet[Definition 2]{koning2025post}. 

\begin{definition}[Post-hoc confidence intervals and p-values]
    \label{def:nonasymp-posthoc-cs}
    We say that a family of random sets $\calH\equiv \{\calH(\alpha)\}_{\alpha>0}$ is a \emph{post-hoc confidence interval}  for $\theta$ and $\calP$ if 
    \begin{equation}
    \label{eq:nonasymp-post-hoc-cs}
    \sup_{P\in\calP} \risk(\calH;\theta,P)\leq 1 \text{~ where ~}
        \risk(\calH; \theta, P) :=  \E_{P} \left[\sup_{\alpha>0} \frac{\ind{\theta(P) \notin \calH(\alpha)}}{\alpha}\right].
    \end{equation}
    We say that a nonnegative random variable $Q$ is a \emph{post-hoc p-value} for $\calP$ if 
	\begin{equation}
    \label{eq:nonasymp-post-hoc-pval}
    \sup_{P\in \calP} \prisk(Q;P)\leq 1 \text{~ where ~}
		\prisk(Q;P):= \E_P \left[\sup_{\alpha>0} \frac{\ind{Q\leq \alpha}}{\alpha}\right]\leq 1. 
	\end{equation}
\end{definition}
In the above definition and the text preceding it, we have used the symbols ``$\calH$'' and ``$\calC$'' to refer to post-hoc and non-post-hoc confidence sets, respectively, and we will maintain this convention throughout the paper. We refer to both $\risk$ and $\prisk$ as \emph{risk} and refer to the guarantees~\eqref{eq:nonasymp-post-hoc-cs} and~\eqref{eq:nonasymp-post-hoc-pval} as \emph{risk control}. 
We will typically drop the ``for $\theta$ and $\calP$'' phrase from the definition unless there is ambiguity. If $\calH$ 
is a post-hoc confidence set then we also refer to it as \emph{post-hoc valid}, and similarly for post-hoc valid p-values. 
We often write $\risk(\calH)$ in place of $\risk(\calH; \theta, P)$ if $\theta$ and $P$ are understood from context. 

\begin{remark}[On the range of $\alpha$ extending beyond 1]
\label{rem:alpha-range}
    When moving from~\eqref{eq:nonasymp-cs-as-risk} to~\eqref{eq:nonasymp-post-hoc-cs} we not only moved $\alpha$ inside the expectation, but also allowed it to take values larger than one. This is only for mathematical convenience later on and does not change the guarantee in any meaningful way. Indeed, for any $\alpha>1$, $\ind{\theta(P)\notin\calH(\alpha)}/\alpha\leq 1$, hence such an $\alpha$ can never break risk control. 
\end{remark}

It is worth pondering the distinction between controlling the risk versus the error probability, noting in particular that the inequalities in \eqref{eq:nonasymp-post-hoc-cs} and \eqref{eq:nonasymp-post-hoc-pval} provide fundamentally different guarantees than those in \eqref{eq:type-I-error} and \eqref{eq:pval}.  
 Consider the case of post-hoc confidence intervals. 
If we design $k$ PH-CIs based on $k$ independent experiments, we should not expect that $\theta$ is contained in our sets at least $(1-\alpha)$\% of the time as $k\to\infty$. Instead, we should expect that for any $\alpha$, the average of $\ind{\theta(P)\not\in \calH_n(\alpha)}/\alpha$ over all trials should be at most one (as $k\to\infty$).

\begin{remark}[From post-hoc risk to type-I error for data-independent $\alpha$ values]
    If $\alpha$ is not data-dependent, then the post-hoc guarantees in~\eqref{eq:nonasymp-post-hoc-cs} and~\eqref{eq:nonasymp-post-hoc-pval} reduce to the usual type-I error guarantees in~\eqref{eq:type-I-error} and~\eqref{eq:pval}. This will be the case for all post-hoc guarantees studied throughout this paper. 
\end{remark}

Given that risk control is different than type-I error control in general, it is natural to wonder what utility the former provides the practitioner. After all, practitioners are familiar with error probabilities---can we expect them to switch to using risk instead? Our view is that such a switch is natural in settings which require post-hoc decision-making \emph{and} frequentist-style error guarantees. While Bayesian decision theory allows for post-hoc decision-making (this is well known, but see~\citet[Appendix C]{chugg2025admissibility} for a formal statement and proof), it notably does not provide bounds on the long-run error, while Definition~\ref{def:nonasymp-posthoc-cs} is designed to do just that. This, coupled with the fact that the theory of post-hoc risk control turns out to be quite rich, leads us to believe that fleshing out benefits and drawbacks of (asymptotic) post-hoc p-values and post-hoc CIs is worthwhile.

\subsection{Asymptotic post-hoc CIs  and p-values}
\label{sec:asymp-post-hoc-cs}

With the definition of post-hoc p-values and post-hoc CIs in hand, let us now introduce their asymptotic counterparts. With the notation introduced in Section~\ref{sec:background}, the guarantee stated in the introduction for an asymptotic confidence interval can be written more precisely as 
\begin{equation}
\label{eq:asymp-cs}
\text{For all }\alpha\in(0,1), \quad \sup_{P\in\calP}\limsup_{n\to\infty}P( \theta(P) \notin \calC_n(\alpha)) \leq \alpha.
\end{equation}
Following a similar argument to that of the previous section, we rewrite~\eqref{eq:asymp-cs} as a bound on the expected value, which then suggests the following definitions of asymptotic post-hoc confidence intervals and p-values.  

\begin{definition}[Asymptotic post-hoc confidence intervals and p-values]
    \label{def:asymp-post-hoc-cs}
    We say that a sequence $\calH\equiv (\calH_n)_{n\geq 1}$ is an \emph{asymptotic, post-hoc confidence interval} (\aphci)  for $\theta$ under $\calP$ if 
    \begin{equation}
    \label{eq:asymp-post-hoc-cs}
        \sup_{P\in\calP} \limsup_{n\to\infty} \risk(\calH_n;\theta,P) \leq 1. 
    \end{equation}
    We say that a sequence of random variables $(Q_n)_{n\geq 1}$ define an \emph{asymptotic post-hoc p-value} (\aphpval) for $\calP$ if 
	\begin{equation}
    \label{eq:asymp-posthoc-pval}
        \sup_{P\in\calP} \limsup_{n\to\infty} \prisk(Q_n;P)\leq 1. 
	\end{equation}
\end{definition}

In other words, just as an asymptotic confidence interval requires that type-I error is bounded in the limit, an \aphci requires that the post-hoc risk is bounded in the limit.

Note that the supremum over $\calP$ is outside the limit in both~\eqref{eq:asymp-post-hoc-cs} and~\eqref{eq:asymp-posthoc-pval}. This implies that the limit supremum is upper-bounded \emph{pointwise} for each distribution. 
One might instead hope for \emph{uniform} boundedness of the limit supremum, which would correspond to moving the supremum over $P \in \calP$ inside the limit. We introduce such ``distribution-uniform'' \aphcis and \aphpvals next.

\subsection{Distribution-uniform asymptotic post-hoc CIs and p-values}
\label{sec:dist-uniform}

Suppose we have developed an \aphci $\calH$ for estimating the mean $\mu(P)$ of a Gaussian distribution $P$. Let $P_{\sigma}$ be the distribution of a Gaussian random variable centered at zero with variance $\sigma^2 > 0$ and consider the family of distributions $\calP = \{P_{\sigma}: \sigma \in (0,\infty)\}$. Suppose for the sake of example that $\risk(\calH_n,P_\sigma) \asymp \sigma/n + 1$. Then for each fixed $\sigma$, $\limsup_n \risk(\calH_n,P_\sigma) =1$, so $\calH$ is indeed an \aphci. Nevertheless, for any $n$, $\sup_{\sigma > 0} \risk(\calH_n, P_\sigma) = \infty$, and hence
\[
\limsup_{n \to \infty} \sup_{P\in \calP} \risk(\calH_n,P)  =\infty.
\]
The above possibility motivates extending the definitions of \aphcis and \aphpvals to the \emph{distribution-uniform} setting, in which  
the limit holds simultaneously for all distributions in $\calP$. 
Distribution-uniform limit theorems have been studied since at least \citet{chung1951strong}, but the 
the study of distribution-uniform asymptotic confidence intervals in particular was initiated by \citet{li1989honest}, who called them ``honest'' confidence intervals. 
Since then, distribution-uniform methods have been studied in a wide array of settings, with a notable surge in momentum in the last decade~\citep{tibshirani2018uniform,rinaldo2019bootstrapping,shah2020hardness,waudby2023distribution,waudby2024distribution,waudby2025nonasymptotic}. 

Mathematically, extending the definition of \aphcis to the distribution-uniform setting is straightforward; we simply swap the limit and the supremum in Definition~\ref{def:asymp-post-hoc-cs}. Recall the definitions of $\risk$ and $\prisk$ from Section~\ref{sec:post-hoc-CIs}. We note that distribution-uniform p-values (though not post-hoc p-values) were studied recently by \citet{ignatiadis2024asymptotic}.

\begin{definition}[Distribution-uniform \aphcis and \aphpvals]
    \label{def:du-asymp-post-hoc-cs}
    Fix a family of distributions $\calP$. For each $\Pin$, let $\theta(P) \in \Re$ be a parameter to be thought of as the target estimand of the \aphci under $P$. We say that a sequence $\calH\equiv (\calH_n)_{n\geq 1}$ is a \emph{distribution-uniform \aphci} for $\theta \equiv (\theta(P))_\Pin$ under $\calP$ (or a \emph{$\calP$-uniform \aphci for $\theta$}) if 
    \begin{equation}
    \label{eq:du-asymp-post-hoc-cs}
         \limsup_{n\to\infty} \sup_{P\in\calP}  \risk(\calH_n;\theta(P),P) \leq 1.
    \end{equation}
    We say that a sequence of random variables $(Q_n)_{n\geq 1}$ is a \emph{distribution-uniform} \aphpval \emph{for} $\calP$ (or, equivalently, a $\calP$-uniform \aphpval) if 
	\begin{equation}
    \label{eq:du-asymp-posthoc-pval}	
        \limsup_{n\to\infty} \sup_{P\in\calP} \prisk(Q_n;P)\leq 1. 
	\end{equation}
\end{definition}

We will often shorten ``distribution-uniform'' to simply ``uniform.'' The distribution $\calP$ will usually be understood from context. If we do not say explicitly that an \aphci or \aphpval is uniform, it should be assumed to hold only pointwise. When we discuss $\calP$-uniform \aphcis going forward, we may simply refer to $\theta$ as ``the estimand'', but this should be thought of as an implicit collection of $\calP$-indexed estimands $(\theta_P)_{P \in \calP}$.

Next we study how to construct uniform \aphcis using uniform asymptotic e-variables.

\subsection{The sufficiency and necessity of asymptotic e-values}
\label{sec:aphci-via-evalue}

Let us now return to the question of how to construct \aphcis. Just as post-hoc confidence intervals can be designed with e-values---as shown  by \citet{xu2024post} and \citet{grunwald2024beyond}---we may design asymptotic post-hoc confidence intervals using \emph{asymptotic e-values}. 
Following Definition 3.5 and Proposition 3.6 in \citet{ignatiadis2024asymptotic}, an asymptotic e-value for $\calP$ is a sequence of nonnegative random variables $(E_n)_{n\geq 1}$ whose components satisfy 
\begin{equation}
\label{eq:asymp-evar}
 \sup_{P\in\calP}\limsup_{n\to\infty} \E_{P}[E_n]\leq 1.   
\end{equation}
(The authors referred to these as \emph{strongly} asymptotic e-variables; we omit the term strongly for brevity.) If the inequality is an equality then we call $(E_n)$ a \emph{sharp} asymptotic e-variable. 

This definition of an asymptotic e-variable is again a pointwise notion vis-a-vis $\calP$. To  construct distribution-uniform \aphcis and \aphpvals, we will need to rely on \emph{distribution-uniform} asymptotic e-variables, defined as a sequence of nonnegative random variables $(E_n)$ such that 
\begin{equation}
\label{eq:du-asymp-evar}
 \limsup_{n\to\infty} \sup_{P\in\calP}\E_{P}[E_n]\leq 1.   
\end{equation}
We will typically call such a sequence a $\calP$-uniform asymptotic e-variable to make the set of distributions $\calP$ explicit.  
These objects were also introduced by \citet{ignatiadis2024asymptotic}, who called them strongly uniform asymptotic e-variables. Again, we drop the ``strongly.''

If either~\eqref{eq:asymp-evar} or \eqref{eq:du-asymp-evar} hold with equality, we call the (uniform) asymptotic e-variables \emph{sharp}. 
If $\calP_\theta$ is the set of all distributions in some class with parameter $\theta$, we will say that $(E_n)$ is an asymptotic e-variable for $\theta$ if it is an asymptotic e-variable for $\calP_\theta$. Similarly, we call it a uniformly asymptotic e-variable for $\theta$ if it is a $\cal_\theta$-uniform asymptotic e-variable.

The following result shows that we can obtain an \aphci and an \aphpval by thresholding the components of an asymptotic e-value. In fact, this is the only way to construct (asymptotic) post-hoc confidence sets, assuming standard behavior of the sets. The proof is in Appendix~\ref{proof:cs-via-evalue} and is inspired by similar results for nonasymptotic e-values and post-hoc inference; see \citet[Section 3]{grunwald2024beyond} and \citet[Theorem 2]{koning2025post}.

\begin{proposition}
\label{prop:cs-via-evalue}
Let $\Theta$ be a set that can be thought of as the ``parameter space'' and 
let $(\mathcal{H}_n(\alpha))_{\alpha>0}$ be a family of subsets of $\Theta$ that is monotonic and right-continuous in $\alpha$, meaning:
\begin{itemize}
    \item $\calH(\alpha_2)\subseteq\calH(\alpha_1)$ for all $0<\alpha_1<\alpha_2$ (monotonicity),
    \item For every $\theta \in \Theta$, the rejection set $\{\alpha>0 : \theta \not\in\calH_n(\alpha)\}$ is a left-closed interval of the form $[\alpha_0,\infty)$ for some $\alpha_0>0$ (right-continuity).
\end{itemize}
Then the sequence $(\mathcal{H}_n)_{n \ge 1}$ is a (uniform) \aphci for $\theta^\star$ if and only if there exists a sequence of random variables $(E_n(\theta))$ for each $\theta\in\Theta$ 
such that for all $\alpha>0$, $H_n(\alpha)$ takes the form
\[
 \mathcal{H}_n(\alpha) = \left\{ \theta : E_n(\theta) < 1/\alpha \right\},
\]
and so that $(E_n(\theta^\star))$ is a (uniform) asymptotic e-variable.  
Furthermore, a sequence of random variables $(Q_n)_{n \ge 1}$ is a (uniform) \aphpval if and only if there exists some (uniform) asymptotic e-variable $(E_n)_{n \geq 1}$ so that $Q_n = 1/E_n$ for each $n \in \NN$.
\end{proposition}

Monotonicity and right-continuity are natural (and often assumed) desiderata for confidence intervals. The former implies that smaller values of $\alpha$ yield smaller intervals, while the latter implies nesting. That is, if $\theta$ is rejected for a sequence of levels approaching $\alpha$ from above, it is also rejected at $\alpha$. 

The ``if and only if'' in Proposition~\ref{prop:cs-via-evalue} holds only if $\alpha$ is allowed to take values larger than one in Definition~\ref{def:nonasymp-posthoc-cs}. Indeed, if $\alpha$ is restricted to $(0,1)$ then $\{\theta: E_n(\theta)<1/\alpha\}$ is an \aphci for $\theta^\star$ if there exists some family of processes $\{(E_n(\theta)):\theta\in\Theta\}$ such that $(E_n(\theta^\star))$ is an asymptotic e-variable for $\theta^\star$, but the converse does not necessarily hold.

\begin{remark}
\label{rem:alpha-dependence}
Proposition~\ref{prop:cs-via-evalue} implicitly assumes that the asymptotic e-value is not a function of $\alpha$. E-values with such a dependence are not rare. Indeed, the ``method of predictable plug-ins'' is a common way of constructing sequential e-values (i.e., e-processes), and typically relies on choosing parameters as functions of $\alpha$~\citep{waudby2024estimating}.  
E-values that integrate out parameters in $(0,\alpha)$  have also been proposed~\citep{fischer2025sequential}. Even though the e-value itself may not depend on $\alpha$, the confidence interval $\calH_n(\alpha)$ obtained from the e-value will. 
\end{remark}

In practice, constructing an \aphci using only a single e-variable often results in one-sided intervals (unless we use the method of mixtures; see Section~\ref{sec:aphcis-via-mom}).  To obtain two-sided intervals we will take the intersection of the sets defined by multiple e-values. 
We can, however, relax the requirement that these sets be constructed by individually valid (asymptotic) e-variables and allow them instead to be (asymptotic) \emph{compound e-variables}~\citep{ignatiadis2024asymptotic}. This result is not central to our story, however, and is thereby deferred to Appendix~\ref{app:aphcis-compound-evals}. 

Proving that sequences of random variables are pointwise asymptotic e-variables typically involves arguments concerning convergence in distribution combined with uniform integrability. By a result of \citet{billingsley1995proba}, this implies convergence of means (not \emph{in} mean). In the distribution-uniform setting, we will develop and rely on a proposition found in Appendix~\ref{app:du-convergence-in-exp} which can be viewed as an extension of Billingsley's result. We mention it here because it may of be independent interest.

\section{Constructing Asymptotic Post-Hoc Confidence Intervals}
\label{sec:constructing-aphcis}

By Proposition~\ref{prop:cs-via-evalue}, 
constructing \aphcis relies on finding asymptotic e-variables. These are less well studied than their nonasymptotic counterparts. Indeed, it is not obvious whether non-trivial asymptotic e-values ought to exist. \citet[Theorem 6.5]{ignatiadis2024asymptotic} lay this fear to rest and show that, for iid $X_1,\dots,X_n$ with mean $\theta$ and finite variance, for any $\lambda \in \Re$
\begin{equation}
\label{eq:e-iwr} 
    E_n^\iwr(\theta;\lambda) := \exp\left(\lambda \frac{S_n(\theta)}{V_n(\theta)} - \frac{\lambda^2}{2}\right)
\end{equation}
is an asymptotic $e$-value
where
\begin{equation}
\label{eq:Sn-Vn}
  S_n(\theta) := \sum_{i\leq n}(X_i - \theta)\text{~ and ~} V_n(\theta) := \bigg(\sum_{i\leq n} (X_i - \theta)^2\bigg)^{1/2}.  
\end{equation}
The superscript \iwr refers to the names Ignatiadis, Wang, and Ramdas~\citep{ignatiadis2024asymptotic}. 
In Section~\ref{sec:iwr-evariable} we will extend this asymptotic e-variable in several ways. First, we observe that it holds not only for all distributions with finite variance, but for those distributions in the domain of attraction of a Gaussian (a weaker condition). Second, we will prove that $(E_n^\iwr)$ is a uniform asymptotic e-variable under a uniformly bounded third-moment assumption. 
Theorem~\ref{thm:iwr-aphci} then presents the 
\aphci defined by $(E_n^\iwr(\theta;\lambda))$. 

Sections~\ref{sec:ex-ante-tuning} and \ref{sec:aphcis-via-mom} explore different ways to choose $\lambda$ to optimize the width of the \aphci. This is more involved than the nonasymptotic case, both because $\lambda$ cannot depend on $\alpha$ (Remark~\ref{rem:alpha-dependence}) and because of the requirement of uniform integrability. While these two sections focus on $E_n^\iwr$ in particular, the principles discussed therein apply more broadly---for instance, to the two additional \aphcis we present in Appendix~\ref{app:alternative-evars}.  

Section~\ref{sec:rws-evariable} provides another asymptotic e-variable using entirely different arguments than are used for the \iwr e-variable. It is based on a truncation technique coupled with a recent nonasymptotic SLLN given by \citet{ruf2025concentration}. The resulting \aphci has pros and cons compared to the \iwr \aphcis; we will discuss these tradeoffs in more depth in the discussions that follow.

\subsection{The IWR asymptotic e-variable and APH-CI}
\label{sec:iwr-evariable}

Throughout this section, let $\sigmahat_n^2 := (n-1)^{-1} \sum_{i\leq n}(X_i - \Xbar_n)^2$ denote the unbiased sample variance and let $\shat_n^2 = n^{-1}\sum_{i\leq n} (X_i - \Xbar_n)^2$ denote the biased sample variance where $\Xbar_n = n^{-1}\sum_{i\leq n}X_i$ is the sample mean.

To state our results for $E_n^\iwr$ we need to introduce several distributional assumptions. 
Recall that for $q \geq 1$, the set $\calM_q(\theta)$ denotes those distributions on $\Re$ with a finite $q$-th moment and mean $\theta$, and $\calM_q$ is the set of all distributions with finite $q$-th moment. 

We say that $(X_n)$ are in the domain of attraction of a Gaussian if there exist sequences $(a_n)$ and $(b_n)$ such that 
\begin{equation}
\label{eq:dan}
    \frac{\sum_{i\leq n}X_i - b_n}{a_n} \xrightarrow{d} N(0,1).  
\end{equation}
If the distribution has finite variance $\sigma^2$ then \eqref{eq:dan} holds with $b_n = n\E[X_1]$ and $a_n = \sigma\sqrt{n}$. But~\eqref{eq:dan} can also be satisfied for distributions without a finite second moment, such a Pareto distributions with shape 2.   See \citet{mikosch1999regular} for more details on domains of attraction. We let $\dan(\theta)$ be all distributions with mean $\theta$ that are in the domain of a attraction of a Gaussian.

Our distribution-uniform result for the \iwr e-variable will require uniformly bounded moments. Specifically, we will make use of a  $\calP$-bounded ``skew,'' meaning that
\begin{equation}
\label{eq:uniform-moment}
    \sup_{P\in \calP} \frac{\E_P|X_1 - \E_P X_1|^3}{\sigma_P^3}<\infty,
\end{equation}
where $\sigma_P^2$ is the variance of $X_1$ under $P$. 
Requiring uniformly bounded moments (typically more than two) is a standard assumption in distribution-uniform inference~\citep{shah2020hardness,waudby2023distribution} and is known to be necessary for uniform strong convergence results \citep{waudby2024distribution}. In our case we will require a uniformly bounded third moment/skew due to our use of the Berry-Esseen bound~\citep{bentkus1996berry}. To elaborate, the proof relies on upper bounding $S_n(\theta)/V_n(\theta)$ deterministically by $\sqrt{n}$ (using Cauchy-Schwarz), which is canceled out by the $1/\sqrt{n}$ rate of decay of the Berry-Esseen bound. This allows us to conclude that $\sup_P\E_P|S_n(\theta)/V_n(\theta)|<\infty$, which is an integral part of the proof. 

With these preliminaries in hand, we can state our result for the \iwr e-variable. In the pointwise case, the proof follows fairly easily by appealing to results of \citet{gine1997student}, which show that $S_n(\theta)/V_n(\theta)\xrightarrow{d} N(0,1)$, and that $\exp(\lambda_n S_n(\theta)/V_n(\theta))$ is uniformly integrable. 
The distribution-uniform case is significantly more involved. The details of the proof can be found in Appendix~\ref{proof:iwr-asymp-evar}.

\begin{theorem}
\label{thm:iwr-asymp-evar}    
Let $(X_n)$ be a sequence of iid random variables.
Let $(\lambda_n)$ be a sequence of nonnegative random variables which converge to some $\lambda\in\Re$ almost surely.  If $\sup_n |\lambda_n|<\infty$ almost surely and $(X_n)$ are drawn from some distribution in $\dan(\theta)$, then
    the process $(E_n^\iwr(\theta,\lambda_n))$ is a sharp asymptotic e-variable. 
    Additionally, if $(\lambda_n)$ is a deterministic sequence which converges to $\lambda$ and 
    $X_1,\dots,X_n$ have $\calP$-uniformly bounded skew for some $\calP\subseteq\dan(\theta)$, then $(E_n^\iwr(\theta;\lambda_n))$ is a $\calP$-uniform asymptotic e-variable. 
\end{theorem}

We have stated Theorem~\ref{thm:iwr-asymp-evar} as generally as possible---asymptotic e-variables are relatively new objects in the literature and new results might thus be interesting in their own right. Going forward, however, we will not require all these degrees of freedom. In particular, the \aphci based on $E_n^\iwr$ will not choose $\lambda$ as a function of $n$.

Let us now give the \aphci that results from inverting the \iwr asymptotic e-variable as implied by Proposition~\ref{prop:cs-via-evalue}. The proof can be found in Appendix~\ref{proof:iwr-aphci}.

\begin{theorem}
\label{thm:iwr-aphci}
Let $(X_n)$ be a sequence of iid random variables drawn from some distribution in $\dan(\mu)$.  
For any $\lambda>0$ and $\alpha\in(0,1)$, set 
\begin{equation*}
    A_{n,\alpha}(\lambda) = \frac{\log(2/\alpha) + \lambda^2/2}{\lambda\sqrt{n}}.  
\end{equation*}
Then for any $\lambda>0$ independent of $\alpha$,  $(\calH_n^\iwr(\alpha;\lambda))$ is an \aphci for $\mu$, where  
\begin{equation}
\label{eq:iwr-aphci}
  \calH_n^\iwr(\alpha;\lambda) :=  (\Xbar_n \pm W_n^\iwr(\alpha;\lambda)), \quad W_n^\iwr(\alpha;\lambda):=  \frac{  A_{n,\alpha}(\lambda)\shat_n}{(1 - A_{n,\alpha}(\lambda)^2)^{1/2}},
\end{equation}
if $n > (\log(2/\alpha)/\lambda + \lambda/2)^2$, otherwise $\calH_n^\iwr(\lambda) = \Re$.  
Additionally, if $(X_n)$ have $\calP$-uniformly bounded skew for some $\calP$, then $(\calH_n^\iwr(\alpha;\lambda))$ is a $\calP$-uniform \aphci. 
\end{theorem}

How do we choose the parameter $\lambda$ in $\calH_n^\iwr$? The width will approach zero for any fixed $\lambda$, since $A_{\alpha,n}(\lambda)\xrightarrow{n\to\infty}0$. For large $n$, $W_n^\iwr(\lambda) \approx A_{n,\alpha}(\lambda)\shat_n$, so the ideal choice of $\lambda$ is $\lambda_\alpha = \sqrt{2\log(2/\alpha)}$ which minimizes $\lambda\mapsto A_{n,\alpha}(\lambda)$. 
Unfortunately, as emphasized both in the theorem statement and in Remark~\ref{rem:alpha-dependence}, such an $\alpha$-dependent choice of $\lambda$ is illegal.

Sections~\ref{sec:ex-ante-tuning} and~\ref{sec:aphcis-via-mom} pursue two distinct strategies for choosing $\lambda$. 
First, however, let us better understand the role of a fixed $\lambda$ by studying how the limiting width of $\calH_n^\iwr$ scales with $\sqrt{n}$. Suppose that $(X_n)$ are iid with finite variance $\sigma^2$. Then a single application of the SLLN combined with the continuous mapping theorem shows that 
\begin{equation}
\label{eq:iwr-width}
    \sqrt{n}W_n^\iwr(\alpha;\lambda) \xrightarrow[n\to\infty]{a.s.} \sigma g(\lambda,\alpha) \text{~ where ~} g(\lambda,\alpha) := \frac{\log(2/\alpha)}{\lambda} + \frac{\lambda}{2}.
\end{equation}

It is fruitful to compare this width to the usual Wald CI for finite (unknown) variance based on the CLT (which we emphasize is not post-hoc valid): 
\begin{equation}
\label{eq:wald-CI}
    \calH_n^\wald(\alpha):= (\Xbar \pm W_n^\wald(\alpha)), \text{~ where ~} W_n^\wald(\alpha) := z_{1-\alpha/2}\frac{\sigmahat_n}{\sqrt{n}}.
\end{equation}
By the SLLN, $\sqrt{n} W_n^\wald(\alpha) \to z_{1-\alpha/2}\sigma$, where $z_{1-\alpha/2} = \Phi^{-1}(1 - \alpha/2)$ is the standard normal quantile. For small $\alpha$,  $\Phi^{-1}(1-\alpha/2)$ behaves as $\sqrt{2\log(2/\alpha})$ (indeed, $\lim_{\alpha\to 0^+} \Phi^{-1}(1 - \alpha/2)/\sqrt{2\log(2/\alpha)}=1$).  Thus, $\sqrt{n} W_n^\wald(\alpha) \approx \sqrt{2\sigma^2\log(2/\alpha})$ for large $n$ and small $\alpha$. 

As discussed above, the ideal choice of $\lambda$ is $\lambda_\alpha := \sqrt{2\log(2/\alpha})$ so that $g(\lambda_\alpha, \alpha) = \log(2/\alpha)/\lambda_\alpha + \lambda_\alpha/2 = \sqrt{2\log(2/\alpha)}$. In this case $\sqrt{n}W_n^\iwr(\alpha;\lambda_\alpha) \to \sqrt{2\sigma^2\log(2/\alpha)}$ almost surely, the same as the Wald CI, approximately the same as the Wald CI for small $\alpha$. 

If one uses $(\calH_n^\iwr)$ simply as regular asymptotic CIs and does not change $\alpha$, 
then one is of course free to choose $\lambda = \lambda_\alpha$. However, in a post-hoc setting (i.e., one in which $\alpha$ may change), this is not allowed. The next two sections investigate two distinct data-independent methods to choose $\lambda$ in Theorem~\ref{thm:iwr-aphci}.  The first simply takes $\lambda=\lambda_{\alpha_0}$ where $\alpha_0$ is a data-independent ``guess'' for $\alpha$. The second creates a new asymptotic e-variable by integrating out the parameter $\lambda$ in $E_n^{\iwr}$.

\subsection{Choosing \texorpdfstring{$\lambda$}{parameters}, Option I: Ex ante anchoring}
\label{sec:ex-ante-tuning}

\begin{wrapfigure}{r}{0.5\textwidth} 
    \centering
    \includegraphics[width=\linewidth]{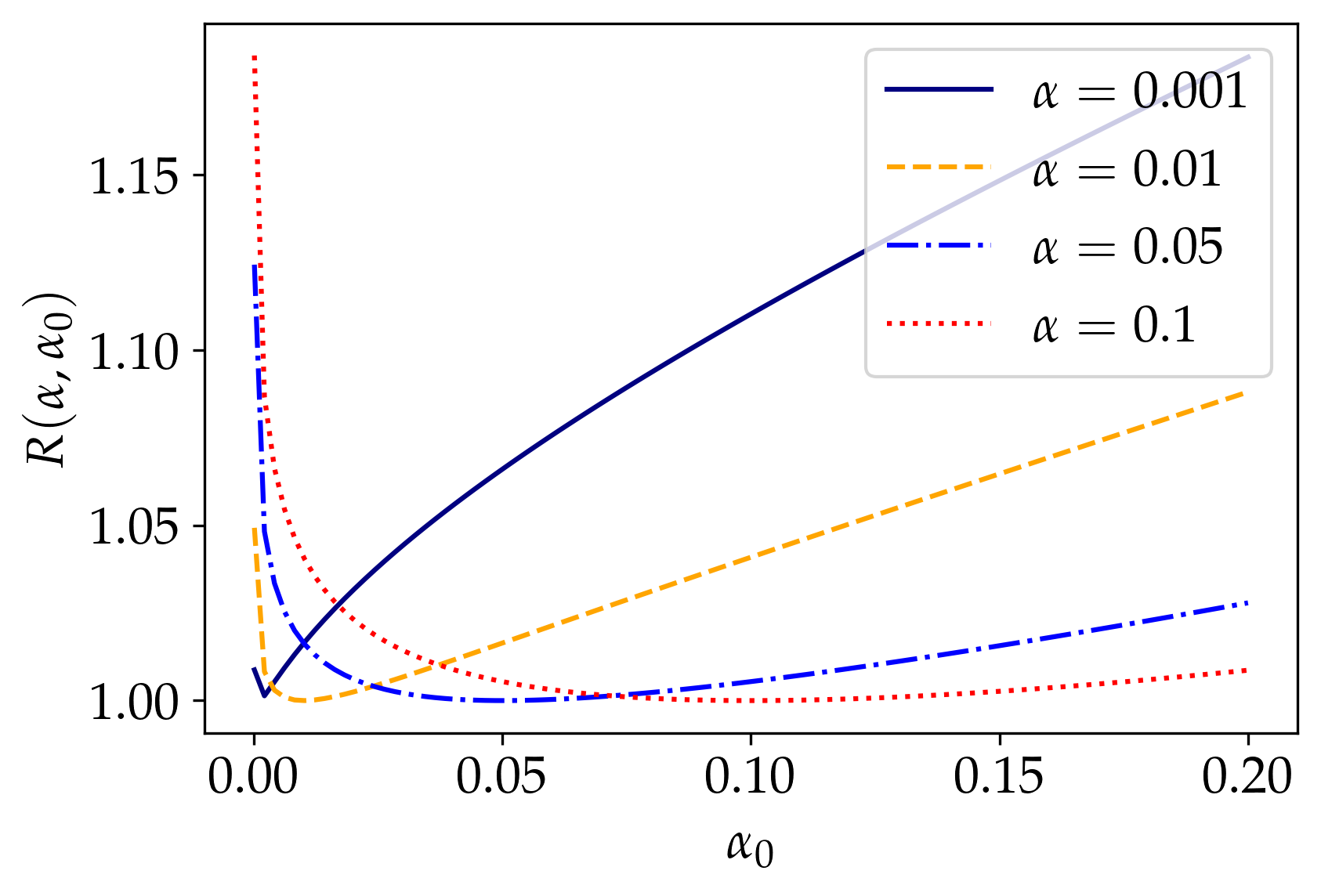}
    \caption{Empirical evaluation of the ratio in~\eqref{eq:g_ratio}. Here $\alpha_0$ ranges from 0.001 to 0.2 and we plot the curves $\alpha_0\mapsto R(\alpha,\alpha_0)$ for select values of $\alpha$. The maximum value $R(\alpha,\alpha_0)$ across all values of $\alpha$ and $\alpha_0$ is 1.184.  }
    \label{fig:anchor_ratio}
\end{wrapfigure}

Here we consider simply fixing some nominal level $\alpha_0$ in advance of seeing the data and choosing $\lambda = \lambda_{\alpha_0}$. This can be viewed as our best guess of what $\alpha$ will be (or perhaps our \emph{hope} of what it will be) and take $\lambda_0 := \sqrt{2\log(2/\alpha_0)}$. 
In this case $g(\lambda_0,\alpha) = \log(2/\alpha)/\sqrt{2\log(2/\alpha_0)} + \sqrt{\log(2/\alpha_0)/2}$, which will be close to $g(\lambda_\alpha,\alpha)$ for $\alpha_0\approx \alpha$.

This strategy, henceforth referred to as \emph{ex ante anchoring}, has a similar flavor to the idea of ``luckiness'' in Bayesian theory (PAC-Bayesian theory in particular)~\citep{shawe2002structural}, in which bounds are tighter the closer the posterior is to the prior, but valid for any prior; see also \citep[Figure 3]{waudby2020confidence} for an illustration of this phenomenon with Bayesian-inspired e-values. We refer to  $\alpha_0$ as the \emph{anchor}.

As we will see in the simulations of Section~\ref{sec:simulations}, ex ante anchoring performs surprisingly well in practice, even for $\alpha\gg \alpha_0$. To get a sense of why, note that the choice of $\alpha$ only affects the ratio of the asymptotic widths of the \aphcis by the square root of a logarithmic factor. Indeed, the ratio of $g(\lambda_0,\alpha)$ to $g(\lambda_\alpha,\alpha)$ (defined in~\eqref{eq:iwr-width}) can be written as 
\begin{equation}
\label{eq:g_ratio}
    R(\alpha,\alpha_0) := \frac{g(\lambda_0,\alpha)}{g(\lambda_\alpha,\alpha)} = \frac{1}{2}\left(\sqrt{\frac{\log(2/\alpha)}{\log(2/\alpha_0)}} + \sqrt{\frac{\log(2/\alpha_0)}{\log(2/\alpha)}}\right).
\end{equation}

For reasonable values of $\alpha$ and $\alpha_0$, this ratio is quite small. Figure~\ref{fig:anchor_ratio} plots the curves $x\mapsto R(x,\alpha_0)$ for select values of $\alpha_0$. For $x$ between 0.0001 and 0.2, the ratio never exceeds 1.19. Thus, by \eqref{eq:iwr-width}, for such values of $\alpha_0$ and $\alpha$, the asymptotic width of $\calH_n^\iwr$ does not change substantially even when $|\alpha - \alpha_0|\gg 0$.

Nevertheless, one may be unsatisfied by asymptotic widths of $\calH_n^\iwr$ scaling linearly with $\log(1/\alpha)$. The \aphci  presented in the next section will remedy this issue, resulting in an asymptotic width which scales as $\Theta(\sqrt{\sigma^2\log(1/\alpha)})$ similarly to the Wald CI. However, despite this theoretical appeal, we find that ex-ante anchoring works best in practice (Section~\ref{sec:simulations}).

\subsection{Choosing \texorpdfstring{$\lambda$}{parameters}, Option II: The method of mixtures}
\label{sec:aphcis-via-mom}

We can provide asymptotic e-values which are free of $\lambda$ using the ``method of mixtures,'' a technique pioneered in its modern form by Herbert Robbins in the context of tail inequalities for supermartingales~\citep{robbins1970statistical}. It has become a cornerstone of anytime-valid inference, deployed in both univariate (e.g.,~\citep{kaufmann2021mixture,howard2021time,waudby2024time}) and multivariate (e.g,~\citep{de2009self,chugg2025time}) settings.  The following lemma gives sufficient conditions under which we may apply the method of mixtures to a family of asymptotic e-variables to obtain another asymptotic e-variable. The proof is in Appendix~\ref{proof:mixture-is-valid}. 

\begin{proposition}
\label{prop:mixture-is-valid}    
Let $\Lambda\subset\Re$ and let $\pi$ be any (data-free) probability distribution on $\Lambda$. Suppose that for each $\lambda\in\Lambda$, the sequence $(E_n(\lambda))$ is an asymptotic e-variable for some set of distributions $\calP$. If $(\sup_{\lambda\in\Lambda} E_n(\lambda))$ is uniformly integrable for all $P\in\calP$, then the process $(E_n(\pi))$ defined as 
\begin{equation}
\label{eq:mixture-of-lambda}
    E_n(\pi) := \int_\Lambda E_n(\lambda)\d\pi(\lambda),
\end{equation}
is an asymptotic e-variable. Further, if $(\sup_{\lambda\in\Lambda} E_n(\lambda))$ is $\calP$-uniformly integrable, then $(E_n(\pi))$ is a $\calP$-uniform asymptotic e-variable. 
\end{proposition}

Appendix~\ref{app:generalized-mixture} provides more general sufficient conditions under which the mixture in \eqref{eq:mixture-of-lambda} can result in an asymptotic e-variable. See Proposition~\ref{prop:mixture-generalized}. However, the result in \eqref{eq:mixture-of-lambda} is satisfactory for our purposes in this section.

\begin{remark}
In the nonasymptotic setting, any mixture of e-variables is again an e-variable. This follows immediately by Tonelli's theorem; no other conditions are needed. We might thus wonder if Proposition~\ref{prop:mixture-is-valid} truly requires any additional assumptions beyond $(E_n)$ being an asymptotic e-variable. The answer is yes. Consider $E_n(\lambda) = 1 + 2^n \ind{\lambda\in (2^{-(n+1)}, 2^{-n}]}$ for $\lambda \in [0,1]$. Let $\pi$ be the uniform distribution over $[0,1]$. For large enough $n$ and any fixed $\lambda\in[0,1]$, we have $E_n(\lambda) =1$, so $\limsup_n \E_P[E_n(\lambda)] = 1$. On the other hand, $\int E_n(\lambda) \d\pi(\lambda) = 1 + 1$, so $\limsup_n \E_P[E_n^\pi]=2$. Therefore, \emph{some} extra condition beyond being an asymptotic e-variable is required for the mixture to also be an asymptotic e-variable. 
\end{remark}

Using that $(E_n^\iwr)$ is uniformly integrable on compact sets (Lemma~\ref{lem:Ereg_ui}), we immediately obtain the following result.  

\begin{corollary}
\label{cor:compact_mix}
    Let $\Lambda\subseteq\mathbb{R}$ be a compact set. Let $\pi$ be any (data-free) probability distribution on $\Lambda$. Then, for any fixed $\eta>0$, the mixture $E_n^\iwr(\theta;\pi) = \int_\Lambda E_n^\iwr(\theta;\lambda)\d\pi$
    is an asymptotic e-variable for $\calM_2(\theta)$.
\end{corollary}

In the nonasymptotic setting, when one has a family of e-variables $E_n(\lambda)$ valid for all $\lambda\in\Re$, a powerful approach is to use a Gaussian mixture (or some other distribution supported on all of $\Re$). Indeed, this was Robbins' original approach~\citep{robbins1970boundary}. Unfortunately, Corollary~\ref{cor:compact_mix} precludes such a choice. It appears that one would need to restrict the class of distributions to those with strict control of the moment-generating function to be able to apply the Gaussian mixture. This is untenable for our purposes, as we want to assume the existence of a few moments at most.

Instead, we will use a \emph{truncated} Gaussian to mix over a finite region $[-R,R]$. Such an approach has previously been used to develop nonasymptotic confidence sequences \citep{chugg2025closed,chugg2025time}.  To state the result, define the object
\begin{equation}
    I_R(y,u) := \sqrt{\frac{\pi}{u}}\exp\left\{\frac{y^2}{4u}\right\} (\Phi(g_1) - \Phi(g_2)), 
\end{equation}
where $\Phi$ is the cdf of the standard normal and 
 $g_1 = R\sqrt{2u} - y/\sqrt{2u}$ and $g_2 = -(R\sqrt{2u} + y/\sqrt{2u})$. It is easy to see that $I_R(y,u)$ is symmetric about $y=0$. Indeed, $I_R(y,u)$ should be thought of as an upward-facing parabola with minimum at $y=0$ and shape and (positive) intercept determined by $u$ and $R$. 

 Mixing $E_n^\iwr$ over a truncated Gaussian gives an asymptotic e-variable stated in terms of $I_R(y,u)$. The result is fairly unenlightening so we relegate it  to Appendix~\ref{app:truncated-gaussian} and proceed immediately to the resulting \aphci. The proof of the following theorem is provided in Appendix~\ref{proof:iwr-mixture-aphci}. 

 \begin{theorem}
     \label{thm:iwr-mixture-aphci}
     Let $(X_n)$ be a sequence of iid random variables drawn from some distribution in $\dan(\mu)$. 
     Fix any $R,\kappa>0$ and let $Z_{R,\kappa} = \kappa\sqrt{2\pi} (\Phi(R/\kappa) - \Phi(-R/\kappa))$.  Then the process  $(\calH_n^{\mix,\iwr}(\alpha;R,\kappa))$ is an  \aphci for $\mu$ where $\calH_n^{\mix,\iwr}(\alpha;R,\kappa) = (\Xbar_n \pm W_n^{\mix,\iwr}(\alpha;R,\kappa))$ with 
    \begin{equation}
    \label{eq:mix-iwr-aphci}
        W_n^{\mix,\iwr}(\alpha;R,\kappa) = \frac{\shat_n y^\star_\alpha}{\sqrt{n - (y^\star_\alpha)^2}},
    \end{equation}
    where  
    $y_\alpha^\star$ solves $I_R(y_\alpha^\star,(\kappa^2 + 1)/(2\kappa^2)) = Z_{R,\kappa}/\alpha$. If the denominator in~\eqref{eq:mix-iwr-aphci} is non-positive, we take the CI to be all of $\Re$.
 \end{theorem}

We note that due to the nice shape of $I_R$, the implicit equation $I_R(y_\alpha^\star,(\kappa^2 + 1)/(2\kappa^2)) = Z_{R,\kappa}/\alpha$ is easy to approximate to arbitrary precision via root finding. Thus, the \aphci of Theorem~\ref{thm:iwr-mixture-aphci} is computationally tractable and easy to implement. See Section~\ref{sec:simulations} for experiments.

Theorem~\ref{thm:iwr-mixture-aphci} is not a distribution-uniform result. This is because it is unclear whether $(\sup_{\lambda\in[-R,R]}E_n^\iwr(\lambda))$ is $\calP$-uniformly integrable, whereas it is pointwise uniformly integrable by Lemma~\ref{lem:Ereg_ui}. Proving $\calP$-uniform integrability of $(E_n^\iwr(\lambda))$ (even without the supremum) seems to necessitate extending the work of \citet{gine1997student} on the asymptotics of the Student t-distribution to the distribution-uniform setting, an interesting problem in and of itself which we leave for future inquiry.

The asymptotic width of $W_n^{\mix,\iwr}$ is particularly easy to analyze under a finite second-moment assumption. Indeed, if $\shat_n^2 \xrightarrow{a.s.}\sigma^2$, then 
\begin{equation}
\sqrt{n}W_n^{\mix,\iwr}(\alpha;R,\kappa) \xrightarrow{a.s.} \sigma y_\alpha^\star. 
\end{equation}
Recalling the discussion in Section~\ref{sec:iwr-evariable}, we had $W_n^\wald \to \sigma z_{1-\alpha/2}$. Experimentally, we find that $y_\alpha^\star<2z_{1-\alpha/2}$ for most reasonable values of $R$ and $\kappa$, implying that the asymptotic widths of these mixture e-variables have asymptotic width at most twice that of the Wald CI. Moreover, the bounds are very stable in practice regardless of the choice of $R$ and $\kappa$.  Analytically, Lemma~\ref{lem:ystar-bound} in Appendix~\ref{app:truncated-gaussian} shows that for $\kappa=1$, $y_\alpha^\star \to 2\sqrt{\log(\sqrt{2}/\alpha)}$ as $R\to\infty$. We thus obtain the desired square root behavior on $\log(1/\alpha)$ which was missing in Section~\ref{sec:ex-ante-tuning}.

\subsection{Event partitioning and the R-WS asymptotic e-variable}
\label{sec:rws-evariable}

Instead of directly defining a sequence of random variables whose means are asymptotically bounded by one, an alternative method of constructing an asymptotic e-variable is to perform event partitioning as follows. We take a ``nice'' sequence of random variables $(G_n)$ that has good behavior on events $(A_n)$, and truncate the sequence by the probability that it misbehaves, i.e., the probabilities of $(A_n^c)$. Formally, consider the following observation.

\begin{observation}
\label{obs:truncation}
    Let $(G_n)$ be a sequence of nonnegative random variables, and let $(A_n)$ be a sequence of events such that $\limsup_n \E_P[G_n \ind{A_n}]\leq 1$ for all $P\in\calP$. Then, for any sequence of scalars $(T_n)$, such that $T_n P(A_n^c) \to 0$ for all $P\in\calP$,  
    $E_n = G_n\wedge T_n$ defines an asymptotic e-variable for $\calP$. Meanwhile, if $\limsup_n\sup_{P\in\calP} \E_P[G_n \ind{A_n}]\leq 1$ and $T_n \sup_{P\in\calP} P(A_n^c) \to 0$ then $(E_n)$ is a $\calP$-uniform asymptotic e-variable. 
\end{observation}
\begin{proof}
    Decompose $\E_P[E_n] = \E_P[G_n \ind{A_n}] + \E_P[T_n\ind{A_n^c}] = \E_P[G_n \ind{A_n}] + T_n P(A_n^c)$. 
    By assumption, the first term goes to one as $n\to\infty$ and the second goes to zero. 
\end{proof}

We will apply Observation~\ref{obs:truncation} to a sequence of random variables designed to approximate the process $(E_n^\star(\lambda))$ in the limit, where for any $\lambda\in\Re$, 
\begin{equation}
\label{eq:howard-evar}
    E_n^\star(\lambda) := \exp \left \{\lambda  \sum_{i=1}^n (X_i - \theta) -  \frac{\lambda^2}{2}U_n^\otimes \right \}, \quad U_n^\otimes := \frac{1}{3}\sum_{i=1}^n ((X_i - \theta)^2 + 2\sigma^2). 
\end{equation}
\citet[Table 3]{howard2020time} demonstrate that $(E_n^\star(\lambda))$ is a supermartingale with initial value one for distributions $(X_n)$ with constant conditional mean $\theta$ and finite variance $\sigma^2$. Therefore, they satisfy $\E_P[E_\tau^\star(\lambda)]\leq 1$ for all stopping times $\tau$ under such distributions $P$. The following theorem is based on building a process that approximates $E_n^\star$ in the limit and bounding its probability of failure. It relies on a nonasymptotic SLLN from \citet{ruf2025concentration} and is thus eponymously named the \rws asymptotic e-variable. The proof can be found in Appendix~\ref{proof:rws-variable}.

\begin{theorem}
\label{thm:rws-evariable}
Let $(X_n)$ be a sequence of iid random variables. 
Fix $\gamma\in(0,1/2)$ and $\delta>0$. Let $(T_n) \nearrow \infty$ and $(\eps_n) \searrow 0$ be nonnegative sequences satisfying $T_n = o(n^{\delta\gamma/2} \eps_n^{2+\delta/2})$. 
Then $(E_n^\rws(\theta; \lambda,\eps_n,T_n))$ is an asymptotic e-variable for all distributions in $\calM_{2+\delta}(\theta)$, where 
\begin{equation}
\label{eq:rws-evar-general}
    E_n^\rws(\theta; \eps,T) := \int_{\lambda\in\Re}\exp\left(\lambda S_n(\theta) - n(\shat_n^2 + \eps_n)\frac{\lambda^2}{2}\right)\d F(\lambda) \wedge T_n, 
\end{equation}
and $F$ is any distribution over $\Re$ (including a point mass). 
Moreover, if the sequence $(X_n)$ has a $\calP$-uniformly bounded $2+\delta$ moment for $\calP\subseteq \calM_{2+\delta}(\theta)$ then $(E_n^\rws(\theta; \lambda,\eps_n,T_n))$ is a $\calP$-uniform asymptotic e-variable. 
\end{theorem}

As made evident in the statement of Theorem~\ref{thm:rws-evariable}, a benefit of the truncation approach is that it can elegantly handle mixtures while sidestepping issues of uniform integrability. Indeed, if $(G_n(\lambda)\wedge T_n)$ is an asymptotic e-variable for all $\lambda\in\Lambda$, then so too is $(\int G_n(\lambda)\d F\wedge T_n)$ for any distribution $F$ over $\Lambda$. 

In particular, as we were tempted to do in Section~\ref{sec:aphcis-via-mom} but could not, we take can $F$ to be a centered Gaussian in the above theorem. Taking $\eps_n = 1/\log(n)$ and $\gamma = 0.49$ for concreteness, we obtain the following corollary. The calculation is standard, but is presented for completeness in Appendix~\ref{proof:rws-normal-mixture}. 

\begin{corollary}
\label{cor:rws-normal-mixture}
Let $(X_n)$ be a sequence of iid random variables. 
Fix $C,\rho>0$ and let $\hat{u}_n = n(\shat_n^2 + \log^{-1}n)$. Then the sequence $(E_n^\rws(\theta;v))$ is an asymptotic e-variable for all distributions in $\calM_{2+\delta}(\theta)$ where 
\begin{equation}
\label{eq:rws-evar-mix}
    E_n^\rws(\theta; \rho) = \frac{1}{\sqrt{\rho \hat{u}_n  + 1}} \exp\left\{\frac{\rho S_n^2(\theta)}{2\rho\hat{u}_n+ 2}\right\} \wedge C\cdot n^{\delta \cdot 0.24}. 
\end{equation}
Moreover, if the sequence $(X_n)$ has a $\calP$-uniformly bounded $2+\delta$ moment for some $\calP\subseteq\calM_{2+\delta}(\theta)$, then $(E_n^\rws(\theta;\rho))$ is a $\calP$-uniform asymptotic e-variable.  
\end{corollary}

The truncation in~\eqref{eq:rws-evar-general} and~\eqref{eq:rws-evar-mix} is an explicit bound on the growth of the e-variable and translates directly to a limit on the size of the resulting \aphci. Indeed, for a truncation sequence $(T_n)$, the \aphci will be vacuous for all $n$ such that $T_n < 1/\alpha$. In~\eqref{eq:rws-evar-mix} for example, if we posit the existence of a third moment $(\delta=1)$, then to have a non-vacuous confidence interval at $\alpha=0.05$ we require $C n^{0.24}\geq 1/0.05$, i.e., $n\approx  2.64\cdot 10^5 / C^{1/0.24}$. One can of course take $C$ extremely large to compensate, but we emphasize that $C$ should not be data-dependent and must be fixed in advance. Further, as is evident from examining the proof, $C$ determines when the asymptotics kick in. Thus, larger values of $C$ will result in more false positives for modest values of $n$. We suggest setting $C=100$ as the standard. As shown by Table~\ref{tab:posthoc_risk}, this choice of $C$ results in an empirical risk well below one. 

While the \rws asymptotic e-variable has explicit truncation, we note that the \iwr e-variable, both in Theorem~\ref{thm:iwr-asymp-evar} and its mixture, have an implicit form of truncation. 
In particular, in the case of ex ante tuning (Section~\ref{sec:ex-ante-tuning}), since $\lambda$ cannot depend on $\alpha$, as $\alpha$ drifts further from the anchor $\alpha_0$, the e-variable becomes less powerful. And in the method of mixtures (Section~\ref{sec:aphcis-via-mom}), the fact that we can only mix over a finite range $[-R,R]$ inherently limits the power of the e-variable. Thus, there seems to be no way to escape some form of truncation, whether implicit or explicit, for asymptotic e-variables. We posit that this is true in general, not just for the e-variables presented in this paper. We leave it as an open question whether this observation can be formalized. 

Let us state the \aphci that results from the asymptotic e-variable in Corollary~\ref{cor:rws-normal-mixture}. The proof follows the usual formula: we find those $\theta$ such that $E_n^\rws(\theta;\rho) < 1/\alpha$, taking into account that the CI is vacuous if the truncation level is below $1/\alpha$. The result then follows from Proposition~\ref{prop:cs-via-evalue}.

\begin{theorem}
\label{thm:rws-aphci}
    Let $(X_n)$ be a sequence of iid random variables drawn from some distribution in $\calM_{2+\delta}(\mu)$ for some $\delta>0$. 
    Fix $C,\rho>0$ and let $\hat{u}_n = n(\shat_n^2 + \log^{-1}n)$. 
    Then $(\calH_n^\rws(\alpha;\rho))$ is an \aphci for $\mu$, where 
    \begin{equation*}
        \calH_n^\rws(\alpha;\rho) = (\Xbar_n \pm W_n^\rws(\alpha;\rho)) \text{~ where ~} W_n^\rws(\alpha;\rho) := \sqrt{\frac{2\rho\hat{u}_n + 2}{n^2\rho}\log\left(\frac{\sqrt{\rho \hat{u}_n + 1} }{\alpha}\right),}
    \end{equation*}
    as long as $C n^{\delta\cdot 0.24}\geq 1/\alpha$,  otherwise $\calH_n^\rws(\alpha;\rho)=\Re$. If $(X_n)$ have a $\calP$-uniformly bounded $2+\delta$ moment, then $(\calH_n^\rws(\rho))$ is a $\calP$-uniform \aphci. 
\end{theorem}

A reasonable value of $\rho$ in Theorem~\ref{thm:rws-aphci} is $\rho=2$, but we show in Appendix~\ref{app:sims} that reasonable values of $\rho$ do not have much of an effect on the width. In fact, the asymptotic width is independent of $\rho$ and scales as $\sqrt{\log(n)/n}$. 
More precisely, a bit of arithmetic gives that  $\sqrt{n}W_n^\rws(\alpha;\rho) \xrightarrow[n\to\infty]{a.s.} \infty$, and 
\begin{equation}
  \sqrt{\frac{n}{\log n}}W_n^\rws(\alpha;\rho) \xrightarrow[n\to\infty]{a.s.} \sigma,  
\end{equation}
just as Robbins' original mixture (see \citep[Appendix C.3]{chugg2025closed}). Note that this rate of $\sqrt{\log(n)/n}$ is distinct from the rate of $1/\sqrt{n}$ achieved by the previous \aphcis. 
Such a rate is common in the literature on time-uniform confidence \emph{sequences}, as opposed to confidence \emph{intervals}. Confidence sequences cannot scale as $1/\sqrt{n}$ due to the law of the iterated logarithm~\citep{robbins1970statistical}, and Robbins' mixture is a classical way to construct such objects. This provides us with a hint that $(\calH_n^\rws(\alpha;\rho))$ is perhaps satisfying a stronger guarantee than the other \aphcis given thus far. This is indeed the case. In fact, (a mildly modified version of) $(\calH_n^\rws)$ is an example of a \emph{post-hoc asymptotic confidence sequence} and $(E_n^\rws)$ an example of an \emph{asymptotic e-process}. We will introduce these objects in Section~\ref{sec:asymp-eprocesses}.

\subsection{Simulations}
\label{sec:simulations}

\begin{figure}[t]
    \centering
    \begin{subfigure}[b]{0.45\textwidth}
        \centering
        \includegraphics[width=\linewidth]{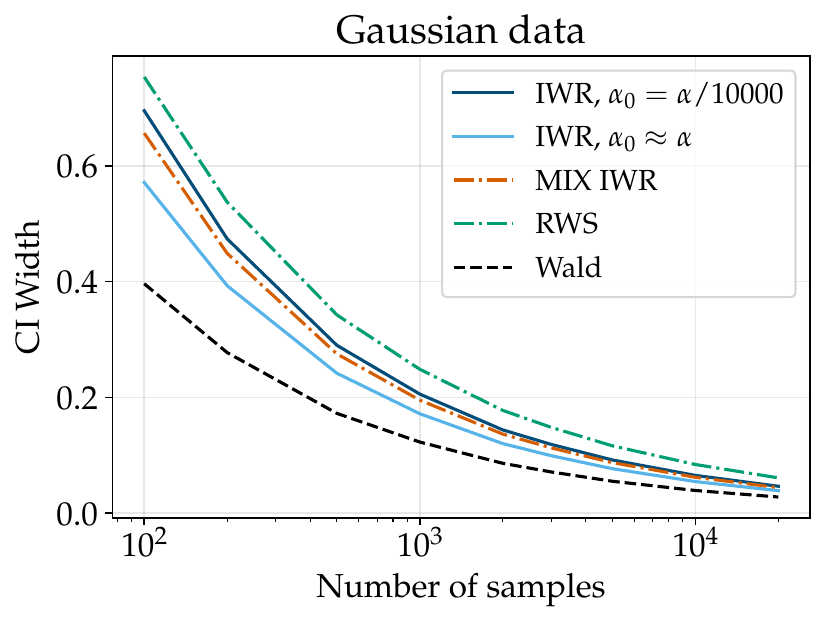}
    \end{subfigure}
    \hspace{0.2cm}
    \begin{subfigure}[b]{0.45\textwidth}
        \centering
        \includegraphics[width=\linewidth]{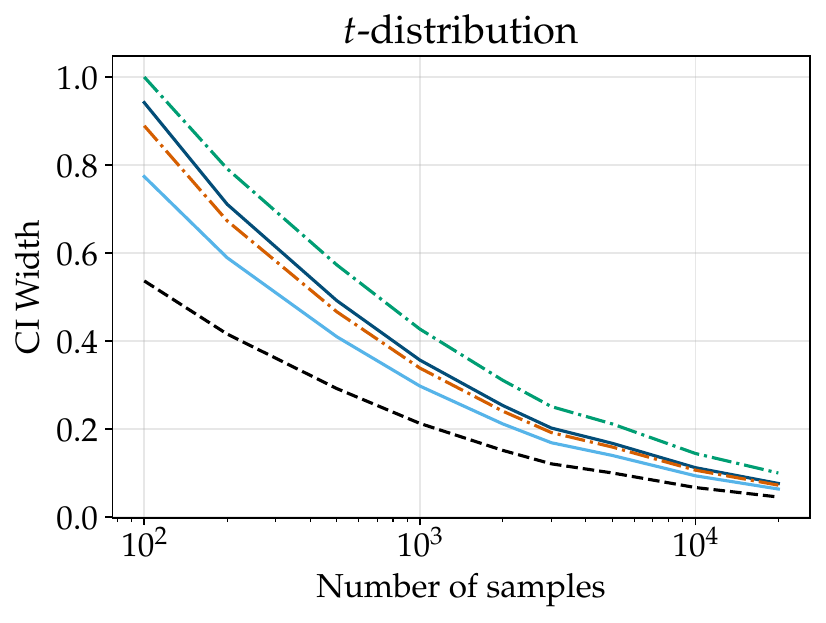}
    \end{subfigure}
    \caption{The width of four \aphcis compared to the Wald CI for Gaussian data and heavy-tailed data coming from a t-distribution with three degrees of freedom. ``IWR'' refers to the \aphci of Theorem~\ref{thm:iwr-aphci} with $\lambda$ chosen via ex ante anchoring. 
    ``MIX IWR'' refers to the \aphci of Theorem~\ref{thm:iwr-mixture-aphci}. 
    We use $\rho=2$ for $\calH_n^\rws$. See Appendix~\ref{app:sims} for further details.}
    \label{fig:widths}
\end{figure}

\begin{figure}[t]
    \centering
    \begin{subfigure}[t]{0.45\textwidth}
        \includegraphics[height=5cm]{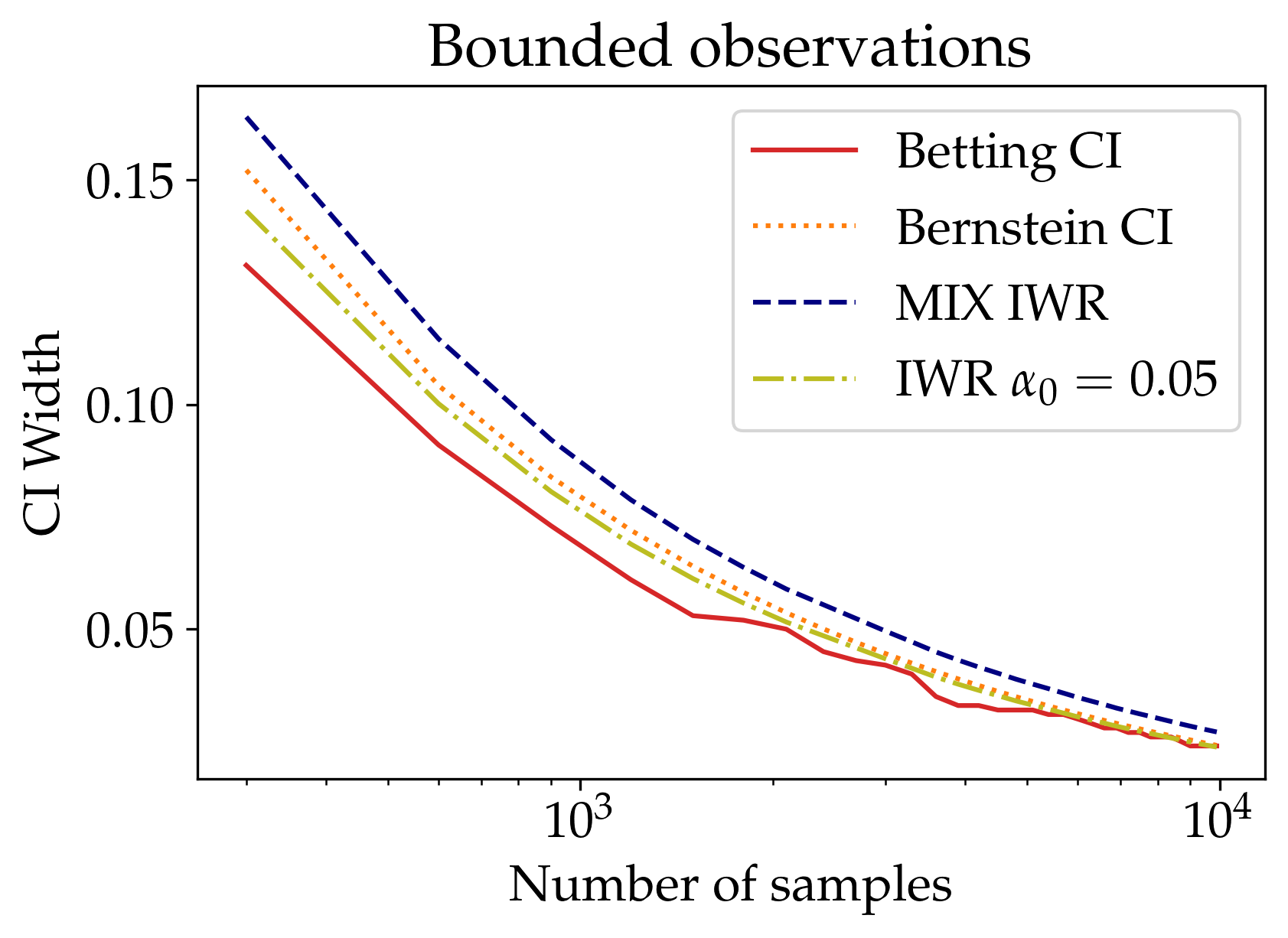}
    \end{subfigure}
    \begin{subfigure}[t]{0.45\textwidth}
        \includegraphics[height=5cm]{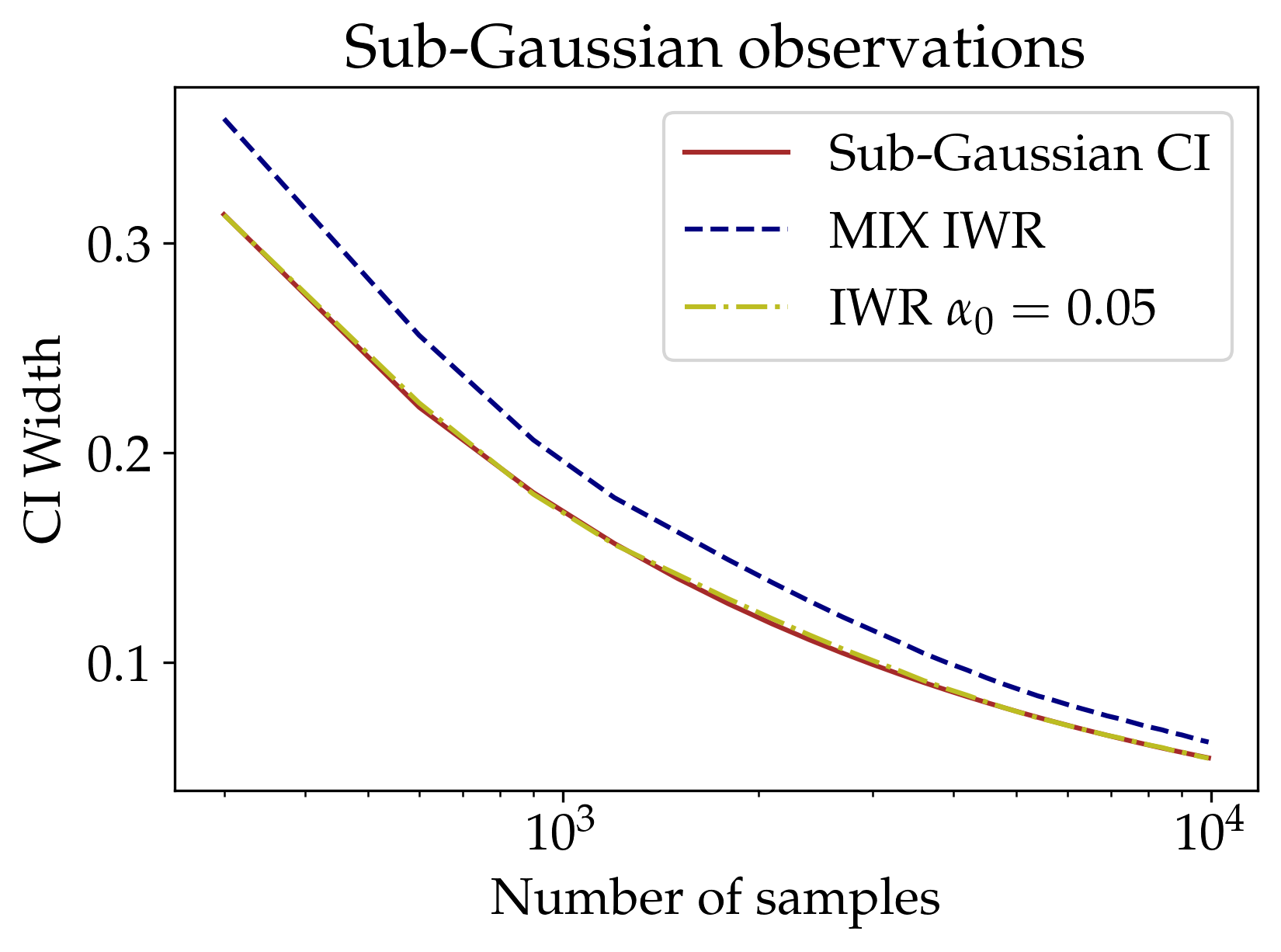}
    \end{subfigure}
    \caption{Comparison of the widths of some of our \aphcis compared to nonasymptotic CIs based on (nonasymptotic) e-variables. For bounded data, which we take to be iid Bernoulli(0.25), we compare to the classical Bernstein CI and also to the state-of-the-art betting-based empirical Bernstein CI of \citet{waudby2024estimating}. For sub-Gaussian observations (generated as $N(0,1)$ random variables) we compare to the standard CI based on the Chernoff method---see the text for more detail. We use $\alpha=0.05$ across all simulations, and we implement $\calH_n^\iwr$ with $\lambda = \sqrt{2\log(2/\alpha)}$. }
    \label{fig:asymp_vs_nonasymp}
\end{figure}

Figure~\ref{fig:widths} plots the widths of the three \aphcis discussed in Section~\ref{sec:constructing-aphcis}, given by  Theorems~\ref{thm:iwr-aphci}, \ref{thm:iwr-mixture-aphci} and \ref{thm:rws-aphci}, against the width of the Wald CI in~\eqref{eq:wald-CI}. As expected, the Wald CI is reliably tighter (but, of course, not post-hoc valid), with the gap closing as the sample size increases.  The difference in widths between the Wald CI and our \aphcis typically less than 0.05 for $n=10^4$. 

For $\calH_n^\iwr$ we choose $\lambda$ via ex ante anchoring (Section~\ref{sec:ex-ante-tuning}). That is, we fix $\alpha_0$ and take $\lambda = \lambda_0= \sqrt{2\log(2/\alpha_0)}$. We test two regimes for the anchor $\alpha_0$: When $\alpha_0 \ll \alpha$ (by a factor of $10^4$) and when $\alpha_0 \approx \alpha$. When $\alpha_0 \approx \alpha$, $\calH_n^\iwr$ is tighter than $\calH_n^{\mix,\iwr}$. In fact, as  foreshadowed in Section~\ref{sec:ex-ante-tuning}, it takes significant divergence between $\alpha$ and $\alpha_0$ for $\calH_n^{\mix,\iwr}$ to be tighter. A practitioner thus faces a tradeoff between $\calH_n^{\iwr}$ and $\calH_n^{\mix,\iwr}$: For many values of $\alpha$ and $\alpha_0$ the former will be tighter, but the latter has a better worst-case performance. 

The \aphci based on the \rws e-variable is looser than both $\calH_n^\iwr$ and $\calH_n^{\mix,\iwr}$. As we hinted at in Section~\ref{sec:rws-evariable} and to be fleshed out in Section~\ref{sec:asymp-eprocesses}, this is principally because $\calH_n^\rws$ is providing a stronger guarantee than the other two \aphcis, and is in fact a post-hoc asymptotic confidence sequence.

A natural question is how the width of our large-sample CIs, which depend on asymptotic e-variables, compare with the widths of finite-sample CIs based on \emph{nonasymptotic} e-variables. Figure~\ref{fig:asymp_vs_nonasymp} compares the width of the two \aphcis based on the \iwr e-variable to several well-known finite-sample CIs for bounded and sub-Gaussian observations. For bounded observations drawn iid from Bernoulli(0.25), we compare to the classical Bernstein CI in addition to the recent bound of \citet{waudby2024estimating} (betting CI). 
The latter is the tightest empirical Bernstein bound (i.e., it does not require knowledge of the variance) in the literature in addition to being nearly theoretically optimal~\citep{shekhar2023near}. 

The width of $\calH_n^\iwr$ is on par or better than the oracle Bernstein CI but looser than the betting CI. It's worth emphasizing that both the betting CI and the oracle Bernstein bound are taking advantage of the fact that the observations are bounded, whereas the two \aphcis are not, being valid for all distributions with a finite second moment. 

We perform a similar experiment for sub-Gaussian data, comparing the two \aphcis to the sub-Gaussian CI based on the Chernoff method, i.e., $(\Xbar_n \pm \sqrt{2\log(2/\alpha)/(n\sigma^2)})$. Again,  $\calH_n^\iwr$ is roughly on par with the finite-sample CI, while $\calH_n^{\mix,\iwr}$ is looser than both. Our main interest in Figure~\ref{fig:asymp_vs_nonasymp} is to investigate the widths of asymptotic versus nonasymptotic CIs, as opposed to the post-hoc nature of the asymptotic CIs. Therefore, we instantiated $\calH_n^\iwr$ with ex ante anchoring and anchor $\alpha = \alpha_0$. 

Our third experiment studies the risk of our \aphcis compared to the Wald CI by mimicking an analyst who selects $\alpha$ in a data-dependent manner. We generate $n=1000$ samples from a Gaussian distribution with mean zero and variance four, and compare the Wald CI $\calH_n^\wald$ against the \aphci $\calH_n^{\iwr}$ with $\lambda=\sqrt{2\log(2/0.01)}$, and the two mixture \aphcis $\calH_n^{\mix,\iwr}$ with $R=20$ and $\kappa=1$, and $\calH_n^{\mix,\rws}$ with $\rho=2$. To verify the risk guarantee, we simulate a ``p-hacking'' procedure where the analyst searches for the smallest significance level $\alpha^\star$ that permits the rejection of the null hypothesis $H_0: \theta = 0$. Specifically, for each trial, we identify: $\alpha^\star = \inf \{ \alpha>0 : 0 \notin \calH_n(\alpha) \}$,
using a binary search over the interval $[\alpha_{\min}, \alpha_{\max}]$ with tolerance $10^{-5}$ (here $\alpha_{\min}=10^{-5}$ and $\alpha_{\max}=500$). Note that this procedure is operationally equivalent to estimating the supremum risk $\sup_{\alpha} \ind{0 \notin \calH_n(\alpha)}/\alpha$. We evaluate the risk under two conditions:
\begin{enumerate}
    \item \textbf{Practical regime ($0<\alpha\le1)$:} The search is restricted to the interpretable limit of statistical inference. If the CI contains zero for all $\alpha\le1$, the risk contribution is simply zero.
    \item \textbf{Theoretical regime ($0<\alpha$):} Experimentally, we extend the search interval to $\alpha_{\max}=500$. This upper bound was selected empirically because it was sufficiently large to ensure that the null hypothesis was rejected in all 1000 simulation trials. While $\alpha>1$ lacks a standard statistical interpretation, forcing the method to reject in every trial allows us to benchmark the tightness of the bound.
\end{enumerate}

\begin{table}[t]
    \centering
    \begin{tabular}{l|c|c}
        \textbf{Method} & \textbf{Practical Risk} ($0<\alpha \le 1$) & \textbf{Theoretical Risk} ($0<\alpha$) \\
        \hline 
         $\risk_n(\calH_n^\wald)$ & \multicolumn{2}{c}{\color{red} 7.089} \\
         \cline{2-3}
         $\risk_n(\calH_n^{\iwr})$ & 0.532 & 0.620 \\
         $\risk_n(\calH_n^{\mix,\iwr})$ & 0.367 & 0.976 \\
         $\risk_n(\calH_n^{\mix,\rws})$ & 0.001 & 0.028
    \end{tabular}
    \caption{Empirical risk under data-dependent selection of $\alpha$. The ``Practical Risk'' column reports the risk when searching for significance within the range $0 <\alpha \le 1$. The ``Theoretical Risk'' column extends the search to $\alpha_{\max}$ to force rejection. For the Wald CI, the risk is identical in both regimes because $\alpha^\star$ is almost surely $\le 1$.}
    \label{tab:posthoc_risk}
\end{table}

Table~\ref{tab:posthoc_risk} presents the results. As expected, the Wald CI yields risk levels well beyond one, while our \aphcis maintain valid risk control. The risk of $(\calH_n^{\mix,\rws})$ is well below one in both regimes, reflecting the fact that it is more conservative than the other \aphcis. As alluded to in Section~\ref{sec:rws-evariable}, this is due to the fact that it is providing a stronger guarantee than either $(\calH_n^\iwr)$ or $(\calH_n^{\mix,\iwr})$---see Section~\ref{sec:asymp-eprocesses}. 

Note that in the practical regime, the risk is strictly below one for all \aphcis. This conservatism is the necessary price of post-hoc validity: the theoretical bound must account for the possibility of rejection at any $\alpha > 0$, including values $\alpha > 1$ which are statistically meaningless but mathematically possible.

\section{Post-hoc Sequential Inference and Asymptotic e-processes}
\label{sec:asymp-eprocesses}

The discussion thus far has focused on post-hoc analogues of asymptotic confidence intervals, hypothesis tests, and p-values. There is a growing literature on \emph{sequential} (but not post-hoc) analogues of statistical inference procedures, taking the form of confidence sequences, sequential hypothesis tests, and anytime p-values, respectively. The roots of these sequential procedures lie in the work of \citet{wald1945sequential,wald1947sequential}, Robbins and colleagues \citep{robbins1970statistical,darling1967confidence,lai1976confidence} and recent years have seen increased activity; see Chapter 7 of \citet{ramdas2024hypothesis}, the review paper by \citet{ramdas2023game}, and other related papers \citep{howard2021time,waudby2024estimating,orabona2023tight}. Furthermore, distribution-uniform asymptotic versions of these procedures have been developed by \citet{waudby2023distribution} being enabled by recent advances in distribution-uniform strong asymptotics \citep{waudby2024distribution,waudby2025nonasymptotic} and distribution-pointwise anytime-valid inference \citep{robbins1970boundary,waudby2024time,bibaut2022near}. 

However, none of these were shown to enjoy post-hoc validity in the sense of the present paper, and indeed, it is not clear \emph{a priori} what should be meant by ``post-hoc sequential inference.'' In this section, we provide definitions for post-hoc asymptotic sequential inference and derive explicit procedures satisfying those definitions. 

\subsection{Definitions and desiderata}

In the definitions to follow, we will frequently rely on the notion of an ``upper triangular array'', an object that frequently appears in time-uniform asymptotics \citep{waudby2023distribution}. An upper triangular array of random variables is a doubly indexed sequence $\infseqm{\infseqkm{X_k^\brackm}}$, so called because it can be visually depicted as 
\[
\begin{tikzpicture}[baseline=(arr.center)]
  \node (arr) {$
  \begin{alignedat}{3}
    & X_{1}^{(1)},\ & X_{2}^{(1)},\ &X_{3}^{(1)}, &\dots \\
    & & X_{2}^{(2)},\ &X_{3}^{(2)}, &\dots \\
    && &X_{3}^{(3)},\ &\dots \\
    &&&& \ddots
  \end{alignedat}
  $};

  \draw[-{Latex[length=2mm]}]
    ($(arr.north east)+(1.2em,0.2em)$) --
    ($(arr.south east)+(1.2em,-0.2em)$)
    node[pos=0.12,right=10pt,rotate=-90] {asymptotics};

  \draw[-{Latex[length=2mm]}]
    ($(arr.south west)+(0.3em,-1.2em)$) --
    ($(arr.south east)+(0.8em,-1.2em)$)
    node[midway,below=3pt] {time-uniformity};
\end{tikzpicture}
\]
where ``asymptotics'' are with respect to $m$ tending to infinity and ``time-uniformity'' is with respect to those indices $k \geq m$. With this language in mind, consider the following definition.

\begin{definition}[Asymptotic post-hoc confidence sequences and anytime p-values]\label{definition:asymptotic post-hoc confseq and p-val}
  Let $\calP$ be a collection of distributions and $f : \NN \to (0, 1)$ a nonincreasing function. We say that the collection of upper triangular arrays $\infseqm{\infseqkm{\calH_k^\brackm(\alpha)}}$, with $\alpha \in (0, 1)$, is a distribution-uniform  \emph{asymptotic post-hoc confidence sequence} (\aphcs) for the parameters $(\theta(P))_\Pin$ and that it has a range growth function $f$ if for any stopping time $\tau$, it holds that
  \begin{equation}
  \label{eq:post-hoc-confseq}
    \limsup_{m \to \infty}\sup_{P \in \calP}\E_P \left [ \sup_{\alpha > f(m)}\frac{\ind{ \theta(P) \notin \calH_{\tau \lor m}^\brackm(\alpha)}}{\alpha} \right ] \leq 1.
  \end{equation}
  Similarly, for a subset of distributions $\calP_0 \subseteq \calP$, we say that the upper triangular array $\infseqm{\infseqkm{p_k^\brackm}}$ is a $\calP_0$-uniform \emph{asymptotic post-hoc anytime p-value} with range growth function $f$ if for any stopping time $\tau$,
  \begin{equation}
  \label{eq:post-hoc-anytime-pval}
      \limsup_{m \to \infty} \sup_{P \in \calP_0} \E_P \left [ \sup_{\alpha > f(m)} \frac{\1 \{ p_{\tau \lor m}^\brackm \leq \alpha \}}{\alpha} \right ] \leq 1.
  \end{equation}
\end{definition}

Let us briefly parse Definition~\ref{definition:asymptotic post-hoc confseq and p-val}, turning our attention to the display \eqref{eq:post-hoc-confseq}. First suppose that the guarantee in~\eqref{eq:post-hoc-confseq} holds for a fixed $m$, not only in the limit. In that case, the set $\calH^\brackm_\tau$ is a post-hoc confidence interval at all stopping times $\tau$ taking values at least $m$ but only for those $\alpha$ such that $\alpha > f(m)$. That is, we imagine an analyst committing to first gathering $m$ observations and performing an analysis. On the basis of those results, they may continue sampling and the results will remain valid (i.e., the risk will remain bounded by one as long as $\alpha > f(m)$). In other words, $\calH^\brackm_\tau$ satisfies optional continuation after an initial burn-in period $m$.
The asymptotic guarantee of an \aphcs is simply one that holds as the burn-in period $m$ tends to $\infty$.

Given the definitions of post-hoc asymptotic sequential inference procedures above, it is natural to ask: \emph{What is the sequential analogue of an asymptotic e-value?} Similar to how e-processes \citep[\S 7]{ramdas2024hypothesis} are the sequential analogue of e-values, the following definition of an asymptotic e-process serves as a sequential analogue of an asymptotic e-value. While asymptotic e-values and non-asymptotic e-processes have appeared and been studied in the literature, the following definition appears to be new.
\begin{definition}[Asymptotic e-processes]\label{definition:asymptotic e-proc}
  Let $\calP$ be a collection of distributions. We say that an upper triangular array $\infseqm{\infseqkm{E_k^\brackm}}$ is a distribution-uniform \emph{asymptotic $\calP$-e-process} if for any stopping time $\tau$, it holds that
  \begin{equation}
    \limsup_{m \to \infty}\sup_{P \in \calP}\E_P \left [ E_{\tau \lor m}^\brackm \right ] \leq 1.
  \end{equation}
\end{definition}

Similar to how asymptotic e-values are necessary and sufficient for the construction of asymptotic post-hoc confidence intervals and p-values, asymptotic e-processes are necessary and sufficient for the construction of the procedures described in Definition~\ref{definition:asymptotic post-hoc confseq and p-val}. The following proposition makes the previous statement formal and serves as a sequential analogue of Proposition~\ref{prop:cs-via-evalue}.
\begin{proposition}[On the necessity and sufficiency of asymptotic e-processes]
\label{prop:aphcs-via-eproc}
Suppose that $\calH_{k}^\brackm(\cdot)$ forms an $\aphcs$ for $\theta^\star$. Then there exists a family of functions $(E_{k}^\brackm(\theta))_{\theta\in \Re}$ so that for any $\alpha \in (0, 1)$, $\calH_k^\brackm(\alpha)$ can be written as
\begin{equation}\label{eq:cs-from-eproc}
  \calH_{k}^\brackm(\alpha) = \left\{ \theta \in \Re : E_{k}^\brackm(\theta) < 1/\alpha \right\},  
\end{equation}
and $E_{k}^\brackm(\theta^\star)$ forms a post-hoc asymptotic $\calP$-e-process. Conversely, if $E_{k}^\brackm(\theta^\star)$ forms an asymptotic $\calP$-e-process, then the set defined by \eqref{eq:cs-from-eproc} for any $\alpha \in (0, 1)$ forms an \aphcs.
\end{proposition}

\begin{proof}
    Let $\calH_{k}^\brackm(\cdot)$ be an \aphcs for $\theta^\star$. For any $\theta \in \Re$ and any $k, m \in \NN$ such that $k \geq m$, define
    \begin{equation}
        E_{k}^\brackm(\theta) := \sup \left \{ \frac{1}{\alpha} : \alpha > 0 \text{ and } \theta \notin \calH_k^\brackm(\alpha) \right \} = \sup_{\alpha > 0} \frac{\ind{\theta \notin \calH_k^\brackm(\alpha) }}{\alpha}.
    \end{equation}
    By the property of $\calH_k^\brackm(\alpha)$ forming an \aphcs for $\theta^\star$, it holds that $E_k^\brackm(\theta^\star)$ is an asymptotic e-process. The remainder of the proof proceeds identically to the case-by-case analysis found in the proof of Proposition~\ref{prop:cs-via-evalue}.
\end{proof}

Note that all of the definitions stated in this section are in the distribution-uniform setting. One can of course give pointwise analogues if desired by taking the collection of distributions $\calP$ to be a singleton, but we will continue to work in the uniform setting as the \aphcs we present in the next section naturally satisfies this criterion. 

\subsection{An asymptotic post-hoc mixture confidence sequence}

Here, we observe that \rws asymptotic e-variable of Section~\ref{sec:rws-evariable} can be used to define an asymptotic e-process. The proof can be found in Appendix~\ref{proof:asymptotic e-process}.

\begin{proposition}
\label{prop:asymptotic e-process}
  Fix $\delta > 0$ and suppose that that $(X_n)$ are iid with a $\calP$-uniformly bounded $(2+\delta)$th moment for some $\calP\subseteq\calM_{2+\delta}(\theta)$. 
  For $\gamma\in(0,1/2)$, suppose $T_n = o(n^{\delta\gamma/2} / \log^{2+\delta/2}n)$. 
  Define the upper triangular array of processes $\infseqkm{E_k^\brackm(\theta)}$ by
  \begin{equation}
    E_k^\brackm(\theta) := \int_{\lambda \in \Re}\exp \left \{ \lambda S_k(\theta)  - k(\shat_k^2 + \log^{-1}(m)) \frac{\lambda^2}{2} \right \} \d F(\lambda) \land T_m,
  \end{equation}
  where $F$ is any probability measure supported on a subset of $\Re$.
  Then $\infseqm{\infseqkm{E_k^\brackm(\theta)}}$ forms a $\calP$-uniform asymptotic e-process.
\end{proposition}

We may now combine Propositions~\ref{prop:aphcs-via-eproc} and \ref{prop:asymptotic e-process} to provide the following \aphcs. Just as in Theorem~\ref{thm:rws-aphci}, we take $F$ to be a centered Gaussian and $\gamma = 0.49$. The mixture is then computed in precisely the same way as in the proof of Corollary~\ref{cor:rws-normal-mixture}. 

\begin{theorem}[Post-hoc mixture confidence sequences and anytime p-values]\label{theorem:post-hoc mixture}
    Let $(X_n)$ be a sequence of iid random variables with  a $\calP$-uniformly bounded $(2+\delta)$-th  moment for some $\delta>0$ and $\calP\subseteq\calM_{2+\delta}(\mu)$. 
    Fix $C,\rho>0$ and let $\hat{u}_k^\brackm = k(\shat_k^2 + \log^{-1}m)$. 
    Then 
    \begin{equation}\label{eq:post-hoc-robbins mixture}
    \calH_k^\brackm(\alpha) := \left(\Xbar_k \pm \sqrt{\frac{2\rho\hat{u}_k^\brackm + 2}{k^2\rho}\log\left(\frac{\sqrt{\rho \hat{u}_k^\brackm + 1} }{\alpha}\right)}\right),
    \end{equation}
    forms a $\calP$-uniform \aphcs for $\mu$ with range growth function $f(m) = C\cdot m^{-0.24}$.
\end{theorem}
Observe that \eqref{eq:post-hoc-robbins mixture} takes essentially the same form as some asymptotic confidence sequences of \citet{waudby2024time}. An alternative interpretation of Theorem~\ref{theorem:post-hoc mixture} is that after a minor modification to the variance estimate, the aforementioned asymptotic confidence sequences also enjoy a $\calP$-uniform post-hoc guarantee. From this perspective, while some of the methods introduced in this paper can be viewed as bespoke constructions for the sake of satisfying post-hoc risk control, Theorem~\ref{theorem:post-hoc mixture} serves to ``upgrade'' the guarantees of an existing result in the literature.

\section{Conclusions}
\label{sec:summary}

We have extended the paradigm of asymptotic inference to the post-hoc setting, thus allowing the significance level to be chosen as a function of the data. There are both benefits and drawbacks to our approach. While significantly more flexible than traditional techniques, we obtain guarantees on risk control instead of error probabilities, the common currency of most inferential methods.

Our goal here has not been to argue against traditional inference, but instead to lay the foundations of an alternative statistical methodology that nevertheless retains rigorous frequentist guarantees. Instead of replacing standard p-values and confidence intervals, one could report them alongside their post-hoc variants studied here. 

While our main focus has been on providing concrete constructions of post-hoc confidence intervals, p-values, and hypothesis tests, our results relied on innovations in the theory of asymptotic e-variables. In particular, we provided new asymptotic e-variables, introduced general methods for constructing them, and extended the scope of existing ones to the distribution-uniform setting.    

We also introduced the notion of ``post-hoc asymptotic sequential inference,'' extending the notions of asymptotic e-variables and asymptotic post-hoc CIs and p-values to the sequential setting. In particular, we introduced asymptotic e-processes and post-hoc asymptotic confidence sequences and gave an example of each. Given the richness of post-hoc inference on one hand and sequential inference on the other, we anticipate that there are many interesting questions to be asked at their intersection.

For the practicing scientist interested in employing asymptotic post-hoc CIs, all of the \aphcis that we have discussed have their benefits. The two \aphcis based on the \iwr asymptotic e-variable have the best performance, with $(\calH_n^\iwr)$ instantiated with ex ante anchoring being the tightest \aphci for reasonable ranges of $\alpha$ and anchor $\alpha_0$, whereas $(\calH_n^{\mix,\iwr})$ is tighter in the worst case. The \aphci based on the \rws e-variable is slighter looser, but as we saw in Section~\ref{sec:asymp-eprocesses} it provides a stronger guarantee, and providing flexibility not only in terms of post-hoc $\alpha$ selection, but also in terms of sample size. Given the prevalence of optional continuation, we recommend the latter as the safest \aphci in practice.  

Overall, we hope to have fruitfully expanded the burgeoning area of post-hoc inference to the asymptotic setting, providing practitioners with more flexible inferential tools that allow for data-driven decision-making.

\subsection*{Acknowledgments}
The authors thank Anastasios Angelopoulos and Arun Kumar Kuchibhotla for helpful discussions.
BC and AR acknowledge support from NSF grants IIS-2229881 and DMS-2310718. EG acknowledges support from the Google PhD Fellowship. IW-S acknowledges support from the Miller Institute for Basic Research in Science.

This project was funded in part by the European Union (ERC-2022-SYG-OCEAN-101071601).
Views and opinions expressed are however those of the author(s) only and do not
necessarily reflect those of the European Union or the European Research Council
Executive Agency. Neither the European Union nor the granting authority can be
held responsible for them. 

This publication is part of the Chair ``Markets and Learning,'' supported by Air Liquide, BNP PARIBAS ASSET MANAGEMENT Europe, EDF, Orange and SNCF, sponsors of the Inria Foundation.

This work has also received support from the French government, managed by the National Research Agency, under the France 2030 program with the reference ``PR[AI]RIE-PSAI" (ANR-23-IACL-0008).

{\small 
\bibliography{main}
}

\appendix 

\section{On Alternative Definitions of Post-Hoc CIs}
\label{sec:alternative-defns}

Proposition~\ref{prop:cs-via-evalue} shows that post-hoc CIs which are constructed via thresholding a random variable $E$ are valid if and only if $E$ is an e-variable. One might ask if relaxing the definition of an e-variable results in a relaxed version of post-hoc CIs.

Consider the following definitions of \emph{approximate} e-values and p-values, given originally by \citet[Definition 3.1]{ignatiadis2024asymptotic}. We say that $Q$ is a $(\eps,\delta)$-approximate p-value for $\calP$ if 
\begin{equation}
    P(Q\leq \alpha) \leq (1 + \eps)\alpha + \delta, \quad \text{ for all } \alpha\in(0,1) \text{ and all }P\in\calP. 
\end{equation}
We say that $E$ is a $(\eps,\delta)$-approximate e-value for $\calP$ if 
\begin{equation}
    \E_P[E\wedge m] \leq 1 + \eps + \delta m, \quad \text{ for all }m>0\text{ and }P\in\calP. 
\end{equation}
Following \citet{ignatiadis2024asymptotic}, 
let us say that $(Q_n)$ is a \emph{weakly asymptotic p-value} for $\calP$ if for all $P\in\calP$, $Q_n$ is an $(\eps_n(P),\delta_n(P))$-approximate p-value for some $\eps_n(P) \to 0$, $\delta_n(P)\to0$. 
Define \emph{weakly asymptotic e-variables} analogously. Note that \citet{ignatiadis2024asymptotic} do not use the terminology ``weakly.'' 
We emphasize that these definitions of asymptotic p-variables and e-variables differ from those presented in the main text. Indeed, \citet[Proposition 3.6]{ignatiadis2024asymptotic} shows that weakly asymptotic e-variables are weaker (hence the name) notion than the asymptotic e-variables we use in the main text.

Based on approximate p-variables and e-variables, we might introduce the idea of an $(\eps,\delta)$-approximate post-hoc CI for the functional $\theta$ as a set $\calH(\alpha)$ such that 
\begin{equation}
\risk(\calH;\eps,\delta,P) \leq 1, \text{~ where ~}
    \risk(\calH;\eps,\delta,P) = \E_P\left[\sup_{\alpha>0} \frac{\ind{\theta(P) \notin \calH(\alpha)}}{(1+\eps)\alpha + \delta}\right]. 
\end{equation}
The motivation is the same as in Section~\ref{sec:post-hoc-CIs}: we begin with the approximate guarantee $P(\theta(P)\notin \calH(\alpha)) \leq (1 + \eps)\alpha + \delta$ and convert it to a post-hoc guarantee. We might then define a \emph{weakly asymptotic post-hoc CI} as a sequence of sets $(\calH_n)$ such that $\calH_n$ is an $(\eps_n(P),\delta_n(P))$-approximate post-hoc CI and $\eps_n(P),\delta_n(P)\xrightarrow{n\to\infty} 0$ for all $P\in\calP$.

There are several questions one can ask about whether weakly asymptotic/approximate e-variables result in post-hoc CIs. Unfortunately, we show below that the answer to most of these question is no. Before we begin, it useful to note that the risk of an $(\eps,\delta)$-approximate post-hoc CI can be written as 
\begin{align*}
    \E_P\left[\sup_{\alpha >0} \frac{\ind{E^\theta \geq 1/\alpha}}{(1 + \eps)\alpha + \delta}\right] &= \E_P\left[\sup_{\alpha \geq 1/E^\theta} \frac{1}{(1 + \eps)\alpha + \delta}\right] \\ 
    &= \E_P\left[\frac{1}{(1 + \eps)1/E^\theta + \delta}\right] = \E_P\left[\frac{E^\theta}{1 + \eps + \delta E^\theta}\right],
\end{align*}
when the post-CI is of the form $\calH(\alpha) = \{\theta: E^\theta < 1/\alpha\}$, just as in Proposition~\ref{prop:cs-via-evalue}. 

\paragraph{1. Do approximate e-variables result in approximate post-hoc CIs?}

It appears not. Consider the $(0,\delta)$-approximate e-variable $E$ defined by 
\[P(E=1) = 1-\delta,\quad P(E = u) = \delta,\]
for some large $u>0$. It's easy to verify that $E$ is $(0,\delta)$-approximate (hence $(\eps,\delta)$-approximate for any $\eps>0$): 
\begin{align*}
    \E_P[E\wedge m] &= (1\wedge m)(1 - \delta) + (u\wedge m)\delta \leq 1 - \delta + m\delta \leq 1 + \delta m. 
\end{align*}
However, $E$ results in a risk strictly greater than 1. Indeed, 
\begin{align*}
    \E_P\left[\frac{E}{1 + \eps +\delta E}\right] &= \frac{1-\delta}{1 + \eps + \delta} + \frac{u \delta}{1 + \eps + u\delta }\xrightarrow{u\to\infty} \frac{1-\delta}{1 + \eps + \delta} + 1 > 1,
\end{align*}
which holds as long as $\delta>0$. 

\paragraph{2. Do weakly asymptotic e-variables result in APH-CIs?}

Here by \aphcis we mean Definition~\ref{def:asymp-post-hoc-cs}, i.e., the usual notion investigated in the rest of the paper. As above, the answer appears to be no. 
Consider the sequence of random variables $(E_n)$ with $P(E_n=1) = 1-\delta_n$ and $P(E_n = u_n) = \delta_n$ where $\delta_n\to 0$ and $u_n\gg 1/\delta_n$. Then as above, $\E_P[E_n \wedge m] \leq 1 + \delta_n m$ for each $n$ so $(E_n)$ is a weakly asymptotic e-variable. But 
\begin{align*}
    \E_P\left[\sup_{\alpha>0} \frac{\ind{E_n \geq 1/\alpha}}{\alpha}\right] = \E_{P}[E_n] = 1-\delta_n + u_n\delta_n \to \infty,
\end{align*}
where the first equality follows from the proof of Proposition~\ref{prop:cs-via-evalue}. 

\paragraph{3. Do weakly asymptotic e-variables result in weakly asymptotic post-hoc CIs?}

Again, the answer is no. Using the same weakly asymptotic e-variables as above, we have 
\begin{align*}
    \E_P\left[\frac{E_n}{1 + \eps_n + \delta_n E_n}\right] &= \frac{1 -\delta_n}{1 + \eps_n + \delta_n} + \frac{u_n \delta_n }{1 + \eps_n + u_n\delta_n} \xrightarrow{n\to\infty} 2>1. 
\end{align*}

\section{Additional Results}
\label{app:additional-results}

\subsection{Two alternative asymptotic e-variables}
\label{app:alternative-evars}

Here we present two additional asymptotic e-variables. For parameters $\lambda,\eta$, and a sequence $(\lambda_n)$ define the objects
\begin{align} 
    E_n^\reg(\theta; \lambda, \eta) &:= \exp\left(\frac{\lambda S_n(\theta)}{\sqrt{n}\sigmahat_n + \eta V_n(\theta)} - \frac{\lambda^2}{2}\right). \label{eq:e-reg} \\ 
    E_n^\sn(\theta; \lambda^n) &:= \exp\bigg(\sum_{i\leq n}\lambda_i(X_i-\theta) - \frac{1}{6}\sum_{i\leq n}\lambda_i^2[(X_i -\theta)^2 + 2\sigmahat_{i-1}^2]\bigg),\label{e:sn} 
\end{align}
The superscript \sn refers to ``self-normalized,'' named as such because it is based on, and proved using, the self-normalized process described in  \citet[Table 3]{howard2020time}. Note that it is defined using a sequence of parameters $(\lambda_n)$ instead of a single value and thus depends on weighted versions of $S_n(\theta)$ and $V_n(\theta)$.  We present $E_n^\sn$ out of independent interest in asymptotic e-variables. It will not result in a useful \aphci. 

Meanwhile,  \reg refers to ``regularized,'' because it is helpful to view the term $\eta V_n(\theta)$ as a regularization term. Indeed, one might expect (or hope) that one could replace $V_n(\theta)$ with $\sigmahat_n$ in the \iwr e-variable. In particular, we could consider $T_n := \exp(\frac{\lambda}{\sqrt{n}} S_n(\theta) / \sigmahat_n - \lambda^2/2)$. This is not the case, however: For Gaussian $X_1,\dots,X_n$, $\log(T_n)$ has a t-distribution and thus no MGF. Therefore $\E[T_n]$ does not exist, let alone converge. Adding the term $\eta V_n(\theta)$ to the denominator regularizes the growth of $(T_n)$, ensuring that the sequence is uniformly integrable and an asymptotic e-variable. 

Let us now state under what conditions these objects define asymptotic e-variables. The proof is in Appendix~\ref{proof:E-sn-and-reg}. 

\begin{theorem}
\label{thm:E-sn-and-reg}
Let $(\lambda_n)$ be a sequence of nonnegative scalars which converge to some finite $\lambda\in\Re$ almost surely.  Then: 
\begin{enumerate}
    \item If $\lambda_n$ is $\calF_n$-measurable for each $n$, $\eta_n\to\eta$ almost surely, $\eta_n> 0$ for all $n$, and $|\lambda_n| \lesssim \eta_n$, then $(E_n^\reg(\theta;\lambda_n,\eta_n))$ is an asymptotic e-variable for all distributions in $\calM_2(\theta)$. 
    \item 
    If $(\lambda_n)$ is predictable (i.e., $\lambda_n$ is $\calF_{n-1}$-measurable for each $n$) such that (i) $\sup_n \lambda_n <\infty$, (ii) $\sum_{i\leq n}\lambda_i^2 \xrightarrow{n\to\infty} L<\infty$ a.s.\, and (iii) $\lambda_i^2 / \sum_{j\leq n}\lambda_j^2 \xrightarrow{n\to\infty}0$ a.s. Then 
    $(E_n^\sn(\theta;\lambda_n))$ is an asymptotic e-variable for $\calM_2(\theta)$. 
\end{enumerate}
The above statements remain true if $\sigmahat_n^2$ is replaced by the (biased) sample variance $s_n^2 = n^{-1}\sum_{i\leq n}(X_i -\Xbar_n)^2$, or any other statistic which converges almost surely to the true variance. 
\end{theorem}

It is worth emphasizing that for $E_n^\reg$, when we say that $\eta>0$ is fixed, we mean both that it cannot be data-dependent, but also that it cannot depend on the sample size $n$. This is contrast to nonasymptotic inference, where a fixed parameter can be optimized depending on the sample size (in the Chernoff method, for instance).

We can also extend $(E_n^\reg)$ to the distribution-uniform setting as follows. The proof is in Appendix~\ref{proof:unif-reg-eval}.

\begin{theorem}
\label{thm:unif-reg-eval}
Let $(X_n)$ be iid with mean $\theta$ and $\calP$-uniformly bounded skew. 
Let $(\lambda_n)$ be a bounded, deterministic sequence converging to $\lambda$.  Suppose $0<\inf_{P\in\calP} \sigma_P$ and let $(\eta_n)$ be a deterministic sequence such that $\lambda_n\lesssim\eta_n$ and $\eta_n>0$. Then $(E_n^\reg(\theta;\lambda_n,\eta))$ is a $\calP$-uniform asymptotic e-variable. 
\end{theorem}

\subsection{An APH-CI based on \texorpdfstring{$E_n^\reg$}{\reg}}
\label{app:reg-aphci}

Next we state the \aphci corresponding to the asymptotic e-variable $(E_n^\reg)$. One can perform the arithmetic to write down an \aphci based on $(E_n^\sn)$, but it is the union of several disjoint intervals. Moreover, while one of these intervals contains the sample mean $\Xbar_n$ and behaves similarly to $\calH_n^\iwr$ and $\calH_n^\reg$  in terms of width, the outermost intervals behave oddly and go off to $\pm \infty$, making the resulting CI vacuous. 

The proof of the following result is in Appendix~\ref{proof:reg-aphci}.

\begin{theorem}
\label{thm:reg-aphci}
Let $X_1,\dots,X_n$ be iid with finite variance and  mean $\mu$. 
For any $\lambda>0$ and $\alpha\in(0,1)$, set 
\begin{equation*}
    A_{n,\alpha}(\lambda) = \frac{\log(2/\alpha) + \lambda^2/2}{\lambda\sqrt{n}}.  
\end{equation*}
For any fixed $\lambda,\eta>0$ both independent of $\alpha$, $(\calH_n^\reg(\alpha;\lambda,\eta))$ is an \aphci for $\E X_1$,  where 
    \begin{equation}
    \label{eq:reg-aphci}
        \calH_n^\reg(\alpha;\lambda,\eta) := (\Xbar_n \pm W_n^\reg(\alpha;\lambda,\eta)), \quad W_n^\reg(\alpha;\lambda,\eta) := \frac{A_{n,\alpha}(\lambda) \sigmahat_n( 1  + \sqrt{B_{\alpha,\eta}(\lambda)})}{ 1 - A_{n,\alpha}^2(\lambda)\eta^2},
    \end{equation}
    and $B_{\alpha,\eta}(\lambda) =  1  - ( 1 - A_{n,\alpha}^2(\lambda)\eta^2)(1 - \eta^2 \frac{n-1}{n})$. 
\end{theorem}

While both $\lambda$ and $\eta$ are allowed to vary with $n$ in Theorem~\ref{thm:E-sn-and-reg} for $E_n^\reg$, they have been fixed in Theorem~\ref{thm:reg-aphci}. This is because we would like both terms in $A_{n,\alpha}(\lambda)$ to decay at the same rate, which forces $\lambda$ to be constant. This in turn forces $\eta$ to be constant since Theorem~\ref{thm:E-sn-and-reg} requires $\lambda_n\lesssim \eta_n$.

Notice that by a single application of the SLLN and the continuous mapping theorem, 
\[\sqrt{n}W_n^\reg(\alpha;\lambda,\eta) \xrightarrow[n\to\infty]{a.s.} \sigma g(\lambda,\alpha)( 1 + \eta),\]
which is strictly larger than the limiting width of $(\calH_n^\iwr)$ given in Section~\ref{sec:iwr-evariable} since $\eta$ must be positive. (Recall that $g(\lambda,\alpha) = \log(2/\alpha)/\lambda + \lambda/2$.) The performance of $(\calH_n^\reg)$ is explored experimentally in Appendix~\ref{app:sims}, as is the effect of $\eta$ on the width of the bounds. 

In practice, we recommend fixing $\eta = 10^{-3}$. While $\eta$ can, theoretically, be taken as small as desired, values smaller than $10^{-3}$ have no noticeable effect on the width.  Further, Proposition \ref{prop:exact-type1-error-reg} in the following subsection demonstrates that small values of $\eta$ increase the Type-I error of $\calH_n^\reg$. Thus, values of $\eta$ that are too small risk inflating error without any corresponding improvement in the width of the CI. 

\subsection{Exact asymptotic type-I error of the APH-CIs}
\label{app:type-I-error}

The \aphcis developed in the main text are designed to bound the post-hoc risk below 1. A natural secondary question is to evaluate their classical asymptotic Type-I error, i.e., the probability of miscoverage at a fixed (data-independent) significance level $\alpha$.

Let $Z \sim \calN(0,1)$ with cdf denoted by $\Phi$. By the CLT and the LLN, the self-normalized ratio converges in distribution to a standard normal under any distribution $P\in\calM_2(\theta^\star)$:
\[
\frac{S_n(\theta^\star)}{V_n(\theta^\star)} = \frac{S_n(\theta^\star)/\sqrt{n}}{V_n(\theta^\star)/\sqrt{n}} \xrightarrow{d} Z.
\]
We begin by studying the \aphci defined by the \iwr e-variable. 

\begin{proposition}
\label{prop:exact-type1-error-iwr}
    Let $\calH_n^\iwr(\alpha;\lambda) =  \{ \theta: E_n^\iwr(\theta; \lambda) <2/\alpha \text{ and }E_n^\iwr(\theta; -\lambda) < 2/\alpha\}$ be the APH-CI induced by the \iwr asymptotic e-value for any $\lambda>0$. Then:
\begin{equation}
\label{eq:asymp-type1error-iwr}
\lim_{n\to\infty} P\big(\theta^\star \notin \mathcal{H}_n^\iwr(\alpha;\lambda)\big) = 2\left(1 - \Phi\left( \frac{\log(2/\alpha)}{\lambda} + \frac{\lambda}{2} \right)\right).
\end{equation}
\end{proposition}
\begin{proof}
    Taking logarithms, the miscoverage event is equivalent to:
    \[
    \lambda \frac{S_n(\theta^\star)}{V_n(\theta^\star)} -\frac{\lambda^2}{2} \ge \log(2/\alpha) \quad\text{or}\quad -\lambda \frac{S_n(\theta^\star)}{V_n(\theta^\star)} -\frac{\lambda^2}{2} \ge \log(2/\alpha),
    \]
    which is equivalent to 
    \[
    \left| \frac{S_n(\theta^\star)}{V_n(\theta^\star)} \right|\ge \frac{\log(2/\alpha)}{\lambda} +\frac{\lambda}{2}.
    \]
    By the continuous mapping theorem, substituting the limit $S_n(\theta^\star)/V_n(\theta^\star)\xrightarrow{d} Z$ yields:
    \[\lim_{n\to\infty} P\big(\theta^\star \notin \mathcal{H}_n^\iwr(\alpha;\lambda)\big) = 2\left(1 - \Phi\left( \frac{\log(2/\alpha)}{\lambda} + \frac{\lambda}{2} \right)\right),
    \]
    completing the argument. 
\end{proof}
Note that if we set $\lambda_\alpha=\sqrt{2\log(2/\alpha)}$, the asymptotic Type-I error evaluates to $2(1-\Phi(\sqrt{2\log(2/\alpha)}))$ which is smaller than $\alpha$ for any $\alpha\in(0,1)$, reflecting the price paid for post-hoc validity.

Next we study the \aphci defined by the mixture of the \iwr e-variable over a truncated Gaussian; see Section~\ref{sec:aphcis-via-mom}. 

\begin{proposition}
\label{prop:exact-type1-error-mix-iwr}
    The asymptotic Type-I error induced by $\calH_n^{\mix,\iwr}$ is given by:
\begin{equation}
\label{eq:asymp-type1error-mix-iwr}
\lim_{n\to\infty} P\big(\theta^\star \notin \mathcal{H}_n^{\mix,\iwr}(\alpha;R,\kappa)\big) = 2(1 - \Phi(y_\alpha^\star)),
\end{equation}
where $y_\alpha^\star$ is the unique positive root of the implicit equation $I_R(y_\alpha^\star,\frac{\kappa^2+1}{2\kappa^2}) = Z_{R,\kappa}/\alpha$.
\end{proposition}
\begin{proof}
    By definition, 
    \[
    E_n^{\mix,\iwr}(\theta^\star; R,\kappa) = Z_{R,\kappa}^{-1} I_R\left(\frac{S_n(\theta^\star)}{V_n(\theta^\star)}, \frac{\kappa^2+1}{2\kappa^2}\right).
    \]
    Because the function $y\mapsto I_R(y,u)$ is continuous, we may apply the continuous mapping theorem to obtain:
    \[
    E_n^{\mix,\iwr}(\theta^\star; R,\kappa) \xrightarrow{d} Z_{R,\kappa}^{-1} I_R\left(Z, \frac{\kappa^2+1}{2\kappa^2}\right).
    \]
    The asymptotic Type-I error is therefore given by:
    \[
    \lim_{n\to\infty} P\big(\theta^\star \notin \mathcal{H}_n^{\mix,\iwr}(\alpha;R,\kappa)\big) = P\left(I_R\left(Z, \frac{\kappa^2+1}{2\kappa^2}\right) \ge Z_{R,\kappa}/\alpha \right).
    \]
    The function $I_R(y,u)$ is symmetric about $y=0$ and strictly monotonically increasing in $|y|$ so the threshold inequality is satisfied if and only if $|Z|\ge y_\alpha^\star$, where $y_\alpha^\star$ is the unique positive root of the implicit equation $I_R(y_\alpha^\star,\frac{\kappa^2+1}{2\kappa^2}) = Z_{R,\kappa}/\alpha$. Thus, the asymptotic Type-I error is:
    \[
    \lim_{n\to\infty} P\big(\theta^\star \notin \mathcal{H}_n^{\mix,\iwr}(\alpha;R,\kappa)\big) = 2(1 - \Phi(y_\alpha^\star)).\qedhere
    \]
\end{proof}

Finally, we study the \aphci given by the \rws asymptotic e-variable. 

\begin{proposition}
\label{prop:exact-type1-error-rws}
    Under the setting of Theorem \ref{thm:rws-aphci}, the asymptotic Type-I error induced by $\calH_n^{\rws}$ vanishes:
\begin{equation}
\label{eq:asymp-type1error-rws}
\lim_{n\to\infty} P\big(\theta^\star \notin \mathcal{H}_n^{\rws}(\alpha;\rho)\big) = 0.
\end{equation}
\end{proposition}
\begin{proof}
    Note that for sufficiently large $n$, the threshold condition $C n^{\delta\cdot 0.24}\geq 1/\alpha$ is satisfied and $\calH_n^\rws(\alpha;\rho) = (\Xbar_n \pm W_n^\rws(\rho))$. The miscoverage event is therefore equivalent to:
    \[
    \left|\Xbar_n - \theta^\star\right| >W_n^\rws(\rho).
    \]
    After scaling by $\sqrt{n}/\sigma$, the miscoverage event can be written equivalently as:
    \[
    \frac{\sqrt{n}\left|\Xbar_n - \theta^\star\right| }{\sigma}> \frac{\sqrt{n}}{\sigma}W_n^\rws(\rho).
    \]
    By the CLT and the continuous mapping theorem, the left-hand side converges in distribution to $|Z|$. However, as noted in the main text, the entire scaled threshold diverges:
    \[
    \frac{\sqrt{n}}{\sigma}W_n^\rws(\rho) \xrightarrow[n\to\infty]{a.s.} \infty.
    \]
    Let $B>0$ be arbitary. For sufficiently large $n$, we have $\frac{\sqrt{n}}{\sigma}W_n^\rws(\rho)\ge B$ almost surely, which implies:
    \[
    P\left(\frac{\sqrt{n}\left|\Xbar_n - \theta^\star\right| }{\sigma}> \frac{\sqrt{n}}{\sigma}W_n^\rws(\rho)\right) \le P\left(\frac{\sqrt{n}\left|\Xbar_n - \theta^\star\right| }{\sigma}> B\right).
    \]
    The right-hand side converges to $2(1-\Phi(B))$, so that:
    \[
    \limsup_{n\to\infty} P\left(\frac{\sqrt{n}\left|\Xbar_n - \theta^\star\right| }{\sigma}> \frac{\sqrt{n}}{\sigma}W_n^\rws(\rho)\right) \le 2(1-\Phi(B)).
    \]
    Since $B>0$ is arbitrary, we conclude:
    \[
    \lim_{n\to\infty} P\big(\theta^\star \notin \mathcal{H}_n^{\rws}(\alpha;\rho)\big) = 0.\qedhere
    \]
\end{proof}
This exact zero error highlights the fundamental structural difference between an \aphci and an \aphcs. Because the confidence sequence must grow by a $\sqrt{\log(n)}$ factor to survive an infinite horizon of stopping times, its width shrinks only at a rate $\sqrt{\log(n)/n}$, which eventually far exceeds the actual $O_P(1/\sqrt{n})$ sampling error, driving the Type-I error to zero.

Additionally, we study the asymptotic Type-I error of the \reg \aphci introduced in Subsection \ref{app:alternative-evars}. The proof is essentially similar to that of Proposition \ref{prop:exact-type1-error-iwr}, noting that
\[
\frac{ S_n(\theta^\star)}{\sqrt{n}\sigmahat_n + \eta V_n(\theta^\star)} \xrightarrow{d} \frac{1}{1+\eta}Z
\]
which follows by Slutsky's theorem, so we omit it.

\begin{proposition}
\label{prop:exact-type1-error-reg}
    Let $\calH_n^\reg(\alpha;\lambda,\eta) =  \{ \theta: E_n^\reg(\theta; \lambda,\eta) <2/\alpha \text{ and }E_n^\reg(\theta; -\lambda,\eta) < 2/\alpha\}$ be the APH-CI induced by the \reg asymptotic e-value for any $\lambda,\eta>0$. Then:
\begin{equation}
\label{eq:asymp-type1error-reg}
\lim_{n\to\infty} P\big(\theta^\star \notin \mathcal{H}_n^\reg(\alpha;\lambda,\eta)\big) = 2\left(1 - \Phi\left(\left(1+\eta\right)\left( \frac{\log(2/\alpha)}{\lambda} + \frac{\lambda}{2} \right)\right)\right).
\end{equation}
\end{proposition}

\begin{figure}[htbp]
    \centering
    \includegraphics[width=0.95\textwidth]{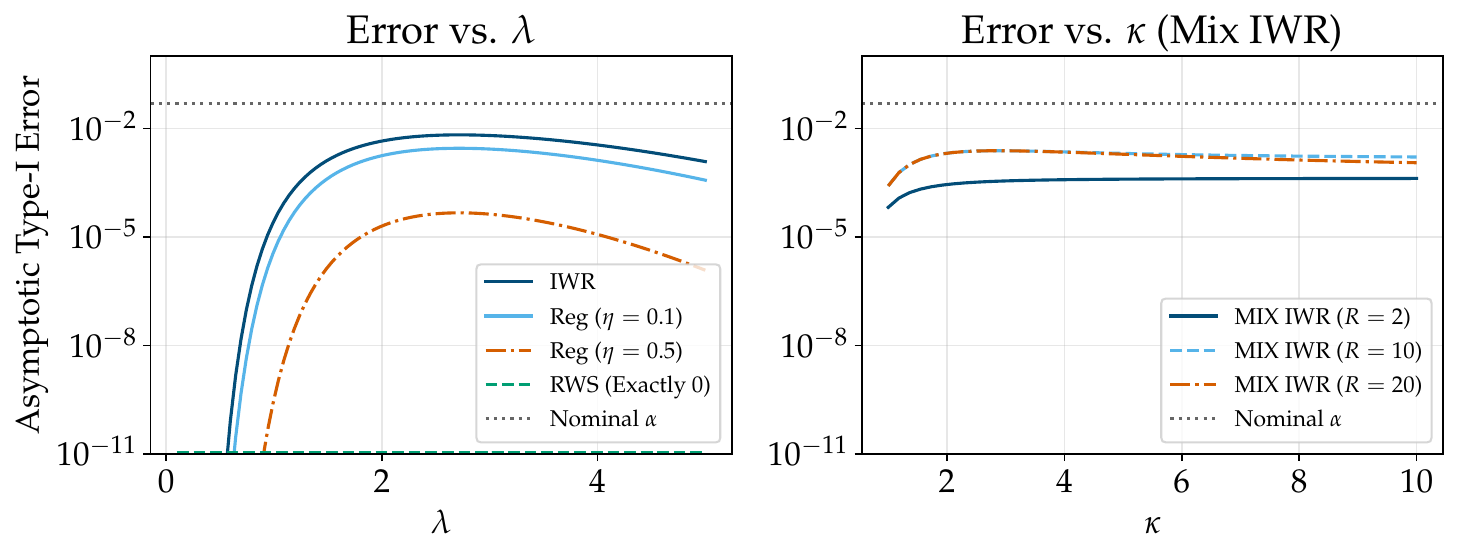}
    \caption{Asymptotic Type-I error of the proposed \aphcis as a function of their respective tuning parameters. The nominal significance level is fixed at $\alpha=0.05$ (dotted black line). Since the true asymptotic errors are much smaller than the nominal level, reflecting the conservative cost of ensuring post-hoc validity, the y-axis is shown on a logarithmic scale. Left: Error for \iwr and \reg ($\eta \in \{0.1,0.5\}$) across values of $\lambda$. The \rws error is plotted at the bottom of the log-scale as its theoretical asymptotic type-I error is exactly zero. Right: Error for \mix,\iwr across $\kappa$ for varying truncation radii ($R\in\{2,10,20\}$). Note that for large truncation bounds $(R\ge10)$, the distribution closely matches an untruncated Gaussian, making the $R=10$ and $R=20$ curves nearly identical.}
    \label{fig:asymptotic_type1_errors}
\end{figure}

\subsection{APH-CIs via compound asymptotic e-variables}
\label{app:aphcis-compound-evals}

Compound e-variables arise naturally in multiple testing; for instance in the e-BH procedure~\citep{wang2022false} and derandomizing knockoffs~\citep{ren2024derandomised}; in fact~\cite{ignatiadis2024asymptotic} show that compound e-values underlie \emph{all} FDR controlling procedures. They refer to a collection of $K$ nonnegative random variables $E^j$, $j=1,\dots,K$, as \emph{compound e-variables} for $\calP$ if $\sup_{P\in\calP}\sum_j \E_P[E^j]\leq K$. We thus refer to $K$ \emph{collections} of random variables $(E_n^j)$, $j=1,\dots,K$ as \emph{asymptotic} compound e-variables for $\calP$ if 
\begin{equation}
    \sup_{P\in\calP} \limsup_{n\to\infty} \sum_{j\leq K}\E_P[E_n^j] \leq K. 
\end{equation}
Likewise, $(E_n^j)$ are $\calP$-uniform asymptotic compound e-variables for $\calP$ if 
\begin{equation}
    \limsup_{n\to\infty} \sup_{P\in\calP}  \sum_{j\leq K}\E_P[E_n^j] \leq K. 
\end{equation}
Asymptotic compound e-variables yield \aphcis as follows. 

\begin{proposition}
\label{prop:cs-via-compound-eval}
    For $j=1,\dots,K$ and each $\theta\in\Theta$,  
    let $(E_n^j(\theta))$ be a collection of random variables. For all $\alpha\in(0,1)$, let $\calH_n(\alpha) = \{ \theta: E_n^j(\theta) < K/\alpha\text{ for all }j\in[K]\}$. Then $(\calH_n)$ is a (uniform) \aphci for $\theta^\star$ if $(E_n^j(\theta^\star))$ are (uniform) asymptotic compound e-variables. 
\end{proposition}
\begin{proof}
The argument is similar to Proposition~\ref{prop:cs-via-evalue}. For each $n$ and $\alpha$, if $\cup_j \{E_n^j(\theta^\star)\geq K/\alpha\}$ occurs, then $1/\alpha \leq E_n^j(\theta^\star)/K$ for some $j$, hence 
    \[\frac{\ind{\cup_j \{E_n^j(\theta^\star)\geq K/\alpha\}}}{\alpha} \leq \max_j \frac{E_n^j(\theta^\star)}{K} \leq \frac{1}{K}\sum_j E_n^j(\theta^\star).\]
    Taking expectations and limits,  
    \[\limsup_{n\to\infty}\E_P\left[\sup_\alpha \frac{\ind{\theta^\star\notin \calH_n(\alpha)}}{\alpha}\right]\leq \limsup_{n\to\infty}\frac{1}{K}\sum_j \E_P[E_n^j(\theta^\star)]\leq 1,\]
    completing the argument. In the distribution-uniform setting we simply take a supremum over distribution before taking limits. 
\end{proof}    

In this work we typically apply this result when $(E^1_n)$ and $(E^2_n)$ are simply two e-variables (e.g., in Theorem~\ref{thm:iwr-aphci}). 

\subsection{Distribution-uniform convergence in expectation}
\label{app:du-convergence-in-exp}

Here we present sufficient conditions under which a sequence of nonnegative random variables constitute a distribution-uniform asymptotic e-variable. In the pointwise setting, to construct asymptotic e-variables we rely on \citet[Theorem 25.12]{billingsley1995proba} which gives conditions under which a sequence of random variables converges in expectation. In particular, he shows that uniform integrability together with convergence in distribution implies convergence in expectation. Since convergence in distribution of a sequence of random variables $(X_n)$ to $X$ is equivalent to $\E[\min(X_n,K)] \to \E[\min(X,K)]$ for all $K>0$, the following result can be seen as an extension of Billingsley's result to the distribution-uniform setting.

\begin{proposition}
\label{prop:uniform-billingsley}
Let $g:\Re\to\Re_{\geq 0}$. Let $\calP$ be a set of distributions and for each $P\in\calP$ let $(T_{n,P})_{n\geq 1}$ be a sequence of statistics depending on $P$.  Let $Z$ be a random variable with law $R$ such that $\E_R[g(Z)]<\infty$.  Suppose further that 
\begin{enumerate}
    \item $\lim_{K\to\infty}\limsup_{n\to\infty} \sup_P \E_P[g(T_{n,P}) \ind{g(T_{n,P})\geq K}] =0$, and 
    \item for all $K>0$, 
    \[\limsup_{n\to\infty} \sup_{P\in\calP} \E_P[g(T_{n,P}) \wedge K] \leq \E_R[g(Z)\wedge K].\]
\end{enumerate}
Then 
\begin{equation}
    \limsup_{n\to\infty}\sup_{P\in\calP} \E_P[g(T_{n,P})]\leq \E_R[g(Z)]. 
\end{equation}
Moreover, if condition (2) is replaced by for all $K>0$, $\limsup_{n\to\infty} \sup_{P\in\calP} |\E_P[g(T_{n,P}) \wedge K] - \E_R[g(Z)\wedge K]|= 0$, then we can conclude that $\limsup_{n\to\infty}\sup_{P\in\calP} |\E_P[g(T_{n,P})]- \E_R[g(Z)]|=0$. 
The same statements hold if $\limsup_n$ is replaced by $\sup_n$ everywhere. 
\end{proposition}
\begin{proof}
    
    Throughout the proof we write $T_n$ instead of $T_{n,P}$; $P$ should be understood from context. For any $x$ and $K>0$, note that we can write $g(x) = g(x) \wedge K + 0\wedge (g(x) - K)\leq g(x) \wedge K + g(x)\ind{g(x)\geq K}$. Therefore, 
    \begin{align*}
        \E_P[g(T_n)] \leq \E_P[g(T_n) \wedge K] + \E_R[g(T_n) \ind{g(T_n) \geq K}].
    \end{align*}
    Taking the supremum over all $P$ and then the limit as $n\to\infty$, we have 
    \begin{align*}
          \limsup_n \sup_P \E_P[g(T_n)] &\leq \limsup_n \sup_P \left(\E_P[g(T_n) \wedge K] + \E_R[g(T_n) \ind{g(T_n) \geq K}]\right) \\ 
          &\leq \E_R[g(Z)\wedge K]  + \limsup_n \sup_P\E_R[g(T_n) \ind{g(T_n) \geq K}]
    \end{align*}
    using assumption (1). Letting $K\to\infty$ the second term on the right hand side above vanishes by assumption (2) and the first converges to $\E_R[g(Z)]$ by the monotone convergence theorem. If we replace assumption (2) with $\limsup_n\sup_P |\E_P[g(T_n) \wedge K] - \E_R[g(Z)]|= 0$ for each $K>0$, a similar proof applies. By the triangle inequality, write 
    \begin{align*}
        |\E_P[g(T_n) - \E_R[g(Z)]] &\leq |\E_P[g(T_n)\wedge K] - \E_R[g(Z) \wedge K]|\\ 
        &\qquad + \E_P[g(T_n)\ind{g(T_n)\geq K}] + \E_R[g(Z)\ind{g(Z) \geq K}].
    \end{align*}
    Hence, 
    \begin{align*}
         &\limsup_n \sup_P|\E_P[g(T_n) - \E_R[g(Z)]] \\
         &\leq 
        \limsup_n\sup_P\E_P[g(T_n)\ind{g(T_n)\geq K}] + \E_R[g(Z)\ind{g(Z) \geq K}].
    \end{align*}
    Letting $K\to\infty$, both terms on the right of the inequality vanish---the first by assumption and the second by integrability of $g(Z)$. This completes the argument. 
\end{proof}

\subsection{Generalized sufficient conditions for the mixture process}
\label{app:generalized-mixture}

Here we give a mild generalization of Proposition~\ref{prop:mixture-is-valid}, by weakening the sufficient conditions for a mixture process to be an asymptotic e-variable. We comment more on the scope of the generalization after stating the result. 

\begin{proposition}
\label{prop:mixture-generalized}
    Let $\Lambda\subset\Re$ and let $\pi$ be any (data-free) probability distribution on $\Lambda$. Suppose that for each $\lambda\in\Lambda$, the sequence $(E_n(\lambda))$ is an asymptotic e-variable for some set of distributions $\calP$. Assume that for every $P \in \calP$, there exists a measurable function $h_P : \Lambda \to [0,\infty)$ such that:
    \begin{itemize}
        \item $\sup_{n\ge1}\E_P[E_n(\lambda)] \le h_P(\lambda)$ for $\pi$-almost all $\lambda$ (domination),
        \item $\int_\Lambda h_P(\lambda) \d\pi(\lambda) < \infty$ (integrability). 
    \end{itemize}
    Then, the mixture process defined by
    \begin{equation}
        E_n(\pi) := \int_\Lambda E_n(\lambda) \d\pi(\lambda),
    \end{equation}
     is an asymptotic e-variable for $\calP$. 
     Further, if $\sup_P h_P$ is measurable and integrable, i.e., $\int_\Lambda h_P(\lambda)\d\pi(\lambda)<\infty$, then the mixture $(E_n(\pi))$ is a $\calP$-uniform asymptotic e-variable. 
\end{proposition}
\begin{proof}
      Fix an arbitrary distribution $P \in \calP$. By Tonelli's Theorem (since $E_n(\lambda)$ is non-negative), we can exchange the expectation and the mixture integral:
    \[
    \E_P[E_n(\pi)] = \E_P\left[\int_\Lambda E_n(\lambda)\d\pi(\lambda)\right] = \int_\Lambda \E_P[E_n(\lambda)] \d\pi(\lambda).
    \]
    We can now apply Reverse Fatou's Lemma to the sequence of functions $\lambda\mapsto\E_P[E_n(\lambda)]$, which states that if such a sequence is dominated by an integrable function, then the limit superior of the mixture is less than or equal to the mixture of their limit superiors:
    \[
    \limsup_{n\to\infty} \int_\Lambda \E_P[E_n(\lambda)]\d\pi(\lambda) \le \int_\Lambda \limsup_{n\to\infty}\E_P[E_n(\lambda)]\d\pi(\lambda).
    \]
    Since $\left(E_n(\lambda)\right)$ is an asymptotic e-variable for each $\lambda$, we have $\limsup_{n\to\infty}\E_P[E_n(\lambda)]\le 1$ for all $\lambda$. Therefore:
    \[
    \limsup_{n\to\infty} \int_\Lambda \E_P[E_n(\lambda)]\d\pi(\lambda) \le \int_\Lambda 1\cdot \d\pi(\lambda)=1.
    \]
    Since $P\in\calP$ was arbitrary, the result holds for the entire class $\calP$.

    In the distribution-uniform case, the proof is largely unchanged after noting that 
    \[\sup_{P}\E_P[E_n(\pi)] = \sup_{P} \int_\Lambda \E_P[E_n(\lambda)]\d\pi(\lambda) \leq \int_\Lambda \sup_{P} \E_P[E_n(\lambda)]\d\pi(\lambda).\]
    Since $g(\lambda) = \sup_P h(\lambda)$ is measurable by assumption, we may apply Reverse Fatou's lemma to $\lambda\mapsto \sup_P \E_P[E_n(\lambda)] \leq g(\lambda)$, which gives 
    \[\limsup_n \int_\Lambda \sup_P \E_P[E_n(\lambda)] \d\pi(\lambda) \leq \int_\Lambda \limsup_n \sup_P \E_P[E_n(\lambda)] \d\pi(\lambda)\leq 1,\]
    by definition of a $\calP$-uniform asymptotic e-variable. This completes the proof.    
\end{proof}

The two conditions of domination and integrability in Proposition~\ref{prop:mixture-is-valid} are strictly less prohibitive than requiring that $(\sup_{\lambda\in\Lambda}E_n(\lambda))$ be uniformly integrable (which is what is required in Proposition~\ref{prop:mixture-is-valid}).  Indeed, consider $X\sim N(0,1)$, $\Lambda=\mathbb{R}$, $\pi = N(0,1)$ and let $$E(\lambda)=\exp\left(\lambda X - \frac{\lambda^2}{2}\right).$$ Then $\sup_\lambda E(\lambda) =\exp\left(X^2/2\right)$ and $\E\left[\exp\left(\frac{X^2}{2}\right)\right]=\infty$, so it is not uniformly integrable. However, $\E[E(\lambda)]=1$ so we can simply choose $h(\lambda)=1$ in Proposition~\ref{prop:mixture-is-valid} and conclude that $E_n(\pi)$ is an e-variable.

\subsection{Further details on the truncated Gaussian mixture}
\label{app:truncated-gaussian}

First let us present the asymptotic e-variable that one obtains from mixing the \iwr e-variable according to Corollary~\ref{cor:compact_mix}. We do the same for the \reg e-variable presented in Appendix~\ref{app:alternative-evars}. 

\begin{proposition}
\label{prop:mixture-evals}
    For any fixed $R,\kappa,\eta>0$, define $Z_{R,\kappa} := \kappa\sqrt{2\pi} (\Phi(R/\kappa) - \Phi(-R/\kappa))$. Then $(E_n^{\mix,\iwr}(\theta;R,\kappa))$ and $(E_n^{\mix, \reg}(\theta; R,\kappa,\eta))$ are asymptotic e-variables for all distributions in $\calM_2(\theta)$, where 
    \begin{equation}
        E_n^{\mix,\iwr}(\theta; R,\kappa) := Z_{R,\kappa}^{-1} I_R\left(\frac{S_n(\theta)}{V_n(\theta)}, \frac{\kappa^2+1}{2\kappa^2}\right),
    \end{equation}
    and 
    \begin{equation}
        E_n^{\mix,\reg}(\theta; R,\kappa,\eta) := Z_{R,\kappa}^{-1} I_R\left(\frac{S_n(\theta)}{\sqrt{n}\sigmahat_n + \eta V_n(\theta)}, \frac{\kappa^2+1}{2\kappa^2}\right).
    \end{equation}
\end{proposition}
\begin{proof}
    Both $E_n^{\iwr}$ and $E_n^\reg$ can be written as $E_n(\lambda) = \exp(\lambda M - \lambda^2/2)$ for some $M$. Let $\rho$ be the density of a truncated Gaussian on $[-R,R]$ with mean 0 and variance $\kappa^2$, i.e., 
\begin{equation}
    \rho(\xi) = Z_{R,\kappa}^{-1}\exp\left(\frac{-\xi^2}{2\kappa^2}\right) \text{~ where ~} Z_{R,\kappa} := \sqrt{2\pi\kappa^2 }\left( \Phi\left(\frac{R}{\kappa}\right) - \Phi\left(\frac{-R}{\kappa}\right)\right),
\end{equation}
where $\Phi$ is the cdf of the standard normal. 
Letting $u = 1/2 + 1/(2\kappa^2)$, observe that 
\begin{align*}
E_n(\rho) = \int_{-R}^R E_n(\lambda) \rho(\lambda) \d\lambda &= 
    Z_{R,\kappa}^{-1}\int_{-R}^R \exp\left(\lambda M - \lambda^2 u \right)\d\lambda. 
\end{align*}
Completing the square, write 
\[\exp(\lambda M - \lambda^2 u) = \exp\left(-u\left(\lambda - \frac{M}{2u}\right)^2\right)\exp\left(\frac{M^2}{4u}\right),\]
and let $t = \sqrt{2u}(\lambda - M/(2u))$ so that $\d t = \sqrt{2u}\d \lambda$. Set $t_2 = \sqrt{2u}(R - M/(2u))$ and $t_1 = -\sqrt{2u}(R + M/(2u))$. Then, 
\begin{align*}
    \int_{-R}^R \exp(\lambda M - \lambda^2 u)\d\lambda &= \frac{1}{\sqrt{2u}}\exp\left(\frac{M^2}{4u}\right)\int_{t_1}^{t_2} \exp(-t^2/2)\d t \\
    &= \frac{1}{\sqrt{2u}}\exp\left(\frac{M^2}{4u}\right)\left(\int_{0}^{t_2} \exp(-t^2/2)\d t - \int_0^{t_1} \exp(-t^2/2) \d t\right) \\
    &= \sqrt{\frac{\pi}{u}}\frac{1}{\sqrt{2u}}\exp\left(\frac{M^2}{4u}\right)\left(\Phi(t_2) - \Phi(t_1)\right) \\ 
    &= I_R(M,u),
\end{align*}
using that $\Phi(z) = \frac{1}{\sqrt{2\pi}}\int_0^z \exp(-t^2/2)\d t$. Therefore, $E_n(\rho) = Z_{R,\kappa}^{-1} I_R(M,u)$. Substituting $M = S_n(\theta)/V_n(\theta)$ for $E_n^\iwr$ and $M = S_n(\theta) / (\sqrt{n}\sigmahat_n + \eta V_n(\theta))$ for $E_n^\reg$ completes the proof. 
\end{proof}

Inverting the asymptotic e-variable $(E_n^{\mix,\iwr})$ above gives rise to Theorem~\ref{thm:iwr-mixture-aphci}. 
Let us state the equivalent result for $(E_n^{\mix,\reg})$.

\begin{theorem}
    \label{thm:reg-mixture-aphci}
    Fix any $R,\kappa,\eta>0$ and let $Z_{R,\kappa} = \kappa\sqrt{2\pi} (\Phi(R/\kappa) - \Phi(-R/\kappa))$.   $(\calH_n^{\mix,\reg}(R,\kappa,\eta))$ is an \aphci where $\calH_n^{\mix,\reg}(R,\kappa,\eta) = (\Xbar_n \pm W_n^{\mix,\reg}(R,\kappa,\eta))$ with 
    \begin{equation}
    \label{eq:mix-reg-aphci}
        \quad W_n^{\mix,\reg}(R,\kappa,\eta) = \frac{\sqrt{n}\cdot s_n y^\star_\alpha  ( 1 + \eta D_{\alpha,\eta})}{n - (\eta y_\alpha^\star)^2},
    \end{equation}
    where $D_{\alpha,\eta} = \sqrt{1 + (y_\alpha^\star)^2(1-\eta^2)/n}$ and 
    $y_\alpha^\star$ solves $I_R(y_\alpha^\star,(\kappa^2 + 1)/(2\kappa^2)) = Z_{R,\kappa}/\alpha$. If the denominator in \eqref{eq:mix-reg-aphci} is non-positive, we take the CI to be all of $\Re$.
\end{theorem}
\begin{proof}
    Now we turn to $E_n^{\mix,\reg}$. We use a version of $E_n^\reg$ with $\shat_n$ replacing $\sigmahat_n$ in the denominator, which is valid by Theorem~\ref{thm:reg-aphci}. 
Let $w_n= \sqrt{n}s_n$. 
The equation $|S_n(\theta)/V_n(\theta)|<y^\star$ thus becomes
\[n|\Delta| - y^\star w_n < y^\star \eta \sqrt{w_n^2 + n\Delta^2}. \]
If $n|\Delta| < y^\star w_n$ then $\theta \in  (\Xbar_n \pm y^\star w_n / n)$, which is thus part of our solution. 
If $n\Delta \geq y^\star w_n$ then both terms in the above display are nonnegative so we can square both sides which eventually yields the inequality 
\[(n^2 - (y^\star\eta)^2n)\Delta^2 - 2ny^\star w_n \Delta + (y^\star w_n)^2 - (y^\star\eta)^2 w_n^2<0.\]
After a significant amount of arithmetic, the roots of this polynomial are given by 
\[\Delta_{\pm} = \frac{y^\star w_n \pm y^\star w_n \eta \sqrt{1 + (y^\star)^2(1-\eta^2)/n}}{n - (y^\star\eta)^2} = \frac{y^\star w_n(1 \pm \eta D)}{n - (y^\star\eta)^2},\]
where $D = D_{\alpha,\eta}$. Conservatively taking the positive root as it defines a larger interval, we obtain 
\[\theta \in \left(\Xbar_n \pm \frac{y^\star w_n(1 \pm \eta D)}{n - (y^\star\eta)^2}\right),\]
which is a superset of $(\Xbar_n \pm y^\star w_n/n)$.  
\end{proof}

The performance of $(\calH_n^{\mix,\reg})$ compared to $(\calH_n^{\mix,\iwr})$ is explored in Figure~\ref{fig:iwr_vs_reg} in Appendix~\ref{app:sims}, in addition to the effect of $\eta$.

Next let us analyze the limiting behavior of $y_\alpha^\star$ so that we can relate the width of $W_n^{\mix,\iwr}$ and $W_n^{\mix,\reg}$ to that of the Wald interval.

\begin{lemma}
\label{lem:ystar-bound}
    For $\kappa=1$, we have that $y_\alpha^\star \xrightarrow{R\to\infty}2\sqrt{\log(\sqrt{2}/\alpha)}$ where $y^\star$ satisfies $I_R(y^\star,1) = Z_{R,1}/\alpha$ and $Z_{R,\kappa} = \kappa\sqrt{2\pi}(\Phi(R/\kappa) - \Phi(-R/\kappa))$. 
\end{lemma}
\begin{proof}
Let $g_1(y) = \sqrt{2}R - y/\sqrt{2}$ and $g_2(y) = -\sqrt{2}R - y/\sqrt{2}$. Then 
\[I_R(y,1) = \sqrt{\pi} \exp\left(\frac{y^2}{4}\right) (\Phi(g_1(y)) - \Phi(g_2(y)). \]
Setting $I_R(y,1 ) = Z_{R,1}/\alpha$ and taking logs, we obtain 
\begin{align*}
   y_\alpha^\star &= 2\sqrt{\log( Z_{R,1}/\alpha) - \log(\sqrt{\pi}) - \log(\Phi(g_1(y_\alpha^\star) - \Phi(g_2(y_\alpha^\star))}  \\
   &= 2\sqrt{\log( \sqrt{2}/\alpha) + \log(\Phi(R)  -\Phi(-R)) - \log(\Phi(g_1(y_\alpha^\star) - \Phi(g_2(y_\alpha^\star))}
\end{align*} 
Now, as $R\to\infty$, $\Phi(R) \to 1$ and $\Phi(-R)\to 0$, so $\log(\Phi(R) - \Phi(-R)) \to 0$. Furthermore, $g_1(y_\alpha^\star) \to 1$ as $R\to\infty$ and $g_2(y_\alpha^\star) \to 0$. To see this, note that either $y_\alpha^\star$ goes to $\infty$, 0, or stays bounded as $R$ grows (it cannot be negative). And in each case we get the desired behavior on $g_1$ and $g_2$. Therefore, $-\log(\Phi(g_1(y_\alpha^\star)) - \Phi(g_2(y_\alpha^\star)))\to 0$ and we conclude that 
\[y_\alpha^\star \to 2\sqrt{\log(\sqrt{2}/\alpha)},\]
completing the argument. 
\end{proof}

\section{Omitted Proofs}
\label{app:proofs}

\begin{lemma}
\label{lem:Ereg_ui}
Let $X_1,\dots,X_n$ be iid with mean $\theta$ and be in the domain of a attraction of a normal law. Let $S_n(\theta)$ and $V_n(\theta)$ be as in~\eqref{eq:Sn-Vn}. 
Fix any $\lambda_{\max}\geq 0$. 
Then: 
\begin{enumerate}
    \item[(i)] The family of random variables $(\sup_{\lambda \in [-\lambda_{\max},\lambda_{\max}]}\exp(\lambda S_n(\theta) / V_n(\theta))$ is uniformly integrable, and 
    \item[(ii)] For a fixed $\eta>0$, the family of random variables $(\sup_{\lambda\in[-\lambda_{\max},\lambda_{\max}]} E_n^\reg(\theta;\lambda,\eta)$, where
\[
E_n^\reg(\theta;\lambda,\eta) = \exp\left(\lambda \frac{\sqrt{n}(\overline{X}_n -\theta)}{\widehat{\sigma}_n + \eta V_n(\theta)/\sqrt{n}} - \frac{\lambda^2}{2}\right),
\]
is uniformly integrable.
\item[(iii)] Let $\lambda_n\to \lambda$ and $\eta_n\to \eta$ almost surely, where $\eta_n>0$ and $\eta\geq 0$. If $|\lambda_n| \lesssim \eta_n$ then the family of random variables $(E_n^\reg(\theta; \lambda_n,\eta_n))$ is uniformly integrable. 
\end{enumerate}

\end{lemma}
\begin{proof}
We begin with (i). Since the $X_i - \theta$ are i.i.d\ with mean zero and in the domain of attraction of a Gaussian, we may apply Theorem 2.5 from \cite{gine1997student} to say that there exists a constant $C>0$ such that:
    \[
    \sup_n \E\left[\exp\left(t \left|\frac{S_n(\theta)}{V_n(\theta)} \right|\right)\right] \le 2\exp\left(Ct^2\right)  \quad \text{for all $t \ge 0$.}
    \]
    In fact, 
    \begin{equation}
        C = \left(1 + \frac{4e}{3}\right)^2 \max\left\{1, \sup_n \E\left|\frac{S_n(\theta)}{V_n(\theta)}\right|\right\}^2. 
    \end{equation}
    We may conclude that $C<\infty$ since $S_n(\theta)/V_n(\theta)\to N(0,1)$ in distribution (by Slutsky's theorem), hence $S_n(\theta)/V_n(\theta)$ is stochastically bounded. Therefore, by Lemma 2.4 in \citet{gine1997student}, $\sup_n \E |S_n(\theta)/V_n(\theta)|<\infty$, implying that $C<\infty$. 
    Note that $C$ is independent of $\lambda$ but not independent of the distribution of $X_i$. 
    This shows that $\exp(\lambda S_n(\theta)/V_n(\theta))$ is integrable for each $n$ and fixed $\lambda$. Moreover, 
    \begin{align*}
        \sup_n \E\sup_\lambda\exp\left\{ \lambda \frac{S_n(\theta)}{V_n(\theta)}\right\}^2 & \le \sup_n \E\exp\left(2\lambda_{\max} \left|\frac{S_n(\theta)}{V_n(\theta)} \right|\right)  
        \le 2\exp\left(4C\lambda_{\max}^2\right) <\infty,    
    \end{align*}
    so by applying de la Vallée Poussin's theorem \citep{valleepoussin1915integrale,meyer1966probability} with the nonnegative increasing convex function $t \mapsto t^2$ on $\Re_+$, we may conclude that $\{\sup_\lambda\exp(\lambda S_n(\theta)/V_n(\theta)\}_n$ is uniformly integrable. This proves (i). 

    For (ii), note that we can write
    \begin{align*}
        \frac{\sqrt{n}(\overline{X}_n -\theta)}{\widehat{\sigma}_n+ \eta V_n(\theta)/\sqrt{n}} &= \frac{\sqrt{n}(\frac{1}{n}\sum_{i\le n} X_i^\theta)}{\sqrt{\sum_{i\leq n} (X_i -\overline{X}_n)^2 / (n-1)} + \eta V_n(\theta)/\sqrt{n}}\\
        &= \frac{\sqrt{n}(\frac{1}{n}\sum_{i\le n} X_i^\theta)}{\sqrt{\sum_{i\leq n} (X_i^\theta -\frac{1}{n}\sum_{i\le n} X_i^\theta)^2 / (n-1)}+ \eta V_n(\theta)/\sqrt{n}}\\
        &= \frac{S_n(\theta)/V_n(\theta)}{\sqrt{(n-(S_n(\theta)/V_n(\theta))^2)/(n-1)}+\eta}.
    \end{align*}
    Then:
    \begin{align*}
        E_n^\reg(\theta;\lambda,\eta) &= \exp\left(\lambda \frac{\sqrt{n}(\overline{X}_n -\theta)}{\widehat{\sigma}_n+ \eta V_n(\theta)/\sqrt{n}} - \frac{\lambda^2}{2}\right)\\
        &=\exp\left(\lambda \frac{\sqrt{n}(\overline{X}_n -\theta)}{\widehat{\sigma}_n+ \eta V_n(\theta)/\sqrt{n}}\right) \exp\left(-\frac{\lambda^2}{2}\right)\\
        &=\exp\left(\lambda \frac{S_n(\theta)/V_n(\theta)}{\sqrt{(n-(S_n(\theta)/V_n(\theta))^2)/(n-1)}+\eta}\right) \exp\left(-\frac{\lambda^2}{2}\right)\\
        &\le\exp\left(\frac{|\lambda|}{\eta} \left|\frac{S_n(\theta)}{V_n(\theta)} \right|\right) \exp\left(-\frac{\lambda^2}{2}\right).
    \end{align*}
    Therefore, 
    \begin{align*}
        \sup_\lambda E_n^\reg(\theta; \lambda,\eta) &\leq \exp\left(\frac{\lambda_{\max}}{\eta} \left|\frac{S_n(\theta)}{V_n(\theta)} \right|\right).
    \end{align*}
    We may now appeal to part (i) with $\lambda_{\max} \gets \lambda_{\max}/\eta$ to complete the proof of (ii). For (iii) let $\infty > M \geq \sup_n |\lambda_n|/\eta_n$, ($M$ is guaranteed to be finite since $|\lambda_n|/\eta_n \lesssim 1$).    Then, proceeding as above, 
    \begin{align*}
    \sup_n \E_P[E_n^\reg(\theta;\lambda_n,\eta_n)] &\leq \sup_n \E_P\left[\exp\left(\frac{|\lambda_n|}{\eta_n}\left|\frac{S_n(\theta}{V_n(\theta)}\right|\right)\right]\exp\left(-\frac{\lambda_n^2}{2}\right) \\
    & \leq \sup_n 
     \E_P\left[\exp\left(M\left|\frac{S_n(\theta}{V_n(\theta)}\right|\right)\right]  \\ 
     &\leq 2\exp( CM^2). 
    \end{align*}
    Then we can apply de la Vall\'ee Poussin's theorem as above, which proves (iii). 
\end{proof}

\subsection{Proof of Proposition~\ref{prop:cs-via-evalue}}
\label{proof:cs-via-evalue}

First, suppose that $\calH_n(\alpha) = \{ \theta: E_n(\theta) < 1/\alpha\}$ and that $(E_n(\theta^\star))$ is an asymptotic e-variable for $\theta^\star$. Notice that for any $\alpha>0$, $\ind{E_n(\theta^\star)\geq 1/\alpha}/\alpha \leq E_n(\theta^\star)$ by case analysis on the indicator. Further, by taking $\alpha =1/E_n(\theta^\star)$, we see that $\ind{E_n(\theta^\star)\geq 1/\alpha}/\alpha \geq E_n(\theta^\star)$. Therefore, 
\begin{equation}
    \sup_{\alpha >0} \frac{\ind{\theta^\star \notin \calH_n(\alpha)}}{\alpha} = \sup_{\alpha >0} \frac{\ind{E_n(\theta^\star) \geq 1/\alpha}}{\alpha} = E_n(\theta^\star). 
\end{equation}
Thus, taking expected values, the fact that $(E_n(\theta^\star))$ is a (uniform) asymptotic e-variable implies that $(\calH_n)$ is a (uniform) \aphci. 

Conversely, let $\calH_n(\cdot)$ be an \aphci for $\theta^\star$. For any $\theta \in \Theta$, define
    \begin{equation}
        E_n(\theta) := \sup \left \{ \frac{1}{\alpha} : \alpha > 0 \text{ and } \theta \notin \calH_n(\alpha) \right \} = \sup_{\alpha > 0} \frac{\1 \{ \theta \notin \calH_n(\alpha) \}}{\alpha}.
    \end{equation}
    By the property of $(\calH_n)$ forming an  \aphci for $\theta^\star$, it holds that $E_n(\theta^\star)$ is an asymptotic e-variable. Likewise in the distribution-uniform setting. 
    
    We will now show that for any $\alpha \in (0, 1)$, $\calH_n(\alpha) = \{ \theta : E_n(\theta) < 1/\alpha \}$.
    Indeed, fix any $\theta_0 \in \Re$ and $\alpha_0 \in (0, 1)$.
    Since $\calH_n(\cdot)$ is closed from below, the supremum in the definition of $E_n({\theta_0})$ is attained, i.e.
    \begin{equation}
        E_n({\theta_0}) = \frac{1}{\alpha(\theta_0)}\quad \text{where}\quad \alpha(\theta_0) = \min\{ \alpha > 0 : \1 \{ \theta_0 \notin \calH_n(\alpha) \} \},
    \end{equation}
    taking the above minimum to be $\infty$ if the set is empty.
    \paragraph{Case I: $\alpha_0 = \alpha(\theta_0)$.} Clearly, $E_n({\theta_0}) \geq 1/\alpha_0$ and $\theta_0 \notin \calH_n(\alpha_0)$.
    \paragraph{Case II: $\alpha_0 < \alpha(\theta_0).$} By definition of $\alpha(\theta_0)$ being the minimum such $\alpha$ for which $\theta_0$ is excluded from $\calH_n(\alpha)$, we have that $\theta_0 \in \calH_n(\alpha_0)$. Moreover, we have that by definition, $E_n({\theta_0}) = 1/\alpha(\theta_0) < 1/\alpha_0$.
    \paragraph{Case III: $\alpha_0 > \alpha(\theta_0).$} By monotonicity of $\calH_n(\cdot)$, it holds that $\theta_0 \notin \calH_n(\alpha_0)$. Once again by definition of $E_n({\theta_0})$, we have $E_n({\theta_0}) = 1/\alpha(\theta_0) > 1/\alpha_0$. 

    In all three cases, we have $E_n({\theta_0}) \geq 1/\alpha_0$ if and only if $\theta_0 \notin \calH_n(\alpha_0)$. Since $\theta_0$ and $\alpha_0 \in (0, 1)$ were arbitrary, this completes the proof.

\subsection{Proof of Theorem~\ref{thm:iwr-asymp-evar}}
\label{proof:iwr-asymp-evar}

We handle the pointwise case and the distribution-uniform case separately. The latter is significantly more involved than the former. 

\subsubsection{Pointwise case.}
Let $(\lambda_n)$ be a bounded sequence which converges to $\lambda\in\Re$ almost surely. 
By Lemma~\ref{lem:Ereg_ui} (i), it follows that $(E_n^\iwr(\theta);\lambda_n))_n$ is uniformly integrable. Moreover, 
\[\frac{S_n(\theta)}{V_n(\theta)} = \frac{\frac{1}{\sqrt{n}} \sum_{i\leq n} (X_i - \theta)}{\frac{1}{\sqrt{n}}\sqrt{\sum_i (X_i - \theta)^2}} \xrightarrow{d} N(0,1),\]
by \citet[Theorem 3.3]{gine1997student}. 
Therefore, by the continuous mapping theorem and a second application of Slutsky's theorem, $E_n^\iwr(\theta;\lambda_n) \xrightarrow{d} \exp(\lambda Z - \lambda^2/2)$ for $Z\sim N(0,1)$. Since convergence in distribution together with uniform integrability imply convergence in expectation \citep[Theorem 25.12]{billingsley1995proba}, we have $\lim_{n\to\infty} \E[E_n^\iwr(\theta;\lambda_n)] = \E[\exp(\lambda Z - \lambda^2/2)]=1$, completing the argument.

\subsubsection{Distribution-uniform case.}
Define the nonnegative function $g(T;\lambda) = \exp(\lambda T - \lambda^2/2)$. Let $Z\sim N(0,1)=R$ be a standard normal random variable, so that $\E_R[g(Z)]=1$. 
    For a fixed $\lambda$ and $T_n = S_n(\theta) / V_n(\theta)$, we will show that $g(T_n;\lambda)$ obeys the conditions of Proposition~\ref{prop:uniform-billingsley}. 

    \paragraph{Step I.} We claim that $g(T_n;\lambda)$ satisfies 
    \begin{equation}
    \label{eq:pf-du-iwr-1}
      \lim_{K\to\infty} \limsup_n \sup_P \E_P[g(T_n;\lambda) \ind{g(T_n;\lambda)\geq K}.  
    \end{equation}
    Observe that $g(T_n;\lambda)\ind{g(T_n;\lambda) \geq K}\leq \exp(\lambda T_n)\ind{\exp(\lambda T_n)\geq K}$ (since $\exp(-\lambda^2/2)\leq 1$). Further, on $\{ \exp(\lambda T_n)\geq K\}$, for any $U>\lambda$ we have $\exp(T_n(U - \lambda)) \geq K^{(U-\lambda)/\lambda)}$, i.e., 
    \[1\leq \exp(T_n ( U -\lambda)) K^{-(U-\lambda)/\lambda}.\]
    Multiplying both sides by $\exp(\lambda T_n)\ind{\exp(\lambda T_n)\geq K}$, we obtain that 
    \begin{equation}
    \label{eq:pf-du-iwr-2}
        g(T_n;\lambda)\ind{g(T_n;\lambda)\geq K} \leq \exp(\lambda T_n) \ind{\exp(\lambda T_n)\geq K} \leq \frac{\exp( U T_n)}{K^{(U-\lambda)/\lambda}}. 
    \end{equation}
    Therefore, in order to show that~\eqref{eq:pf-du-iwr-1} holds, it suffices to show that 
    \begin{equation}
        \limsup_n \sup_P \E_P[\exp(U T_n)] <\infty. 
    \end{equation}
    As we did in the pointwise case, we may apply Theorem 2.5 from \cite{gine1997student} to conclude that there exists a constant $C>0$ such that:
    \[
    \E_P[\exp(U T_n)] \leq 2\exp\left(C_PU^2\right), 
    \]
    where 
    \begin{equation}
        C_P = \left(1 + \frac{4e}{3}\right)^2 \max\left\{1, \max_{\ell\leq n} \E_P|T_\ell|\right\}^2. 
    \end{equation}
    We therefore need to bound $\E_P[|T_\ell|]$ uniformly over all $P$. To do so, we appeal to a Berry-Esseen bound for self-normalized sums (see \citet{bentkus1996berry} or \citet[Chap.\ 5, Theorem 5.9]{de2009self}), which states that 
    \begin{equation}
        \sup_{z\in\Re}\left|P(T_n \leq z) - \Phi(z)\right| \leq \frac{25\E_P|X_i - \theta|^3}{\sqrt{n}\sigma_P^3}\leq \frac{25}{\sqrt{n}} M,
    \end{equation}
    for some finite $M$ by assumption of uniformly bounded skew, 
    and where $\Phi$ is the cdf is the standard normal. 
    From here, note that $|T_n|\leq \sqrt{n}$ by Cauchy-Schwarz. Hence 
    \begin{align*}
        \E_P|T_n| &= \int_0^{\sqrt{n}} P(|T_n|>r)\d r = \int_0^{\sqrt{n}} P(T_n > r) + P(T_n < -r) \d r \\
        &= \int_0^{\sqrt{n}} P(T_n > r)\d r + \int_{-\sqrt{n}}^0 P(T_n < r)\d r \\
        &\leq \int_0^{\sqrt{n}} P(Z > r)\d r + \int_{-\sqrt{n}}^0 P(Z < r)\d r + 50 M\\
        & \leq \int_0^{\infty}P(|Z|\geq r)\d r + 50M < 1 + 50M < \infty.  
    \end{align*}
    Therefore, we can upper bound the distribution-dependent constant $C_P$ by the distribution independent constant $C$: 
    \begin{equation*}
        C_P \leq  \left(1 + \frac{4e}{3}\right)^2 (1 + 50M)^2 =: C, 
    \end{equation*}
    in which case 
    \begin{equation*}
        \sup_P \E_P[\exp(U T_n)] \leq 2 \exp(C U^2). 
    \end{equation*}
    Returning to~\eqref{eq:pf-du-iwr-2}, we have 
    \begin{align*}
        \limsup_n \sup_P \E_P[g(T_n;\lambda)\ind{g(T_n;\lambda)\geq K}] &\leq \limsup_n \sup_P \E_P\left[\frac{\exp(UT_n)}{K^{(U-\lambda)/\lambda}}\right]  \\ 
        &\leq \frac{2\exp(CU^2)}{K^{(U-\lambda)/\lambda}}\xrightarrow{K\to\infty} 0,
    \end{align*}
    as needed. 

\paragraph{Step II.} Next we want to show that for any $K>0$, 
\begin{equation}
    \limsup_n \sup_P \left|\E_P[g(T_n)\wedge K] - \E_R[g(Z) \wedge K]\right| = 0,
\end{equation}
where $Z\sim N(0,1)=R$. Let $h(x) = g(x)\wedge K$. Using integration by parts for Riemann-Stieltjes integrals, write 
\[\E_P[h(T_n)] = \int_{-\infty}^\infty h(x) \d P (x)= h(x) P(x)\bigg|_{-\infty}^\infty - \int_{-\infty}^\infty P(x) \d h(x) = K - \int_{-\infty}^\infty P(x) \d h(x). \]
Using similar logic on $\E_R[h(Z)]$, we have 
\begin{align*}
    \left|\E_P[h(T_n)] - \E_R[h(Z)]\right| &= \left|\int_{-\infty}^\infty R(x) \d h(x) - \int_{-\infty}^\infty P(x) \d h(x) \right| \\
    &= \left|\int_{-\infty}^\infty R(x) - P(x) \d  h(x)\right|  \\
    &\leq \sup_{z\in\Re} |R(x) - P(x)| \int_{-\infty}^\infty | \d h(x)|,  
\end{align*}
where the final integral $\int |\d h(x)|$ is the total variation of $h$. Since $\lim_{x\downarrow-\infty} h(x) = 0$ and $\lim_{x\uparrow \infty} h(x) = K$, the total variation of $h$ is $K$. 
Hence, using the Berry-Esseen bound above we have 
\begin{align*}
    \sup_P \left|\E_P[h(T_n)] - \E_R[h(Z)]\right|\leq \frac{25KM}{\sqrt{n}} \xrightarrow{n\to\infty}0,
\end{align*}
which is the desired result. 

\paragraph{Step III.} The previous two steps, combined with Proposition~\ref{prop:uniform-billingsley}, have shown that 
\begin{equation}
\label{eq:pf-du-iwr-3}
  \limsup_{n\to\infty} \sup_P |\E_P[g(T_n;\lambda) - 1|=0.  
\end{equation}
It remains to allow $\lambda$ to vary with $n$. By the triangle inequality, write 
\begin{align}
\label{eq:pf-du-iwr-4}
    |\E_P[g(T_n;\lambda_n)] - 1| \leq |\E_P[g(T_n;\lambda_n) - g(T_n;\lambda)]| + |\E_P[g(T_n;\lambda)] - \E_R[g(Z)]|. 
\end{align}
By~\eqref{eq:pf-du-iwr-2} we have control of the second difference. For the first, we apply the mean value theorem to $g$ to conclude that there exists some $\widetilde{\lambda}$ such that 
\begin{align*}
    g(T_n;\lambda_n) - g(T_n;\lambda) &\leq |\lambda_n - \lambda| \frac{\partial g(T_n;\widetilde{\lambda})}{\partial \lambda} = |\lambda_n - \lambda||T_n - \widetilde{\lambda}| g(T_n;\widetilde{\lambda}) \\ 
    &\leq |\lambda_n - \lambda| |T_n + U|\exp( U T_n) 
    \leq U |\lambda_n - \lambda| \exp((U + 1/U)T_n),
\end{align*}
where we've used that $1 + x\leq e^x$ for $x\geq 0$. Therefore, using the results from step I, we have 
\begin{align*}
    \sup_P|\E_P[g(T_n;\lambda_n) - g(T_n;\lambda)]| &\leq U|\lambda_n - \lambda| \sup_P \E_P\exp((U + 1/U)T_n) \\
    &\leq 2U|\lambda_n - \lambda| \exp(C (U + 1/U)^2) 
\end{align*}
which converges to 0 as $n\to\infty$ since $\lambda_n \to \lambda$. Therefore, combining~\eqref{eq:pf-du-iwr-3} and~\eqref{eq:pf-du-iwr-4}, we have 
\begin{equation*}
    \limsup_{n\to\infty} \sup_P |\E_P[g(T_n;\lambda)] - 1| = 0,
\end{equation*}
which completes the proof for $(E_n^\iwr)$.

\subsection{Proof of Theorem~\ref{thm:iwr-aphci}}
\label{proof:iwr-aphci}

Throughout the proof, let $A_n = A_{n,\lambda,\alpha}$. 
Fix $\lambda>0$. We will apply Proposition~\ref{prop:cs-via-compound-eval} with $K=2$ and $E_n^\iwr(\theta; \lambda)$ and $E_n^\iwr(\theta; -\lambda)$. That is, 
\[\calH_n^\iwr(\alpha) = \{ \theta: E_n^\iwr(\theta; \lambda) <2/\alpha \text{ and }E_n^\iwr(\theta; -\lambda) < 2/\alpha\}.\]
Let us first consider $\theta$ such that $E_n^\iwr(\theta;\lambda)<2/\alpha$, which we can rewrite as 
\begin{equation}
\label{eq:pf-aphcis-1}
    S_n(\theta) < \sqrt{n} V_n(\theta) A_n.
\end{equation}
Since $V_n(\theta),A_n\geq 0$, this equation is satisfied whenever $S_n(\theta) < 0$, i.e., whenever $\theta>\Xbar_n$. Thus part of the solution set to~\eqref{eq:pf-aphcis-1} is $R_1 = (\Xbar_n,\infty)$. To find the remaining values of $\theta$ which satisfy~\eqref{eq:pf-aphcis-1}, assume that $S_n(\theta)\geq 0$. Then we may square both sides to obtain 
\begin{equation}
\label{eq:pf-aphcis-2}
 S_n^2(\theta) < n V_n^2(\theta)A_n^2   
\end{equation}
Let $y = \Xbar_n - \theta$ and notice that $S_n(\theta) = ny$ and $V_n(\theta) = \sqrt{\sigmahat_n^2(n-1) + ny^2}$. Hence~\eqref{eq:pf-aphcis-2} can be rewritten as $n^2y^2 < n(\sigmahat_n^2 (n-1) + ny^2) A_n^2$. If $1 - A_n^2>0$, this becomes 
\begin{equation}
   y < \sqrt{\frac{(n-1)\sigmahat_n^2 A_n^2}{n(1 - A_n^2)}} = \frac{\sigmahat_n A_n }{(1 - A_n^2)^{1/2}}\sqrt{\frac{n-1}{n}}=:B_n.
\end{equation}
Which gives the interval $R_2 = (\Xbar_n - B_n, \infty)\supset R_1$ as our solution set. If $1 - A_n^2\leq 0$, then~\eqref{eq:pf-aphcis-2} is solved by any $S_n(\theta)>0$, i.e., $\theta < \Xbar_n$, in which case the resulting confidence interval is the entire real line. To ensure that $1 -A_n^2 >0$ it suffices that 
\[ n\geq \left(\frac{\log(2/\alpha)}{\lambda} + \frac{\lambda}{2}\right)^2,\]
which gives the constraint on $n$ in the statement of the theorem. Solving for $\theta$ in $E_n^\iwr(\theta;-\lambda) < 2/\alpha$ in similar fashion gives the interval $R_2'= (-\infty, \Xbar_n + B_n)$. Taking the intersection $R_2\cap R_2'$ gives the result.

\subsection{Proof of Proposition~\ref{prop:mixture-is-valid}}
\label{proof:mixture-is-valid}

This result follows from the more general version stated in Appendix~\ref{app:generalized-mixture}. However, let us prove it directly for clarity. 

Let $\eps>0$ be arbitrary. By assumption of uniform integrability, there exists some $T_\eps>0$ such that 
\[\sup_n \E_P\left[\sup_\lambda E_n(\lambda) \ind{\sup_\lambda E_n(\lambda) \geq T_\eps}\right]\leq \eps,\]
where $\lambda$ ranges over all $\lambda\in\Lambda$. 
Write 
\begin{align*}
    E_n(\pi) &= \int E_n(\lambda) \ind{E_n(\lambda)< T_\eps} \d\pi(\lambda) + \int E_n(\lambda) \ind{E_n(\lambda)\geq  T_\eps} \d\pi(\lambda).
\end{align*}
By Tonelli's theorem, 
\begin{align*}
    \E_P[E_n(\pi)] &= \int \E_P[E_n(\lambda) \ind{E_n(\lambda)< T_\eps}] \d\pi(\lambda) + \int \E_P[E_n(\lambda) \ind{E_n(\lambda)\geq  T_\eps}]\d\pi(\lambda) \\ 
    &\leq \int \E_P[E_n(\lambda) \ind{E_n(\lambda)< T_\eps}] \d\pi(\lambda) + \eps. 
\end{align*}
Since $E_n(\lambda) \ind{E_n(\lambda)\geq  T_\eps}$ is bounded, reverse Fatou's lemma gives 
\begin{align*}
  \limsup_n \E_P[E_n(\lambda)]  &\leq \int \limsup_n \E_P[E_n(\lambda) \ind{E_n(\lambda)< T_\eps}] \d\pi(\lambda) + \eps \\ 
  &\leq \int \limsup_n \E_P[E_n(\lambda)] \d\pi(\lambda) + \eps \leq 1 + \eps. 
\end{align*}
Since $\eps$ was arbitrary, this gives the desired result in the pointwise case. In the distribution-uniform case the argument is similar. We have the guarantee
\[\sup_n \sup_P \E_P\left[\sup_\lambda E_n(\lambda) \ind{\sup_\lambda E_n(\lambda) \geq T_\eps}\right]\leq \eps,\]
and then after applying Tonelli's theorem we take the supremum over $P$ to obtain 
\begin{align*}
    \sup_P\E_P[E_n(\pi)] &= \int \sup_P \E_P[E_n(\lambda) \ind{E_n(\lambda)< T_\eps}] \d\pi(\lambda) +\eps. 
\end{align*}
From there the argument is nearly identical, mutatis mutandis.

\subsection{Proof of Theorem~\ref{thm:iwr-mixture-aphci}}
\label{proof:iwr-mixture-aphci}
Before we begin, observe that we can write 
\[S_n(\theta) = n\Delta \text{~ and ~} V_n^2(\theta) = ns_n^2 + n\Delta^2,\]
where $\Delta = (\Xbar_n - \theta)$ and $s_n = n^{-1}\sum_{i=1}^n (X_i - \Xbar_n)^2 = (n-1)\sigmahat_n^2$ is the biased sample variance. 

Let us start with $E_n^{\mix,\iwr}$. Let $u = (\kappa^2+1)/(2\kappa^2)$. Since $I_R(y,u)$ is increasing in $|y|$, we are searching for those $\theta$ such that $|S_n(\theta)/V_n(\theta)|<y^\star$ where $I_R(y^\star,u) < Z_{R,\kappa}\alpha$. Rewriting $|S_n(\theta)/V_n(\theta)|<y^\star$ gives 
\[|n\Delta|<y^\star\sqrt{n s_n^2 + n\Delta^2},\]
and squaring both sides leads to 
\[|\Delta| < \frac{y^\star s_n}{\sqrt{n - (y^\star)^2}},\]
which is precisely the interval defined by~\eqref{eq:mix-iwr-aphci}.

\subsection{Proof of Theorem~\ref{thm:rws-evariable}}
\label{proof:rws-variable}

We will use the following lemma from \citet[Theorem 2.3]{ruf2025concentration} (instantiated with $q=1$). 
\begin{lemma}[A nonasymptotic SLLN]\label{lemma:nonasymp-slln}
  Let $Z_1, Z_2, \dots$ be iid with finite mean $\theta$ and law $P$. For $x\geq 0$, define $U(x) = \E_P[|Z_1-\theta|\ind{|Z_1 - \theta|\geq x}]$. Then, for all $\eps \in (0 ,6.13]$ and $\gamma\in(0,1/2)$,
  \begin{equation}
    P \left [ \sup_{k \geq m} \left \lvert \frac{1}{k} \sum_{i=1}^k Z_i - \theta \right \rvert \geq \eps \right ] \leq c_\gamma \exp \{ -m^{1-\gamma} \} + \frac{451}{\eps^2} U\left(\frac{\eps m^\gamma}{6.13}\right),
  \end{equation}
  where $c_\gamma = (3-4\gamma)/(2 - 4\gamma)$. 
\end{lemma}
Now let us proceed to the proof. Recall the definition of $U_n^\otimes$ from~\eqref{eq:howard-evar} and note that $\E[U_n^\otimes] = n\sigma^2$. 
Further note that 
\[\frac{1}{n}\sum_{i=1}^n (X_i - \theta)^2 = \frac{1}{n}\sum_{i=1}^n (X_i - \Xbar_n)^2 + (\Xbar_n-\theta)^2,\]
for all $n\geq 1$. Hence 
\begin{align*}
  \shat_n^2 - \frac{U_n^\otimes}{n} &= \frac{2}{3}\left(\frac{1}{n}\sum_{i=1}^n (X_i -\theta)^2 - \sigma^2\right) - (\Xbar_n-\theta)^2. 
\end{align*}
Taking absolute values, a union bound gives 
\begin{equation*}
    P\left(\bigg|\shat_n^2 - \frac{U_n^\otimes}{n}\bigg|\geq \eps\right) \leq P\left(\frac{2}{3}\left|\frac{1}{n}\sum_{i=1}^n (X_i -\theta)^2-\sigma^2\right|\geq \frac{\eps}{2}\right) + P\left( (\Xbar_n-\theta)^2\geq \frac{\eps}{2}\right). 
\end{equation*}
We may apply Lemma~\ref{lemma:nonasymp-slln} to both terms on the right hand side of the above display. This gives 
\begin{align*}
    &P\left(\frac{2}{3}\left|\frac{1}{n}\sum_{i=1}^n (X_i -\theta)^2-\sigma^2\right|\geq \frac{\eps}{2}\right) \\
    &= P\left(\left|\frac{1}{n}\sum_{i=1}^n (X_i -\theta)^2-\sigma^2\right|\geq \frac{3\eps}{4}\right) \\
    &\leq 2\exp(-n^{1-2\gamma}) + \frac{(4/3)^2\cdot 451}{\eps^2} \E_P\left[|X_1 - \theta|^2\ind{|X_1 - \theta|^2 \geq \frac{3\eps n^{\gamma}}{4}}\right] \\
    &\leq 2\exp(-n^{1-2\gamma}) + \frac{A_\delta}{\eps^{2+\delta/2} n^{\delta\gamma/2}} \E_P[|X_1-\theta|^{2+\delta}], 
\end{align*}
where  $A_\delta = (4/3)^2 \cdot 451 \cdot (4\cdot 38/3)^{\delta/2}$ and 
the final inequality uses that $\E[Z^2 \ind{Z\geq u}] \leq u^{-\delta} \E[Z^{2+\delta}]$ for any nonnegative random variable $Z$ (on $\{Z\geq u\}$ we have $Z^{2+\delta} = Z^2 Z^\delta \geq Z^2 u^\delta$). Note that here we're using that $\delta>0$. Further, 
\begin{align*}
   &P\left( (\Xbar_n-\theta)^2\geq \frac{\eps}{2}\right) \\
   &= P\left( |\Xbar_n-\theta|\geq \sqrt{\frac{\eps}{2}}\right) \\
   &\leq 2\exp(-n^{1-2\gamma}) + \frac{902}{\eps}\E_P\left[|X_1 - \theta|\ind{|X_1 - \theta|\geq \frac{\sqrt{\eps/2} \cdot n^{\gamma}}{38}}\right] \\ 
   &\leq 2\exp(-n^{1-2\gamma}) + \frac{902}{\eps}\left(\frac{38\sqrt{2}}{\sqrt{\eps} n^{\gamma}}\right)^{1+\delta}\E_P\left[|X_1 - \theta|^{2+\delta}\right].
\end{align*}
Therefore, combining the above two displays, 
\begin{align*}
    P\left(\bigg|\shat_n^2 - \frac{U_n^\otimes}{n}\bigg|\geq \eps\right) &\leq 4\exp(-n^{1-2\gamma}) + \left(\frac{A_{\delta}}{n^{\delta\gamma/2} \eps^{2+\delta/2}} + \frac{B_\delta}{\eps^{(3 + \delta)/2} n^{(1 + \delta)\gamma}}\right)\E_P|X_1 - \theta|^{2+\delta},
\end{align*}
for constants $A_\delta,B_\delta$. 
Define the ``good'' event, 
\begin{equation*}
    A_n := \left\{ \bigg|\shat_n^2 - \frac{U_n^\otimes}{n}\bigg| \leq \eps_n\right\}. 
\end{equation*}
On $A_n$ we have that $n(\shat_n^2 +\eps_n) \geq U_n^\otimes$, hence 
\begin{align*}
\E_P[E_n^\rws(\theta; \eps_n,T_n)\ind{A_n}] &\leq \E_P\left[\left(\int E_n^\star(\lambda) \d F\wedge T_n \right)\ind{A_n}\right] \\
&\leq \E_P\left[\int E_n^\star(\lambda) \d F\right] = \int \E_P [E_n^\star(\lambda)] \d F \leq 1, 
\end{align*}
where the final equality follows from Tonelli's theorem and the final inequality comes from the fact that $(E_n^\star(\lambda))$ is a nonnegative supermartingale with initial value 1. 
Therefore, letting $E_n = E_n^\rws(\theta;\eps_n,T_n)$ for brevity, 
\begin{align*}
    \E_P[E_n] &= \E_P[E_n \ind {A_n}] + \E_P[E_n \ind{A_n^c}] 
    \leq 1 + T_n P(A_n^c) \\
    &\leq 1 + \left(4\exp(-n^{1-2\gamma}) + \left(\frac{A_{\delta}}{\eps_n^{2+\delta/2}n^{\delta\gamma/2}} + \frac{B_\delta}{\eps_n^{(3+\delta)/2} n^{(1+\delta)\gamma}}\right)\E\left[|X_1 - \theta|^{2+\delta}\right]\right) \cdot T_n. 
\end{align*}
That is, 
\[\E_P[E_n]\lesssim 1 + T_n \exp\{-n^{1-2\gamma}\} + \frac{T_n}{\eps_n^{2+\delta/2} n^{\delta\gamma/2}} + \frac{T_n}{\eps_n^{(3+\delta)/2} n^{(1+\delta)\gamma}},\]
which converges to one as long as $T_n \eps_n^{-2-\delta/2} = o(n^{\delta\gamma/2})$. Note that precisely the same argument goes through in the distribution-uniform case as long as $\sup_P \E_P|X_1 - \theta|^{2+\delta}<\infty$.

\subsection{Proof of Corollary~\ref{cor:rws-normal-mixture}}
\label{proof:rws-normal-mixture}

After taking $\eps_n = 1/\log n$ and $\gamma = 0.49$, $C\cdot n^{\delta 0.24} = o(n^{\delta\gamma/2} / \log^{2 + \delta/2}n)$ for any constant $C$. Therefore, it remains to compute the integral. 
Let $F$ be a Gaussian centered at zero with variance $v^2$. Set $U(X^n) = n(\shat_n^2 + \eps_n)$. 
Compute 
\begin{align*}
    G_n &:= \int\exp\left\{ \lambda S_n(\theta) - \frac{\lambda^2}{2}U(X^n)\right\}\d F \\ 
    &= \frac{1}{\sqrt{2\pi v^2}} \int \exp\left\{\lambda S_n(\theta) - \frac{\lambda^2}{2}U(X^n) - \frac{\lambda^2}{2v^2}\right\}\d \lambda \\ 
    &= \frac{1}{\sqrt{2\pi v^2}} \int \exp\left\{-\frac{\lambda^2 (v^2U(X^n) +1) - 2v^2\lambda S_n(\theta)}{2v^2}\right\}\d\lambda. 
\end{align*}
Set $h_n = v^2U(X^n) + 1$ in which case we can rewrite the above as 
\begin{align*}
    G_n &= \frac{1}{\sqrt{2\pi v^2}} \int \exp\left\{-\frac{\lambda^2 - 2v^2\lambda S_n(\theta)/h_n}{2v^2/h_n}\right\}\d\lambda \\ 
    &=  \frac{1}{\sqrt{2\pi v^2}} \int \exp\left\{-\frac{(\lambda - v^2S_n(\theta)/h_n)^2- v^4 S_n^2(\theta)/h_n^2}{2v^2/h_n}\right\}\d\lambda\\ 
    &= \exp\left\{\frac{v^2 S_n^2(\theta)}{2h_n}\right\}\frac{1}{\sqrt{2\pi v^2}} \int \exp\left\{-\frac{(\lambda - v^2S_n(\theta)/h_n)^2}{2v^2/h_n}\right\}\d\lambda. 
\end{align*}
Observing that the final integrand is the density of an unnormalized Gaussian with variance $v^2/h_n$, we have 
\begin{align*}
    G_n = \exp\left\{\frac{v^2S_n^2(\theta)}{2h_n}\right\}\frac{1}{\sqrt{h_n}} = \exp\left\{\frac{v^2\varphi^2(X^n)}{2h_n} - \frac{1}{2}\log(h_n)\right\}.  
\end{align*}
Setting $\rho = v^2$ completes the proof.

\subsection{Proof of Proposition~\ref{prop:asymptotic e-process}}
\label{proof:asymptotic e-process}
  Fix $m \in \NN$ and let $\tau$ be an arbitrary stopping time. Define $T_m = m^{\delta \cdot 0.24}$ and consider the event 
  \[A_m = \left\{ \sup_{k\geq m} \left|\shat_k^2 - \frac{U_k^\otimes}{k}\right| \leq \frac{1}{\log(m)}\right\}.\]
  In the proof of Theorem~\ref{thm:rws-evariable} we showed that 
  \begin{equation}
    P(A_m^c) \leq 4\exp(-m^{1-2\gamma}) + \left(\frac{A_{\delta}}{m^{\delta\gamma/2} \eps^{2+\delta/2}} + \frac{B_\delta}{\eps^{(3 + \delta)/2} m^{(1 + \delta)\gamma}}\right)\E_P|X_1 - \theta|^{2+\delta},
  \end{equation}
  Note that in that proof we did not have the supremum over $k\geq m$ in the event $A_m$ since it was not required. However, the bound remains the same with this supremum due to the uniformity of the bound in the nonasymptotic SLLN of \citet{ruf2025concentration}, stated in Lemma~\ref{lemma:nonasymp-slln}. 
  From here, write 
  \begin{align*}
      \E_P\left[E_{\taumax}^\brackm\right] &= \E_P\left[E_{\taumax}^\brackm\ind{A_m}\right] + \E_P\left[E_{\taumax}^\brackm\ind{A_m^c}\right] \\ 
      &\leq \E_P\left[E_{\taumax}^\brackm\ind{A_m}\right] + T_m P(A_m^c). 
  \end{align*}
  As we argued in the proof of Theorem~\ref{thm:rws-evariable}, $T_mP(A_m^c)\xrightarrow{m\to\infty} 0$. Further, on $A_m$, we have $\shat_k - \log^{-1}(m) \geq U_k^\otimes/k$ for all $k\geq m$, including $k = \tau \vee m$. Therefore, 
  \begin{align*}
      \E_P\left[E_{\taumax}^\brackm\ind{A_m}\right] &= \E_P\left[\left\{\int_{\lambda \in \Re}\exp \left \{ \lambda \sum_{i=1}^{\tau\vee m} (X_i - \theta) - U_{\tau\vee m}^\otimes\frac{\lambda^2}{2} \right \} d F(\lambda) \wedge T_m  \right\}\ind{A_m} \right]\\
      & \leq  \E_P\left[\int_{\lambda \in \Re}\exp \left \{ \lambda \sum_{i=1}^{\tau\vee m} (X_i - \theta) - U_{\tau\vee m}^\otimes\frac{\lambda^2}{2} \right \} d F(\lambda) \right] \leq 1,
  \end{align*}
  where the final inequality follows from the results of \citet{howard2020time} (in particular, the process above is a nonnegative supermartingale with initial value one, thus has expectation at most one at all stopping times). This completes the proof.

\subsection{Proof of Theorem~\ref{thm:E-sn-and-reg}}
\label{proof:E-sn-and-reg}

We take each asymptotic e-variable in turn. 

\subsubsection{$(E_n^\sn)$ is an asymptotic e-value.}

This asymptotic e-value is based on the following result from \citet[Lemma 3f]{howard2020time}. See also \citet[Lemma 3.1]{chugg2025time} for a statement which more closely resembles our presentation here (though in the multivariate setting), as well as a direct proof.

\begin{lemma}
\label{lem:howard-eval}
For any predictable sequence $(\lambda_n)$ taking values in $\Re$,  (i.e., $\lambda_n$ is $\calF_{n-1}=\sigma(X_1,\dots,X_{n-1})$-measurable), 
the process $(E_n^\star)_n$ given by
        \begin{equation}
            E_n^\star(\lambda) := \exp \left \{  \sum_{i=1}^n \lambda_i(X_i - \theta) -  \frac{1}{2}V_n^\otimes \right \},
        \end{equation}
        is a test supermartingale, where $V_n^\otimes$ is given by
        \begin{equation}
        \label{eq:vn-sn}
            V_n^\otimes := \frac{1}{3}\sum_{i=1}^n\lambda_i^2 ((X_i - \theta)^2 + 2\sigma^2).
        \end{equation}
        In particular, $E_\tau^\star$ is an e-value at any stopping time $\tau$.
    \end{lemma}
    Recall that we're trying to prove that 
    \begin{equation*}
        E_n^\sn(\theta; \lambda^n) = \E_P\exp\left\{ \sum_{i\leq n}\lambda_i (X_i -\theta) - \sum_{i\leq n}\frac{\lambda_i^2}{6}[(X_i - \theta)^2 + 2\sigmahat_{i-1}^2]\right\},
    \end{equation*}
    defines an asymptotic e-value. Write 
    \[
        \lambda_i (X_i-\theta) - \frac{1}{6}\lambda_i^2[(X_i-\theta)^2 + 2\sigmahat_{i-1}^2] = \lambda_i (X_i-\theta) - \frac{1}{6}\lambda_i^2[(X_i-\theta)^2 + 2\sigma^2] + \frac{\lambda_i^2}{3}( \sigma^2-\sigmahat_{i-1}^2),\]
        so 
        \begin{equation}
            \E_P\exp\left\{ \lambda_i (X_i -\theta) - \frac{\lambda_i^2}{6}[(X_i - \theta)^2 + 2\sigmahat_{i-1}^2]|\calF_{i-1}\right\} \leq \E_P\exp\left\{ \frac{\lambda_i^2}{3}(\sigma^2 - \sigmahat_{i-1}^2)\right\},
        \end{equation}
        hence the law of total expectation gives
        \begin{equation*}
            \E_P\exp\left\{ \sum_{i\leq n}\lambda_i (X_i -\theta) - \sum_{i\leq n}\frac{\lambda_i^2}{6}[(X_i - \theta)^2 + 2\sigmahat_{i-1}^2]\right\} \leq \E_P\exp\left\{ \sum_{i\leq n}\frac{\lambda_i^2}{3}( \sigma^2-\sigmahat_{i-1}^2)\right\}.
        \end{equation*}
        It thus remains to show that the term on the right has limit at most 1. This is done via the following lemma. 

         \begin{lemma}
             Suppose that (i) $\sup_n \lambda_n <\infty$, (ii) $\sum_{i\leq n}\lambda_i^2 \xrightarrow{n\to\infty} L<\infty$ a.s.\, and (iii) $\lambda_i^2 / \sum_{j\leq n}\lambda_j^2 \xrightarrow{n\to\infty}0$ a.s. Then 
             \begin{equation}
                 \lim_n \E_P\exp\bigg(\frac{1}{3}\sum_{i\leq n}\lambda_i^2(\sigma^2 - \sigmahat_{i-1}^2)\bigg)=1.
             \end{equation}
         \end{lemma}
         \begin{proof}
             Let $A_n = \sum_{i\leq n} \lambda_i^2$. 
             We claim that the term $A_n^{-1}\sum_{i\leq n}\lambda_i^2 \sigmahat_{i-1}^2$ goes to $\sigma^2$ almost surely. If so, then by the continuous mapping theorem, 
             \[\sum_{i\leq n}\lambda_i^2(\sigma^2 - \sigmahat_{i-1}^2) = A_n\left(\sigma^2 -  \frac{\sum_{i\leq n}\lambda_i^2 \sigmahat_{i-1}^2}{A_n}\right)\xrightarrow{a.s.} L(\sigma^2 - \sigma^2) =0.\]
             Since $\sup_n \sum_{i\leq n}(\sigma^2 - \sigmahat_{i-1}^2)\leq \sup_n \sum_{i\leq n}\lambda_i^2 \sigma^2 <\infty$ by assumption,  we may apply the dominated convergence theorem to conclude that 
             \[\lim_n \E_P\exp\bigg(\frac{1}{3}\sum_{i\leq n}\lambda_i^2(\sigma^2 - \sigmahat_{i-1}^2)\bigg)=\E_P\exp\bigg(\frac{1}{3}\lim_n \sum_{i\leq n}\lambda_i^2(\sigma^2 - \sigmahat_{i-1}^2)\bigg)=1.\]
             Now it remains to show that $\sum_{i\leq n} w_{i,n}\sigmahat_{i-1}^2 \xrightarrow{a.s.}\sigma^2$, where $w_{i,n} = \lambda_i^2 / A_n$. Fix $\eps>0$ and let $\omega$ be a sample path on which $\sigmahat_{i-1}^2\to \sigma^2$ and $w_{i,n}\to 0$.  Then there exists some $N = N_{\eps,\omega}$ such that $|\sigmahat_{i-1}^2(\omega) -\sigma^2|\leq \eps$ for all $i>N$. Therefore, noting that $\sum_{i\leq n}w_{i,n} = 1$, 
             \begin{align*}
                 \left|\sum_{i\leq n}w_{i,n}\sigmahat_{i-1}^2(\omega) - \sigma^2\right| &\leq \sum_{i\leq n}w_{i,n}|\sigmahat_{i-1}^2(\omega) - \sigma^2| \leq \sum_{i\leq N} w_{i,n}|\sigmahat_{i-1}^2(\omega) - \sigma^2| + \eps \sum_{N<i\leq n}w_{i,n}. 
             \end{align*}
             The final sum is at most $\eps$ since $\sum_{N<i\leq n}w_{i,n} \leq \sum_{i\leq n}w_{i,n}=1$. Letting $W_N = \max_{i\leq N} |\sigmahat_{i-1}^2(\omega) - \sigma^2|$, the penultimate sum can be bounded as 
             \begin{align*}
                 \sum_{i\leq N} w_{i,n}|\sigmahat_{i-1}^2(\omega) - \sigma^2| &\leq N\cdot W_N \cdot \max_{i\leq N} w_{i,n} \xrightarrow{n\to\infty}, 
             \end{align*}
             since $w_{i,n} = \lambda_i^2 / A_n\to 0$ for each fixed $i$ by assumption. Since $\omega$ lies in a set which occurs with probability 1, we conclude that $\sum_{i\leq n} w_{i,n} \sigmahat_{i-1}^2 \to \sigma$ with probability 1. 
             \end{proof}

\subsubsection{$(E_n^\reg)$ is an asymptotic e-value} 
Let $\lambda_n\to \lambda$ almost surely. First let us suppose that $\eta>0$ is fixed. 
    Write 
    \begin{align*}
    E_n^\reg(\theta;\lambda_n,\eta) &= 
 \exp\left( 
    \lambda_n \,\frac{\sqrt{n}\,(\overline{X}_n - \theta)}
        {\widehat{\sigma}_n + \eta\, V_n(\theta) / \sqrt{n}}
    - \frac{\lambda_n^2}{2}
\right)
\\
&= \exp\left(
    \lambda_n \,\frac{\sqrt{n}(\overline{X}_n - \theta)}{\sigma}
    \cdot \frac{\sigma}{\widehat{\sigma}_n + \eta\, V_n(\theta) / \sqrt{n}}
    - \frac{\lambda_n^2}{2}
\right).
    \end{align*}
By the strong law of large numbers, we have
\[
\widehat{\sigma}_n \xrightarrow{\text{a.s.}} \sigma
\quad\text{and}\quad
\frac{V_n(\theta)}{\sqrt{n}} \xrightarrow{\text{a.s.}} \sigma.
\]
Moreover, by the central limit theorem,
\[
\frac{\sqrt{n}(\overline{X}_n - \theta)}{\sigma} \xrightarrow{d} Z
\quad\text{with}\quad Z \sim N(0,1).
\]
Applying Slutsky’s theorem, it follows that
\[
\exp\left(
    \lambda_n \,\frac{\sqrt{n}(\overline{X}_n - \theta)}
        {\widehat{\sigma}_n + \eta\, V_n(\theta) / \sqrt{n}}
    - \frac{\lambda_n^2}{2}
\right)
\xrightarrow{d}
\exp\left(
    \lambda \,\frac{Z}{1+\eta} - \frac{\lambda^2}{2}
\right).
\]
Lemma~\ref{lem:Ereg_ui} part (ii) ensures uniform integrability of $\{E_n^\reg(\theta;\lambda_n,\eta)\}_n$, so Theorem~25.12 in \cite{billingsley1995proba} implies that convergence in distribution yields
\[
\limsup_{n \to \infty} \E_P\left[E_n^\reg(\theta;\lambda_n,\eta)\right]
= \E_P\!\left[
    \exp\left(
        \lambda \,\frac{Z}{1+\eta} - \frac{\lambda^2}{2}
    \right)
\right].
\]
Since $Z \sim N(0,1)$, we can evaluate the Gaussian moment generating function to obtain
\[
\E_P\!\left[
    \exp\left(
        \lambda \,\frac{Z}{1+\eta} - \frac{\lambda^2}{2}
    \right)
\right]
= \exp\left(
    -\frac{\lambda^2}{2}
    \left[1 - \frac{1}{(1+\eta)^2}\right]
\right) \leq 1,
\]
with equality iff $\lambda=0$. 
This proves that $E_n^\reg(\theta;\lambda_n,\eta)$ is an asymptotic e-value. If $\eta_n\to \eta$ with $\lambda_n \lesssim \eta_n$ then Lemma~\ref{lem:Ereg_ui} part (iii) implies that $\{E_n^\reg(\theta;\lambda_n,\eta_n)\}$ is uniformly integrable. By Slutsky's theorem again, we have 
\[
\exp\left(
    \lambda_n \,\frac{\sqrt{n}(\overline{X}_n - \theta)}
        {\widehat{\sigma}_n + \eta_n\, V_n(\theta) / \sqrt{n}}
    - \frac{\lambda_n^2}{2}
\right)
\xrightarrow{d}
\exp\left(
    \lambda \,\frac{Z}{1+\eta} - \frac{\lambda^2}{2}
\right),
\]
and the rest of the proof remains unchanged from that for the \iwr e-variable.

\subsection{Proof of Theorem~\ref{thm:unif-reg-eval}}
\label{proof:unif-reg-eval}

Throughout the proof, when we write $\sup_P$, the supremum should be understood to be over $Q(\theta)$. 

Let $A_n = \lambda_n B_n - \lambda_n^2/2$,  $B_n = \lambda_n S_n(\theta) / (n\sigmahat_n^2 + V_n(\theta)\eta)$, and $g(y) = \exp(y)$. Let $Z=-\lambda^2/2$ with probability 1 (where $\lambda_n\to\lambda$) so $g(Z) = \E_P[g(Z)] = \exp(-\lambda^2/2)$.  We want to show that $g(A_n) = E_n^\reg(\theta;\lambda_n,\eta_n)$ satisfies the two conditions of Proposition~\ref{prop:uniform-billingsley}.

First, recall that in the proof of Lemma~\ref{lem:Ereg_ui} we showed that 
\[E_n^\reg(\theta;\lambda,\eta) \leq \exp\left(\frac{\lambda}{\eta} |T_n|\right),\]
where $T_n = S_n(\theta)/V_n(\theta)$ as usual. Since $\lambda_n\lesssim \eta_n$ by assumption, we can pick some $U$ with 
\[\max\left\{ \sup_n \frac{\lambda_n}{\eta_n},\lambda\right\}< U <\infty,\]
whence 
\[E_n^\reg(\theta;\lambda_n,\eta_n) \leq \exp(U T_n).\]
Then using a similar argument as in the case of $E_n^\iwr$, write 
\[g(A_n)\ind{g(A_n)\geq K} \leq \frac{\exp(\lambda_n T_n/\eta_n)}{K^{(U- \lambda_n)/\lambda_n}}\leq \frac{\exp(U T_n)}{K^{(U- \lambda_n)/\lambda_n}}.\]
As $n\to\infty$, $K^{(U - \lambda_n)/\lambda_n} \to K^{U/\lambda - 1}>K$, so $1/K^{(U-\lambda_n)/\lambda_n} \xrightarrow{n\to\infty}0$. (If $\lambda=0$, we interpret $U/0$ as $\infty$ and $1/\infty$ as 0.)
Consequently,   
to show that
\[\lim_{K\to\infty}\limsup_{n\to\infty} \sup_P \E_P[E_n^\reg(\theta;\lambda_n,\eta_n)\ind{E_n^\reg(\theta;\lambda_n,\eta_n)\geq K}]=0,\] 
it suffices to show that $\sup_P \exp(U T_n)$ is bounded by a finite constant independent of $n$ and $k$. This is done in precisely the same way as it was for $E_n^\iwr$. 

Next, we want to show that $\limsup_n \sup_P \E_P[g(A_n) \wedge K]\leq \exp(-\lambda^2/2)\wedge K$ for fixed $K>0$, which is the second condition of Proposition~\ref{prop:uniform-billingsley}. In fact, we will show that $\sup_P \E_P[g(A_n)\wedge K]\to 0$. Write 
\begin{align}
    \E_P[g(A_n)\wedge K] &= \int_0^K P(g(A_n) > u)\d u \notag \\ 
    &= \int_0^K P\left(B_n > \frac{\log(u)}{\lambda_n} + \frac{\lambda_n}{2}\right) \d u \notag \\
    &\leq \int_0^K P\bigg(\frac{|S_n(\theta)|}{n\sigmahat_n^2} > \underbrace{\frac{\log(u)}{\lambda_n} + \frac{\lambda_n}{2}}_{:=\eps_n}\bigg)\d u. \label{eq:pf-reg-du-1}
\end{align}
The remainder of the proof focuses on showing that the integrand decays at rate $1/\text{poly(n)}$ simultaneously for all $P$ which will complete the proof. While the details become rather devilish, the intuition is straightforward: we expect that $|S_n(\theta)|/(n\sigmahat_n^2)$ converges to 0, since $S_n(\theta)/n\to 0$ and $\sigmahat_n^2\to \sigma$. Meanwhile, $\eps_n$ converges to a constant or $+\infty$.

Let $\sigma_P^2$ be the variance of $P$, which is finite since the skew is finite.  
Notice that if $|S_n(\theta)| / (n\sigmahat_n^2)> \eps_n$, then either $|S_n(\theta)| > \eps n\sigma_P^2/2$ or $\sigmahat_n^2 < \sigma_P^2/2$. (If not, then $|S_n(\theta)| \leq \eps_n n\sigma^2/2 \leq \eps_n n \sigmahat_n^2$.) Therefore, 
\begin{align}
\label{eq:pf-reg-du-2}
    P\left(\frac{|S_n(\theta)|}{n\sigmahat_n^2} > \eps_n\right) \leq P\left(|S_n(\theta)| > \frac{\eps_n n \sigma_P^2}{2}\right) + P\left(\sigmahat_n^2 < \frac{\sigma_P^2}{2}\right). 
\end{align}
For the first term on the right side of~\eqref{eq:pf-reg-du-2} we apply Markov's inequality: 
\begin{align}
\label{eq:pf-du-reg-3}
    P\left(|S_n(\theta)| > \frac{\eps_n n \sigma_P^2}{2}\right) \leq \frac{4\Var(\sum_i X_i)}{n^2 \eps_n^2 \sigma_P^4} = \frac{4}{n\eps_n^2 \sigma_P^2} \leq \frac{4C}{n\eps_n^2}, 
\end{align}
for some $C>0$ by assumption of a uniformly lower bounded variance. 
Notice that $n\eps_n^2 \to \infty$ since 
\[n\eps_n^2 \asymp \frac{n}{\lambda_n^2} + n \lambda_n + n\lambda_n^2.\]
At least one of these terms goes to $+\infty$: If $n\lambda_n^2 \asymp c$ then $\lambda \asymp 1/\sqrt{n}$ so $n/\lambda_n^2 \asymp n^2 \to \infty$. 
Now for the next term in~\eqref{eq:pf-reg-du-2}. Let $Y_i = X_i - \theta$, $\Bar{Y} = \frac{1}{n}\sum_y Y_i$ and $\Bar{Y^2} = \frac{1}{n}\sum_i Y_i^2$. Recall the standard identity 
\[\sigmahat_n^2 = \frac{n-1}{n} \left(\Bar{Y^2} - \Bar{Y}^2\right).\]
Observe that $\sigmahat_n^2 \geq \sigma_P^2/2$ iff $\Bar{Y^2} - \Bar{Y}^2 \leq \frac{n-1}{2n}\sigma_P^2$ which implies that either $\Bar{Y}^2 > \sigma_P^2/4$ or $\Bar{Y^2} \leq \frac{n-1}{2n} \sigma_P^2 + \frac{1}{4}\sigma_P^2 \leq \frac{3}{4}\sigma_P^2$. A union bound gives 
\begin{equation}
\label{eq:pf-du-reg-4}
    P\left(\sigmahat_n^2 \geq \frac{\sigma_P^2}{2}\right) \leq P\left(\Bar{Y}^2 > \frac{\sigma_P^2}{4}\right) + P\left(\Bar{Y^2} \leq \frac{3}{4}\sigma_P^2\right). 
\end{equation}
Under $P$, $\Bar{Y^2}$ is an average of $n$ iid random variables with mean $\sigma_P^2$. Hence, by Markov's inequality, 
\begin{align*}
       P\left(\Bar{Y^2} \leq \frac{3}{4}\sigma_P^2\right) &= P\left(\sigma^2 - \Bar{Y^2} \geq \frac{1}{4}\sigma_P^2\right) \\ 
       &\leq P\left(|\sigma^2 - \Bar{Y^2}| \geq \frac{1}{4}\sigma_P^2\right)  
       \leq \frac{\E_P|n\sigma_P^2 - \sum_i Y_i^2|^{3/2}}{(n\sigma_P^2/4)^{3/2}}. 
\end{align*}
From here we appeal to an inequality of \citet{von1965inequalities}, which states that for $Z_i$ iid mean zero and any $1\leq p\leq 2$ with $\E|Z_1|^p<\infty$, it holds that
\[\E\bigg|\sum_{i\leq n}Z_i\bigg|^p \leq C_p \sum_{i\leq n} \E|Z_i|^p,\]
where $C_p$ is some (distribution-independent) constant depending on $p$. Using this above with $p=3/2$ gives 
\begin{equation*}
    \frac{\E_P|n\sigma_P^2 - \sum_i Y_i^2|^{3/2}}{(n\sigma_P^2/4)^{3/2}} \leq \frac{C_p n \E_P|\sigma_P^2 - Y_1^2|^{3/2}}{(n\sigma_P^2/4)^{3/2}} \leq \frac{D_p (\sigma_P^3 + \E_P|Y_1|^3}{\sqrt{n}\sigma_P^3},
\end{equation*}
where $D_p = \sqrt{2} \cdot 4^{3/2} C_p$. Here we've used the inequality $(a + b)^p \leq 2^{p-1}(a^p + b^p)$. Using the assumption of a uniform third moment, we have 
\begin{equation}
   \sup_P P\left(\Bar{Y^2} \leq \frac{3}{4}\sigma_P^2\right) \leq \frac{D_p}{\sqrt{n}}\left(1 + \sup_P \frac{\E_P|Y_1|^3}{\sigma_P^3}\right) \leq \frac{D_p}{\sqrt{n}}\left(1 + M\right).
\end{equation}
Next we need to handle the term $P(\Bar{Y}^2 > \sigma_P^4/4)$ in~\eqref{eq:pf-du-reg-4}, which is done via Chebyshev's inequality. Write 
\begin{align*}
    P\left(\Bar{Y}^2 > \frac{\sigma_P^2}{4}\right) = P\bigg(\bigg|\sum_{i\leq n}(X_i-\theta) \bigg| > \frac{n\sigma_P}{2}\bigg) \leq \frac{4\Var(\sum_i X_i)}{n^2\sigma_P^2} = \frac{4}{n}.
\end{align*}
Returning to~\eqref{eq:pf-du-reg-4}, we have shown that 
\begin{equation*}
    P\left(\sigmahat_n^2 \geq \frac{\sigma_P^2}{2}\right) \leq \frac{D_p}{\sqrt{n}}(1 + M) + \frac{4}{n}, 
\end{equation*}
which when combined with~\eqref{eq:pf-du-reg-3}, \eqref{eq:pf-reg-du-2}, and~\eqref{eq:pf-reg-du-1} gives 
\begin{align*}
    \limsup_{n\to\infty}\sup_P \E_P[g(A_n)\wedge K] &\leq \limsup_{n\to\infty} \int_0^K\left(\frac{4C}{n\eps_n^2} + \frac{D_p}{\sqrt{n}}(1+M) + \frac{4}{n}\right)\d u \\ 
    &\leq  \int_0^K\limsup_{n\to\infty}\left(\frac{4C}{n\eps_n^2} + \frac{D_p}{\sqrt{n}}(1+M) + \frac{4}{n}\right)\d u =0, 
\end{align*}
where we've used reverse Fatou's lemma to switch the limit superior and the integral, in addition to the fact that $n\eps_n^2\to \infty$ as argued above. 

We have thus demonstrated that $g(A_n)$ satisfies the two conditions of Proposition~\ref{prop:uniform-billingsley}, which completes the proof.

\subsection{Proof of Theorem~\ref{thm:reg-aphci}}
\label{proof:reg-aphci}

We now turn to $\calH_n^\reg$. Proceeding as usual, rewrite $E_n^\reg(\theta;\lambda,\eta)<2/\alpha$ as 
\begin{align*}
    y -\sigmahat_n A_n <  \eta \frac{V_n(\theta)}{\sqrt{n}}A_n, 
\end{align*}
where again, $y = \Xbar_n - \theta = S_n(\theta)/n$. Since $\eta>0$ and $V_n(\theta) \geq 0$, this inequality is trivially satisfied for $y<\sigmahat_n A_n$, i.e., $\theta > \Xbar_n - \sigmahat_n A_n$.
Thus $R_1 = (\Xbar_n - \sigmahat_n A_n, \infty)\subset \{\theta: E_n^\reg(\theta;\lambda,\eta)<2/\alpha\}$. To find the rest of the set, assume that $y\geq \sigmahat_n A_n$. Then both sides of the above display are positive and we may square them to obtain the inequality
\begin{equation}
    y^2 - 2y\sigmahat_n A_n + \sigmahat_n^2A_n^2 < \frac{\eta^2}{n} (\sigmahat_n^2(n-1) + ny^2)A_n^2. 
\end{equation}
Gathering terms, rewrite this as the polynomial equation 
\begin{equation}
    Q(y) := y^2( 1 - \eta^2A_n^2) - 2y\sigmahat_n A_n + \sigmahat_n^2A_n^2\left( 1 - \eta^2\left(\frac{n-1}{n}\right)\right) < 0. 
\end{equation}
Performing the relevant arithmetic on the quadratic equation, $Q(y)$ has roots 
\begin{equation}
    y_\pm = \frac{\sigmahat_n A_n\left( 1   \pm \sqrt{1 - (1 - \eta^2A_n^2)(1 - \eta^2\frac{n-1}{n})}\right)}{1 - \eta^2A_n^2}
\end{equation}
For $1 - \eta^2A_n^2>0$, we claim that $y_- < \sigmahat_nA_n$, thus disqualifying it from consideration (since we are considering $y$ such that $y\geq \sigmahat_n A_n$). Indeed, let $a = \eta^2 A_n^2$ and $b = \eta^2(n-1)/n$ and notice that 
\begin{align*}
    \sqrt{1 - (1 - a)(1 - b)}  
    &\geq 1 - (1 - a)(1 - b)  
    = b - ab + a > a,
\end{align*}
since $a,b<1$. Therefore, 
\[y_- = \frac{\sigmahat_n A_n (1 - \sqrt{1 - (1-a)(1-b)}}{1 - a} > \frac{\sigmahat_n A_n (1 - a)}{1-a} = \sigmahat_n A_n.\]
Consequently, the solution we are considering is $y_+$, which gives the half interval $\theta > \Xbar_n - y_+$, which subsumes the previous half interval since $y_+ \geq \sigmahat_n A_n$ (this is easy to see: $y_+$ multiples $\sigmahat_nA_n$ by something $>1$ in the numerator and divides it by something $<1$ in the denominator). Solving $E_n^\reg(\theta;-\lambda,\eta)<2/\alpha$ gives the other half interval $\theta < \Xbar_n - y_+$, completing the proof.

\section{Additional Experiments and Simulation Details}
\label{app:sims}

\subsection{Simulation details}

Code both to replicate the experiments and to use the methods independently can  be found at \url{https://github.com/bchugg/asymp-posthoc-confint}. 

\paragraph{Figure~\ref{fig:widths}.}
We generate a single realization of length $N=20000$ from an iid Gaussian model using random seed $0$. 
For each $t \in \{100,200,500,1000,2000,3000,5000,10000,20000\}$, we apply each method to the common prefix $X_{1:t}$ of this same realization, and record the resulting two-sided interval width $U_t - L_t$.
We fix $\alpha = 0.05$.
For ex ante anchoring, we consider both $\alpha_0 = \alpha$ and $\alpha_0 = \alpha/10000$.
We use $\rho=2$ and $\delta=0.1$ for $(\mathcal{H}_n^{\rws})$, and $R=20$ and $\kappa=5$ for $(\mathcal{H}_n^{\mix,\iwr})$.

\paragraph{Figure~\ref{fig:asymp_vs_nonasymp}.}
We implement the betting confidence interval using the \texttt{confseq} package
(\url{https://github.com/gostevehoward/confseq/tree/master/src/confseq}).
We generate a single data stream of length $N=10,000$ consisting of i.i.d.\ Bernoulli observations
$X_1,\dots,X_N \sim \mathrm{Bernoulli}(p)$ with $p=0.25$ (random seed fixed to $0$).
For each $t \in \{300,600,\ldots,9900\}$, all methods are applied to the common prefix $X^t$ and we record the resulting two-sided interval width $U_t-L_t$.
The Bernstein interval is computed on $X^t$ and the known variance $p(1-p)=0.25(1-0.25)$. 
For the mixture IWR method, we use $R=20$ and $\kappa=5$. 
For $(\mathcal H_n^{\iwr})$, we use ex ante anchoring with $\alpha_0=\alpha$ (as discussed in the text), i.e. we set 
$\lambda=\sqrt{2\log(2/\alpha_0)}$,

\newpage 

\subsection{Effect of \texorpdfstring{$R$ and $\kappa$ on $E_n^{\mix,\iwr}$}{R and kappa on truncated Gaussian}}

\begin{figure}[h!]
    \centering
    \begin{subfigure}[b]{0.45\textwidth}
        \centering
        \includegraphics[width=\linewidth]{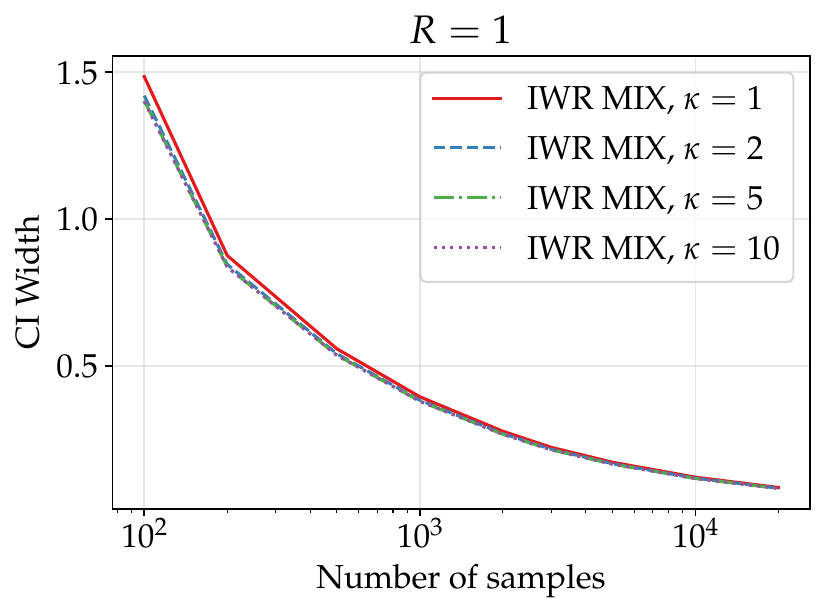}
    \end{subfigure}
    \hspace{0.2cm}
    \begin{subfigure}[b]{0.45\textwidth}
        \centering
        \includegraphics[width=\linewidth]{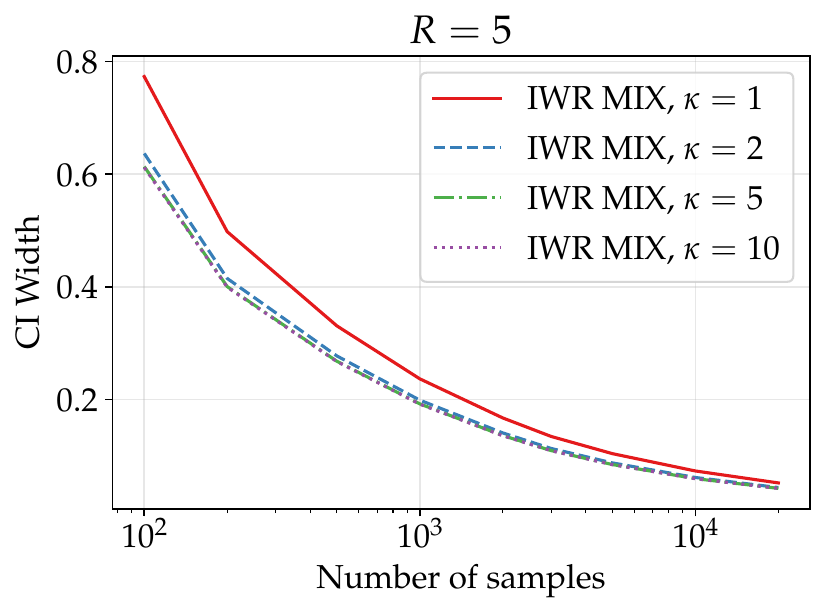}
    \end{subfigure}
    \begin{subfigure}[b]{0.45\textwidth}
        \centering
        \includegraphics[width=\linewidth]{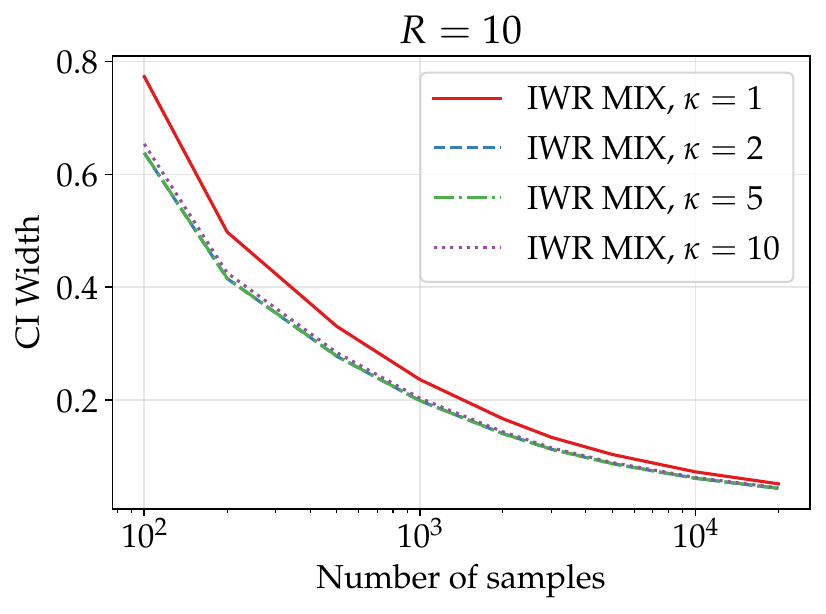}
    \end{subfigure}
    \hspace{0.2cm}
    \begin{subfigure}[b]{0.45\textwidth}
        \centering
        \includegraphics[width=\linewidth]{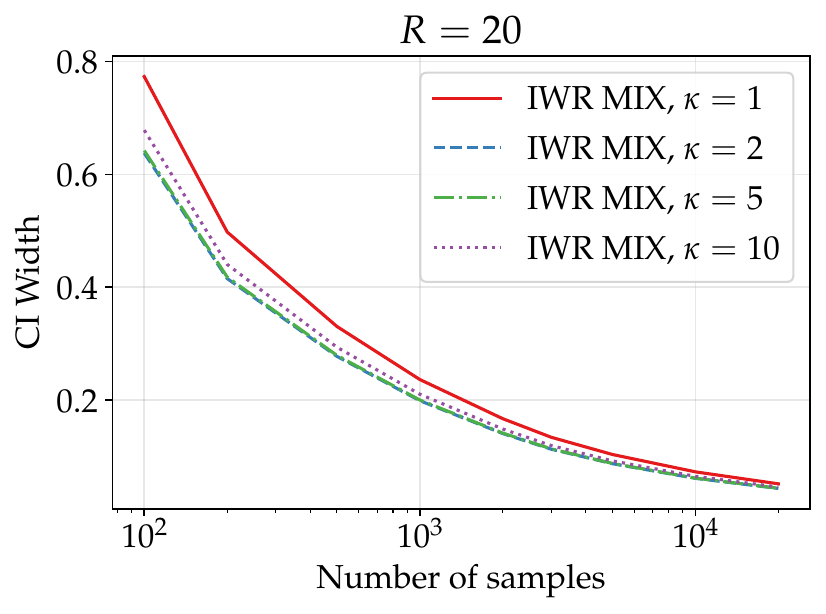}
    \end{subfigure}
    \begin{subfigure}[b]{0.45\textwidth}
        \centering
        \includegraphics[width=\linewidth]{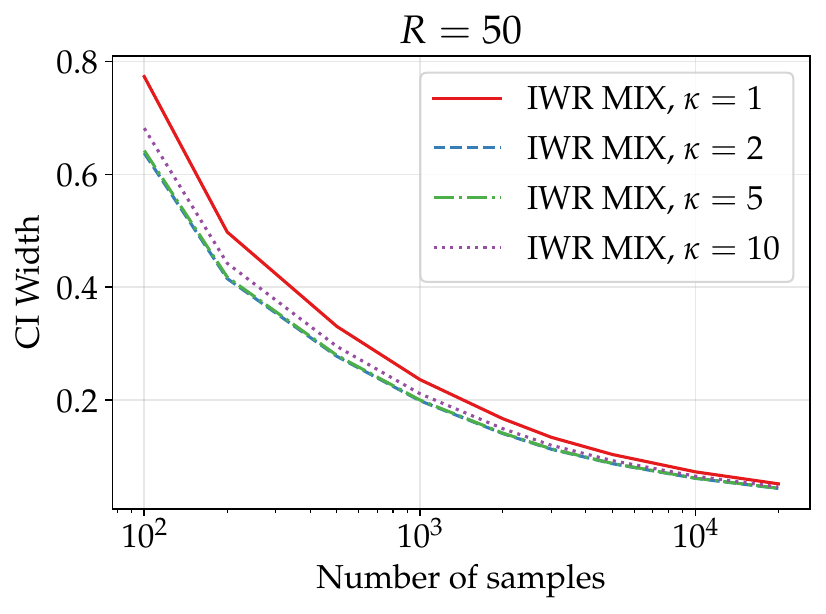}
    \end{subfigure}
    \hspace{0.2cm}
    \begin{subfigure}[b]{0.45\textwidth}
        \centering
        \includegraphics[width=\linewidth]{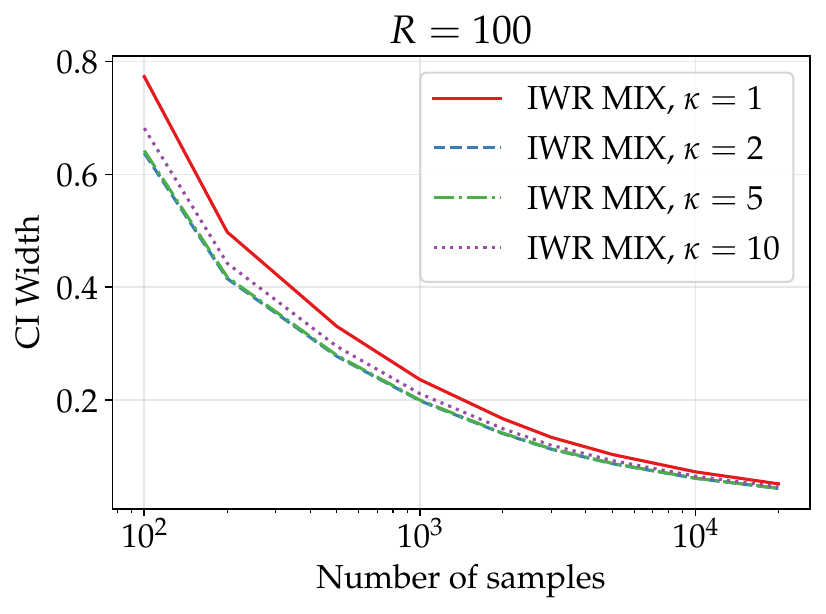}
    \end{subfigure}
    \caption{Effect of $R$ and $\kappa$ on the widths of \aphcis based on the truncated Gaussian mixture. Observations are drawn iid from a centered Gaussian with variance one. We use $\alpha=0.05$ across all figures.}
    \label{fig:R_kappa_ablations}
\end{figure}

Figure~\ref{fig:R_kappa_ablations} illustrates the effect of the parameters $R$ (range of the truncated Gaussian) and $\kappa$ (its variance) on the width of $(\calH_n^{\mix,\iwr})$. The results are plotted for a Gaussian distribution but the results are similar for a t-distribution. The results are also similar for $(\calH_n^{\mix,\reg})$, presented in Appendix~\ref{app:truncated-gaussian}.

\newpage

\subsection{APH-CIs: REG vs IWR}

\begin{figure}[h!]
    \centering
    \begin{subfigure}[b]{0.45\textwidth}
        \centering
        \includegraphics[width=\linewidth]{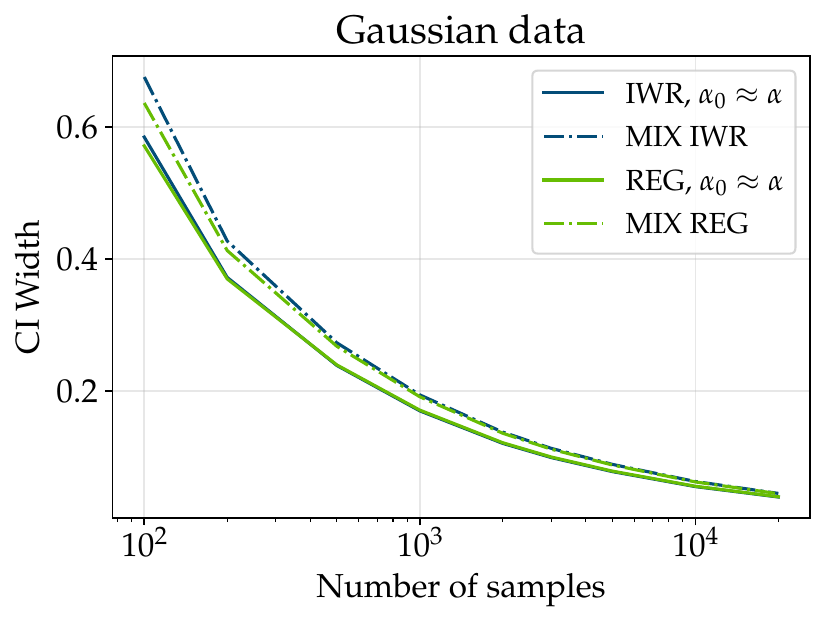}
    \end{subfigure}
    \hspace{0.2cm}
    \begin{subfigure}[b]{0.45\textwidth}
        \centering
        \includegraphics[width=\linewidth]{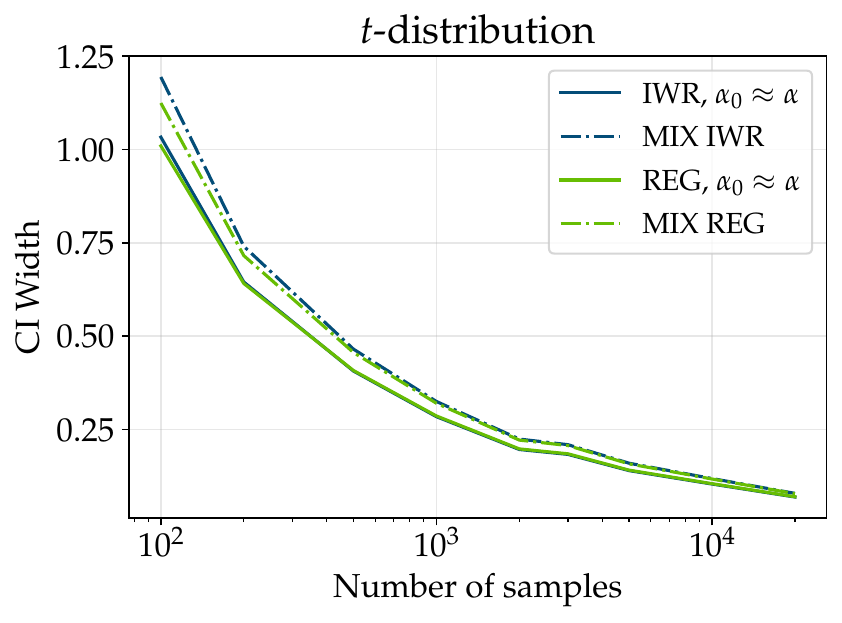}
    \end{subfigure}
    \begin{subfigure}[b]{0.45\textwidth}
        \centering
        \includegraphics[width=\linewidth]{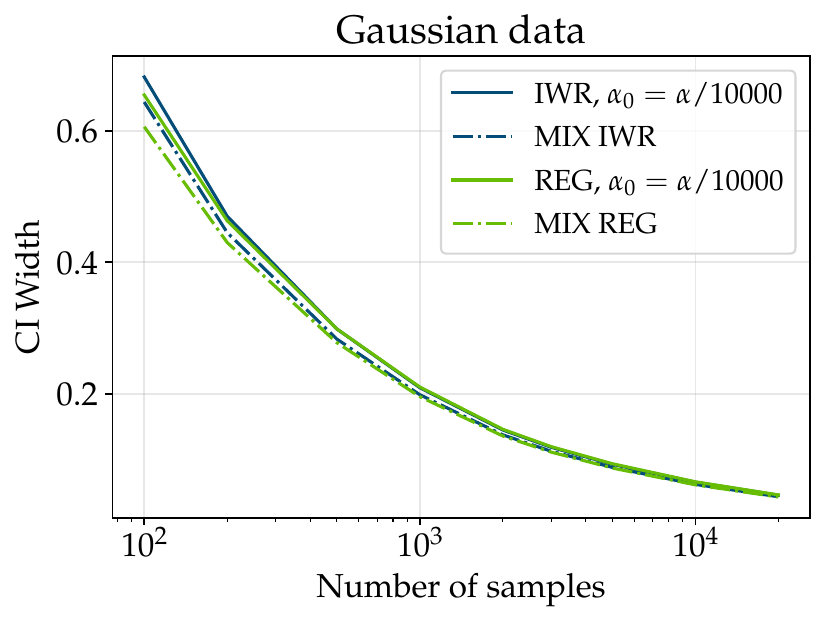}
    \end{subfigure}
    \hspace{0.2cm}
    \begin{subfigure}[b]{0.45\textwidth}
        \centering
        \includegraphics[width=\linewidth]{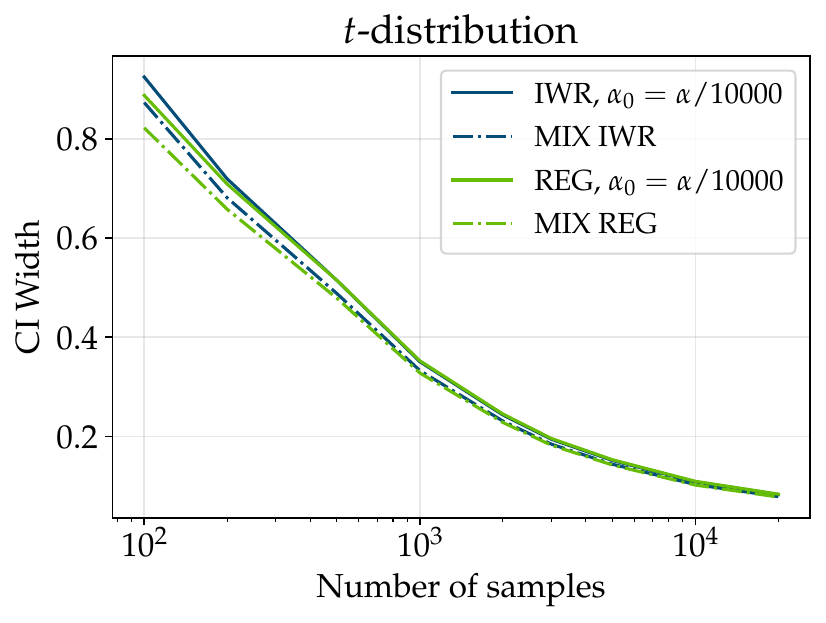}
    \end{subfigure}
    \caption{The two \aphcis based on the \iwr asymptotic e-variable vs the two \aphcis based on the \reg e-variable (Appendices~\ref{app:reg-aphci} and \ref{app:truncated-gaussian}). We generate data either from a Gaussian or from a t-distribution with three degrees of freedom. 
    Similarly to Figure~\ref{fig:widths}, \iwr and \reg are instantiated with ex ante anchoring, and we study two regimes: when $\alpha_0 \approx \alpha$ and when $\alpha/\alpha_0 \gg 0$. }
    \label{fig:iwr_vs_reg}
\end{figure}

Figure~\ref{fig:iwr_vs_reg} studies the two \aphcis resulting from the \iwr asymptotic e-variable, Theorems~\ref{thm:iwr-asymp-evar} and \ref{thm:iwr-mixture-aphci}, and the two resulting from the \reg asymptotic e-variable, Theorems~\ref{thm:reg-aphci} and \ref{thm:reg-mixture-aphci}. We instantiate $(\calH_n^\iwr)$ and $(\calH_n^{\reg})$ with ex ante anchoring. The remaining two are based on the method of mixtures. We choose $R=20$ and $\kappa=5$ as in the main paper. We take $\eta=0.01$ (see the following section for an investigation into the effect of $\eta$). Overall, we find that the two methods perform extremely similarly. The \aphcis based on the \reg e-variable are sometimes mildly tighter at small sample sizes, but this effect is quickly washed out. Given that the \iwr-based \aphcis have fewer tuning parameters, we recommend using them in practice.

\newpage

\subsection{Effect of \texorpdfstring{$\eta$}{eta} on \texorpdfstring{$\calH_n^\reg$}{REG APH-CI}}

\begin{figure}[h!]
    \centering
    \begin{subfigure}[b]{0.45\textwidth}
        \centering
        \includegraphics[width=\linewidth]{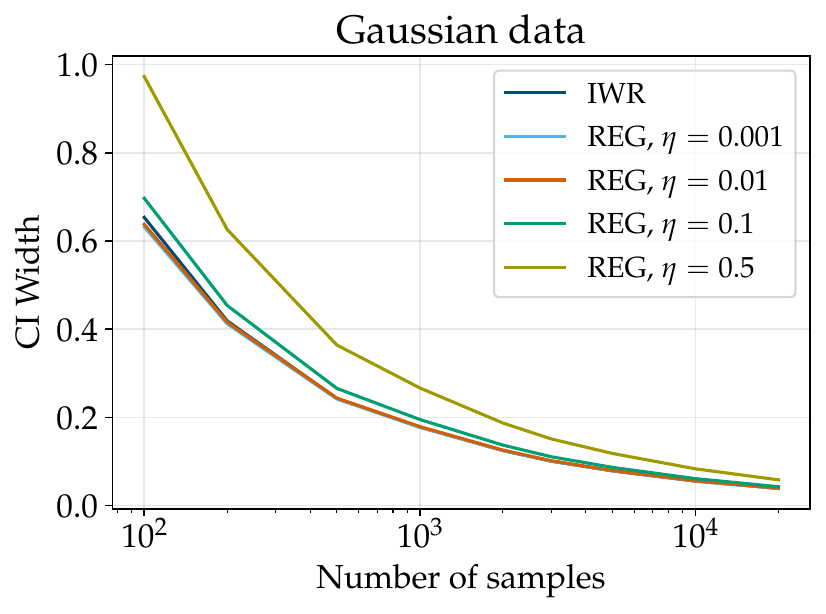}
    \end{subfigure}
    \hspace{0.2cm}
    \begin{subfigure}[b]{0.45\textwidth}
        \centering
        \includegraphics[width=\linewidth]{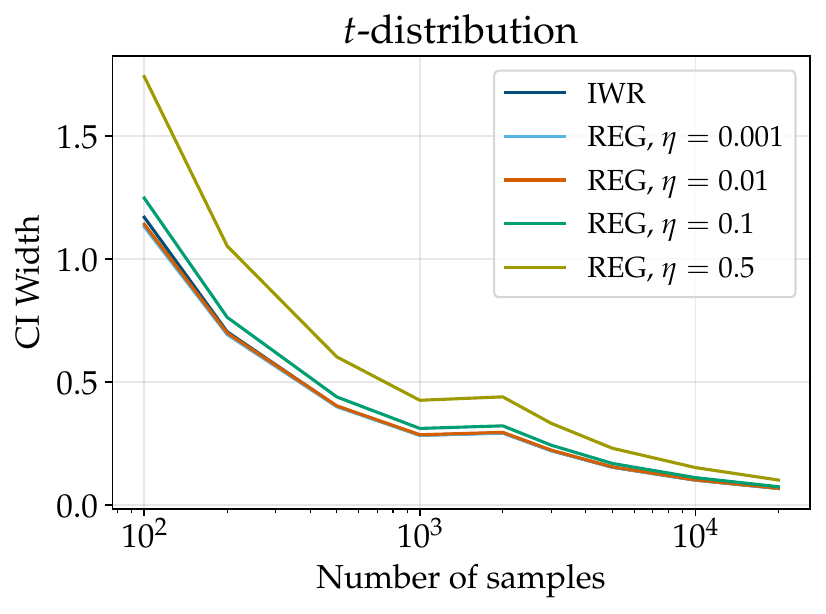}
    \end{subfigure}
    \begin{subfigure}[b]{0.45\textwidth}
        \centering
        \includegraphics[width=\linewidth]{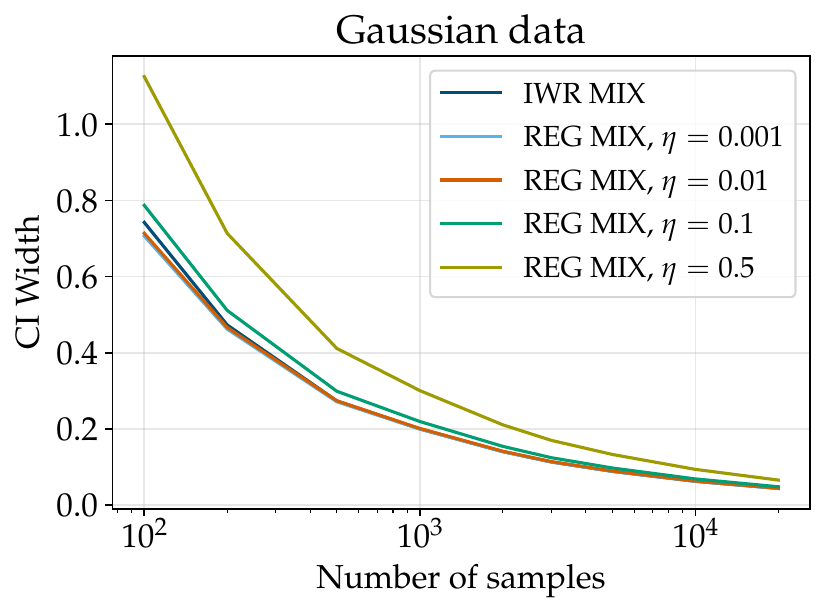}
    \end{subfigure}
    \hspace{0.2cm}
    \begin{subfigure}[b]{0.45\textwidth}
        \centering
        \includegraphics[width=\linewidth]{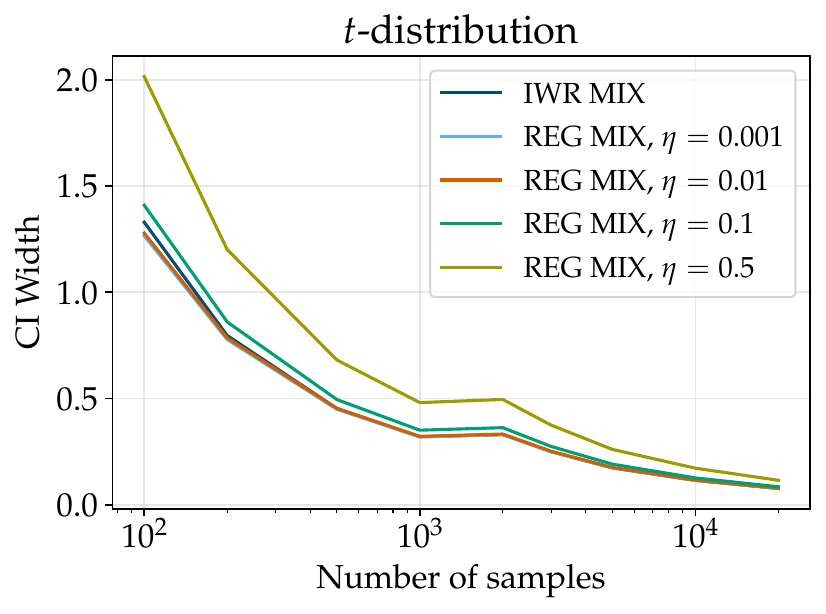}
    \end{subfigure}
    \caption{The effect of $\eta$ on the two \aphcis based on the \reg asymptotic e-variable. Here $\alpha = 0.05$. The \aphcis in the figures in the top row are all implemented with ex ante anchoring and $\alpha_0 = 0.01$. The mixture \aphcis in the bottom row are all implemented with $R=20$ and $\kappa=5$.  }
    \label{fig:eta_ablation}
\end{figure}

Figure~\ref{fig:eta_ablation} studies the effect of the free parameter $\eta$ on $(\calH_n^\reg)$ and $(\calH_n^{\mix,\reg})$. As $\eta\to 0$, the widths converge. In general, for small values of $\eta$, there differences between the \reg \aphcis and the \iwr \aphcis are negligible.

\end{document}